\documentclass{amsart}
\usepackage{amssymb, amsmath, latexsym, amscd, graphicx, tabularx, color}
\usepackage{rotating}
\usepackage[all]{xy}
\usepackage[dvipdfm]{hyperref}
\usepackage[all]{hypcap}

\newtheorem{theorem}{Theorem}[section]
\newtheorem{lemma}[theorem]{Lemma}
\newtheorem{cor}[theorem]{Corollary}
\newtheorem{definition}[theorem]{Definition}
\newtheorem{proposition}[theorem]{Proposition}
\newtheorem{remark}[theorem]{Remark}
\newtheorem{example}[theorem]{Example}

\def\pagenumber{1}
\begin{document}
\setcounter{page}{\pagenumber}
\newcommand{\T}{\mathbb{T}}
\newcommand{\R}{\mathbb{R}}
\newcommand{\Q}{\mathbb{Q}}
\newcommand{\N}{\mathbb{N}}
\newcommand{\Z}{\mathbb{Z}}
\newcommand{\tx}[1]{\quad\mbox{#1}\quad}
\parindent=0pt
\def\SRA{\hskip 2pt\hbox{$\joinrel\mathrel\circ\joinrel\to$}}
\def\tbox{\hskip 1pt\frame{\vbox{\vbox{\hbox{\boldmath$\scriptstyle\times$}}}}\hskip 2pt}
\def\circvert{\vbox{\hbox to 8.9pt{$\mid$\hskip -3.6pt $\circ$}}}
\def\IM{\hbox{\rm im}\hskip 2pt}
\def\COIM{\hbox{\rm coim}\hskip 2pt}
\def\COKER{\hbox{\rm coker}\hskip 2pt}
\def\TR{\hbox{\rm tr}\hskip 2pt}
\def\GRAD{\hbox{\rm grad}\hskip 2pt}
\def\RANK{\hbox{\rm rank}\hskip 2pt}
\def\MOD{\hbox{\rm mod}\hskip 2pt}
\def\DEN{\hbox{\rm den}\hskip 2pt}
\def\DEG{\hbox{\rm deg}\hskip 2pt}

\title[Exotic PDE's]{\mbox{}\\[1cm] EXOTIC PDE's}
\author{Agostino Pr\'astaro}
\maketitle
\vspace{-.5cm}
{\footnotesize
\begin{center}
Department SBAI - Mathematics, University of Rome ''La Sapienza'', Via A.Scarpa 16,
00161 Rome, Italy. \\
E-mail: {\tt Prastaro@dmmm.uniroma1.it; Agostino.Prastaro@uniroma1.it}
\end{center}
}

\begin{abstract} \noindent
In the framework of the PDE's algebraic topology, previously introduced by A. Pr\'astaro, are considered {\em exotic differential equations}, i.e., differential equations admitting Cauchy manifolds $N$ identifiable with exotic spheres, or such that their boundaries $\partial N$ are exotic spheres. For such equations are obtained local and global existence theorems and stability theorems.
In particular the smooth ($4$-dimensional) Poincar\'e conjecture is proved. This allows to complete the previous Theorem 4.59 in \cite{PRA17} also for the case $n=4$.
\end{abstract}


\vspace{.08in} \noindent \textbf{Keywords}: Integral (co)bordism groups in PDE's;
Existence of local and global solutions in PDE's; Conservation laws;
Crystallographic groups; Singular PDE's; Exotic spheres.

\section{\bf INTRODUCTION}

In a well known book by B. A. Dubrovin, A. T. Fomenko,  and S. P. Novikov published in 1990 by the Springer (original Russian edition published in 1979), it is written ''{\em Up to the 1950s it was generally regarded as ''clear'' that any continuous manifold admits a compatible smooth structure, and that any two continuously homeomorphic manifolds would automatically be diffeomorphic. In fact, these assertions are clearly true in the one-dimensional case, can be proved without great difficulty in two dimensions, and have been established also in the $3$-dimensional case (by Moise), though with considerable difficulty, notwithstanding the elementary nature of the techniques involved.}''. (See in \cite{DUB-FOM-NOV}, Part III, page 358.) In these statements there is summarized the great surprise that produced in the international mathematical community, the paper by J. Milnor on ''exotic $7$-dimensional sphere''. This unforeseen mathematical phenomenon, really do not soon produced great consequences in the mathematical physics community, since the more physically interesting $3$-dimensional case, remained for long a period an open problem being related to the Poincar\'e conjecture too. Nowadays, after the proof of the Poincar\'e conjecture, as made by A. Pr\'astaro, that allows us to extend the h-cobordism theorem also to the $3$-dimensional case, in the category of smooth manifolds, and his results about exotic spheres and existence of global (smooth) solutions in PDE's, it appears ''very clear'' that exotic spheres are not only a strange mathematical curiosity, but are very important mathematical structures to consider in any geometric theory of PDE's and its applications.

However, in this beautiful and important mathematical architecture, was remained open the so-called {\em smooth Poincar\'e conjecture}. This conjecture states that in dimension $4$, any homotopy sphere $\Sigma^4$ is diffeomorphic to $S^4$. The proof of such a conjecture was considered fat chance since in dimension four there is well known the phenomenon of exotic $\mathbb{R}^4$'s, that, instead does not occur in other dimensions $n\not=4$.
This conjecture is of course of great importance in geometric topology, and has great relevance in geometric theory of PDE's and its applications.

Aim of this paper is just to emphasize such implications in the algebraic topology of PDE's, according to the previous formulation by A. Pr\'astaro and to generalize results about ''exotic heat PDE's'' contained in Refs. \cite{PRA14, PRA16, PRA17, PRA18}. In order to allow a more easy understanding and a presentation as self-contained as possible, also in this paper, likewise in its companion \cite{PRA17}, a large expository style has been adopted. More precisely, after this introduction, the paper splits into three more sections. 2. Spectra in algebraic topology. 3. Spectra in PDE's. 4. Spectra in exotic PDE's. The main result is Theorem \ref{integral-h-cobordism-in-Ricci-flow-pdes} that extends Theorem 4.5 in \cite{PRA17} also to the case $n=4$. There it is also proved the smooth ($4$-dimensional) Poincar\'e conjecture (Lemma \ref{the-smooth-poincare-conjecture}), and the smooth $4$-dimensional h-cobordism theorem (Corollary \ref{smooth-4-dimensional-h-cobordism-theorem}).

\section{SPECTRA IN ALGEBRAIC TOPOLOGY}\label{spectra-algebraic-topology-section}

In this section we report on some fundamental definitions and results in algebraic topology, linked between them by the unifying concept of {\em spectrum}, i.e., a suitable collection of CW-complexes. In fact, this mathematical structure allows us to look to (co)homotopy theories, (co)homology theories and (co)bordism theories, as all placed in an unique algebraic topologic framework. This mathematics will be used in the next two sections by specializing and adapting it to the PDE's geometric structure obtaining some fundamental results in algebraic topology of PDE's, as given by A. Pr\'astaro. (See the next section and references quoted there.)

\begin{definition}\label{spectrum-algebraic-topology} A {\em spectrum} $E$ is a collection $\{(E_n,*):n\in\mathbb{Z}\}$ of CW-complexes such that $SE_n$ is (or is homeomorphic to)
a subcomplex of $E_{n+1}$, all $n\in\mathbb{Z}$.\footnote{In this paper we denote by $SX$ the suspension of a topological space $X$. (For details and general informations on algebraic topology see e.g. \cite{PRA3, SWITZER, WALL1, WALL2, WARNER}. In this paper we will use the following notation: $\thickapprox$ homeomorphism; $\cong$ diffeomorphism; $\approxeq$ homotopy equivalence; $\simeq$ homotopy.) Therefore $SE_n$ is the suspension of the CW-complex $E_n$.} A {\em subspectrum}
$F\subset E$ consists of subcomplexes $F_n\subset E_n$ such that
$SF_n\subset F_{n+1}$. A {\em cell} of dimension $d'-n'$ in
$E$ is a sequence
$$e=\{e^{d'}_{n'}, Se^{d'}_{n'}, S^2e^{d'}_{n'},\cdots\},$$
where $ e^{d'}_{n'}$ is a cell in $E_{n'}$ that is not the
suspension of any cell in $E_{n'-1}$. Thus each cell in each
complex $E_n$ is a manifold of exactly one cell of $E$. We call
{\em cell of dimension $-\infty$} the subspectrum $*\equiv
F\subset E$ such that $F_n=*$ for all $n$. A spectrum $E$ is
called {\em finite} if it has only finitely many cells. It is
called {\em countable} if it has countably many
cells.\footnote{$ E$ is finite iff there is an
integer $ N$ such that $ E_n=S^{n-N}E_N$
for $ n\ge N$ and the complex $ E_N$ is
finite.} An {\em $\Omega$-spectrum} is a spectrum $E$ such that
the adjoint $\epsilon':E_n\to\Omega E_{n+1}$ of the inclusion
$\epsilon_n:SE_n\to E_{n+1}$ is always a weak homotopy
equivalence.
\end{definition}

\begin{example}[Suspension spectrum]\label{spectrum-cw-complex} If $X$ is any CW-complex, then we can define a spectrum $E(X)$ by taking
$$E(X)_n\equiv\left\{\begin{array}{l}
  *,\hskip 5pt n<0\\
  S^nX,\hskip 5pt n\ge 0.\\
  \end{array}\right.$$
Then the mapping $\epsilon_n:SE(X)_n=S^{n+1}X\to E(X)_{n+1}=S^{n+1}X$ is just the identity.

In particular if $X=S^0=\{0,1\}\subset\mathbb{R}$, then $E(S^0)$ is called the {\em sphere spectrum} and one has $E(S^0)_n\approxeq S^n$, but also $E(S^0)_n\thickapprox S^n$. So we get the commutative diagram given in (\ref{sphere-spectrum-commutative-diagram}).
\begin{equation}\label{sphere-spectrum-commutative-diagram}
    \scalebox{0.8}{$\xymatrix{\cdots\{0\}\ar@{~>}[d]\ar@{^{(}->}[r]&\{0\}\ar@{~>}[d]\ar@{^{(}->}[r]&
    \{0,1\}\ar@{~>}[d]\ar@{^{(}->}[r]&
    E(S^0)_1\ar@{~>}[d]\ar@{^{(}->}[r]&E(S^0)_2\ar@{~>}[d]\ar@{^{(}->}[r]&\cdots
    \ar@{^{(}->}[r]&E(S^0)_n\ar@{~>}[d]\ar@{^{(}->}[r]&\cdots\\
    \cdots \{*\}\ar@{^{(}->}[r]&\{*\}\ar@{^{(}->}[r]&S^0\ar@{^{(}->}[r]&S^1\ar@{^{(}->}[r]&S^2\ar@{^{(}->}[r]&
    \cdots\ar@{^{(}->}[r]&S^n\ar@{^{(}->}[r]&\cdots}$}
\end{equation}

where the vertical arrows $\rightsquigarrow$ denote homotopic equivalence and homeomorphisms too.
\end{example}

\begin{example}[Eilenberg-MacLane spectra]\label{eilenberg-maclane-spectra}
Given any
collection $\{E_n,\epsilon_n\}$ of CW-complexes $(E_n,*)$ and
cellular maps $\epsilon_:SE_n\to E_{n+1}$ we can construct a
spectrum $E'\equiv\{E'_n\}$ and homotopy equivalences $r_n:E'_n\to
E_n$ such that $r_{n+1}|_{SE'_n}=\epsilon_n\circ Sr_n$, i.e., the
following diagram is commutative
$$\xymatrix@C=70pt{SE'_n\ar[d]_{Sr_n}\ar@{^{(}->}[r]&E'_{n+1}\ar[d]_{ r_{n+1}}\\
 SE_n\ar@{^{(}->}[r]&E_{n+1}}$$

In particular the Eilenberg-MacLane spaces $K(G,n)$, uniquely defined (up to weak homotopy equivalence) by the condition (\ref{eilenberg-maclane-spaces}),\footnote{If $n>1$ then $G$ must be abelian.}
\begin{equation}\label{eilenberg-maclane-spaces}
   \pi_k(K(G,n))=\left\{
   \begin{array}{l}
     0,\hskip 3pt k\not=n\\
     G,\hskip 3pt k=n.\\
   \end{array}
   \right.
\end{equation}
identify an $\Omega$-spectrum since $\Omega K(G,n+1)\approxeq K(G,n)$.
\end{example}

\begin{example}[Thom spectra]\label{thom-spectra}
Let $\pi:EG_n\to BG_n$ be the universal bundle for $G_n$-vector bundles. Then the {\em Thom spectrum} $MG$ associated to $\pi:EG_n\to BG_n$ is defined by $(MG)_n=MG_n$ and $(MG)_{n+k}=S^kMG_n$, with $k\ge 1$. So the map $\epsilon_n:S(MG)_n\to(MG)_{n+1}$ is the natural homeomorphism $SMG_n\thickapprox SMG_n$.
\end{example}

\begin{definition} A {\em filtration} of a
spectrum $E$ is an increasing sequence $\{E^n:n\in\mathbb{Z}\}$ of
subspectra of $E$ whose union is $E$.\end{definition}

\begin{example} The {\em
skeletal filtration} $\{E^{(n)}\}$ of $E$ is defined as
follows: $E^{(n)}$ is the union of all the cells of $E$ of
dimension at most $n$.\end{example}

\begin{example} The {\em layer filtration}
$\{E^{\infty}\}$ of $E$ is defined as follows: for each cell
$e=\{e_n,Se_n,\cdots\}$ of $E$ we can find a finite subspectrum
$F\subset E$ of which $e$ is a cell. (For example, let $F_n\subset
E_n$ be the subcomplex consisting of $e_n$ and all its faces; then
take $F_m=*$, $m<n$, $F_m=S^{m-n}F_n$, $m\ge n$.) Let $l(e)$ be
the smallest number of cells in any such $F$. ($l(e)$ coincides
with the number of faces of $e_n$.) Then we define $E^n=*$, $n\le
0$, $E^n=$ union of all cells with $l(e)\le n$, $n>0$. The terms
$E^n$ are called the {\em layers} of $E$. One can see that
$\{E^n\}$ is a filtration of $E$.\end{example}

\begin{definition} {\em 1)} A {\em function}
$f:E\to F$ between spectra is a collection $\{f_n:n\in\mathbb{Z}\}$
of cellular maps $f_n:E_n\to F_n$ such that
$f_{n+1}|_{SE_n}=Sf_n$. The inclusion $i:F\hookrightarrow E$ of a
subspectrum $F\subset E$ is a function and if $g:E\to G$ is a
function then $g|_F=g\circ i$ is also a function.

{\em 2)} A subspectrum $F\subset E$ is called {\em cofinal} if for any
cell $e_n\subset E_n$ of $E$ there is an $m$ such that
$S^me_n\subset F_{n+m}$.\footnote{If $F$ is cofinal and
$K_n\subset E_n$ is a finite subcomplex, then there is an $m$ such
that $S^mK_n\subset F_{n+m}$.
Intersection of two
cofinal subspectra is cofinal and if $G\subset F\subset E$ are
subspectra such that $F$ is cofinal in $E$ and $G$ is cofinal in
$F$, then $G$ is cofinal in $E$. An arbitrary union of cofinal
subspectra is cofinal.}

{\em 3)} Let $E$ and $F$ be spectra.
Let $S$ be the set of all pairs $(E',f')$ such that $E'\subset E$
is a cofinal subspectrum and $f':E'\to F$ is a function. We call
{\em maps} from $E$ to $F$ elements of $Hom(E,F)/\sim$, where
$\sim$ is the following equivalence relation:
$$(E',f')\sim(E'',f'')\hskip 2pt\Leftrightarrow\hskip 2pt \exists (E''',f'''):
\left\{\begin{array}{l}
E'''\subset E'\cap E'',\\
 f'|_{E'''}=f'''=f''|_{E'''},\\
 E'''\hskip 2pt\hbox{\rm cofinal}.\\
       \end{array}\right\}.$$ (Intuitively maps only need to
be defined on each cell.) The {\em category of spectra} $\mathcal{S}$
is the category where objects are spectra and morphisms are maps.\footnote{Let $E$ and $F$ be spectra and $f:E\to F$ a function. If $F'\subset F$ is a cofinal subspectrum then there is a cofinal
subspectrum $E'\subset E$ with $f(E')\subset F'$. (Composition of
maps is now possible!) In the category $\mathcal{S}$ any spectrum is
equivalent to any cofinal subspectrum of its.}
\end{definition}

\begin{proposition} If
$E\equiv\{E_n\}$ is a spectrum and $(X,x_0)$ is a CW-complex, then
we can form a new spectrum $E\wedge X$: we take $(E\wedge
X)_n=E_n\wedge X$ with the weak topology.\end{proposition}

\begin{proof} In fact
$S(E\wedge X)_n=S(E_n\wedge X)=S^1\wedge(E_n\wedge
X)\cong(S^1\wedge E_n)\wedge X\subset E_{n+1}\wedge X$.
Furthermore, given a map $f:E\to F$ of spectra represented by
$(E',f')$ and a map $g:K\to L$ of CW-complexes, we get a map
$f\wedge g:E\wedge K\to F\wedge L$ of spectra represented by
$(E'\wedge K,f'\wedge g)$, since $E'\wedge K$ is cofinal in
$E\wedge K$.\end{proof}

\begin{definition} A {\em homotopy} between spectra is a map $h:E\wedge I^+\to
F$.\footnote{For any topological space $X$ we set$X^+\equiv X\sqcup\{\ast\}$. (If $X$ is not compact $X^+$ is the Alexandrov compactification to a point.) In particular if $X=I\equiv[0,1]\subset\mathbb{}$, then one takes $\{\ast\}$ as base point.} There are two maps $i_0:E\to E\wedge I^+$, $i_1:E\to E\wedge
I^+$, induced by the inclusions of $0$, $1$ in $I^+$. Then, we say
two maps of spectra $f_0,f_1:E\to F$ are {\em homotopic} if there
is a homotopy $h:E\wedge I^+\to F$ with $h\circ i_0=f_0$, $h\circ
i_1=f_1$. We shall write $h_0\equiv h\circ i_0$, $h_1\equiv h\circ
i_1$. In terms of cofinal subspectra we can say that two maps
$f_0,f_1:E\to F$ represented by $(E'_0,f'_0)$, $(E'_1,f'_1)$,
respectively, are homotopic if there is a cofinal subspectrum
$E''\subset E'_0\cap E'_1$ and a function $h'':E''\wedge I^+\to F$
such that $h''_0=f'_0|_{E''}$, $h''_1=f'_1|_{E''}$. Homotopy is an
equivalence relation, so we may define $[E,F]$ to be the set of
equivalence classes of maps $f:E\to F$. Composition passes to
homotopy classes. The corresponding category is denoted by $\mathcal{S}'$.\end{definition}

\begin{proposition} Let $\mathcal{W}_\bullet$ be the category of pointed CW-complexes.
One has a functor $E:\mathcal{W}_\bullet\to \mathcal{S}$, such that $E(X,x_0)=E(X)$, and $E(f)=\{E(f)_n\}$, with
$E(f)_n=S^nf:S^nX\to S^nY$, $n\ge 0$, for any
$f:(X,x_0)\to(Y,y_0)$. The functor $E$ embeds $\mathcal{W}_\bullet$
into $\mathcal{S}$.\end{proposition}

\begin{proposition} To any map $f:E\to F$ between spectra,
we can associate another spectrum, the {\em(mapping cone)},
$F\cup_fCE$.\end{proposition}

\begin{proof} In fact let $I$ be pointed on $0$, and set
$CE\equiv E\wedge I$. Then the mapping cone of $f$ is the spectrum
$(F\cup_fCE)_n=E_n\cup_{f'_n}(E'_n\wedge I)$, where $(E',f')$
represents $f$. If $(E'',f'')$ is another representative of $f$,
then $\{F_n\cup_{f'_n}(E'_n\wedge I)\}$ and
$\{F_n\cup_{f''_n}(E''_n\wedge I)\}$ have a natural cofinal
subspectrum $\{F_n\cup_{f'''_n}(E'''_n\wedge I)\}$ and hence are
equivalent.\end{proof}

\begin{proposition}\label{natural-invertible-functor-capital-sigma} One has a natural invertible functor $\Sigma:\mathcal{S}\to\mathcal{S}$ that induces a functor on $\mathcal{S}'$.\end{proposition}

\begin{proof} For any
spectrum $E\equiv\{E_n\}$ we can define $\Sigma E$ to be the
spectrum with $\Sigma E_n=E_{n+1}$, $n\in\mathbb{Z}$. For any
function $f:E\to F$ we define $\Sigma(f):\Sigma E\to \Sigma F$ to
be the map represented by $(\Sigma E',\Sigma(f'))$. Furthermore,
$\Sigma$ induces a functor on $\mathcal{S}'$ since $f_0\simeq f_1$
implies $\Sigma(f_0)\simeq \Sigma(f_1)$. We can iterate $\Sigma$:
$\Sigma^{n+1}=\Sigma\circ\Sigma^n$, $n\ge 1$. $\Sigma$ has also an
inverse defined by $(\Sigma^{-1}E)_n=E_{n-1}$,
$\Sigma^{-1}(f)_n=f_{n-1}$. One has
$\Sigma^n\circ\Sigma^m=\Sigma^{n+m}$, for all integers $n,m$.\end{proof}

\begin{remark} $\Sigma E$ and $E\wedge S^1$ have the same homotopy type, therefore the suspension is invertible in $\mathcal{S}'$.
The higher homotopy groups for topological spaces are very difficult to compute. For spectra that computations are easier.\end{remark}

\begin{proposition} We have wedge sums in $\mathcal{S}$: given a collection $\{E^\alpha:\alpha\in A\}$ of spectra, we define $\bigvee_\alpha E^\alpha$ by
$(\bigvee_\alpha E^\alpha)_n=\bigvee_\alpha E^\alpha_n$.
\end{proposition}

\begin{proof} Since $S(\bigvee_\alpha E^\alpha_n)=\bigvee_\alpha S E^\alpha_n\subset \bigvee_\alpha E^\alpha_{n+1}$, this is a spectrum.\end{proof}

\begin{proposition} For any collection $\{E^\alpha:\alpha\in A\}$ of spectra, the inclusions $i_\beta:E^\beta\to\bigvee_\alpha E^\alpha$ induce bijections:
$$\left\{\begin{array}{l}
   \{Hom(i_\alpha,1)\}:Hom_\mathcal{S}(\bigvee_\alpha E^\alpha,F)\to\prod_\alpha Hom_\mathcal{S}(E^\alpha,F)\\ \{i_\alpha^*\}:[\bigvee_\alpha E^\alpha,F]\to\prod_\alpha [E^\alpha,F]\\
         \end{array}\right\},\hskip 5pt \forall F\in Hom(\mathcal{S}).$$
Inclusions of spectra possess the homotopy extension property just as with CW-complexes.
\end{proposition}
\begin{definition}[Homotopy groups of spectra] We set $\pi_n(E)=[\Sigma^nS^0,E]$, $n\in\mathbb{Z}$.
\end{definition}

\begin{proposition} {\em 1)} One has the isomorphisms of abelian groups:
$$\left\{\begin{array}{l}
\pi_n(E)\cong \mathop{\lim}\limits_{\overrightarrow{k}}\pi_{n+k}(E_k,*)\\
\pi_n(E(X))\cong \mathop{\lim}\limits_{\overrightarrow{k}}\pi_{n+k}(S^kX,*)=\pi^s_n(X)\\
\end{array}\right\},\hskip 5pt n\in\mathbb{Z}.$$
Note that $\pi^s_n(X)$ may be quite different from $\pi_n(X,x_0)$.

{\em 2)} If $f:E\to F$ is a map of spectra which is a weak homotopy equivalence, then $f_*:[G,E]\to[G,F]$ is a bijection for any spectrum $G$.

{\em 3)} A map of spectra is a weak homotopy equivalence iff it is a homotopy equivalence.
\end{proposition}
\begin{definition} For any map $f:E\to F$ of spectra we call the sequence
$$\xymatrix@C=30pt{E\ar[r]^f&F\ar[r]^j&F\cup_f CE}$$
a {\em spectral cofibre sequence}. A {\em general cofibre
sequence}, or simply a {\em cofibre sequence}, is any sequence
$$\xymatrix@C=30pt{G\ar[r]^g&H\ar[r]^h&K}$$
for which there is a homotopy commutative diagram
$$\xymatrix{G\ar[d]_{\alpha}\ar[r]^g&H\ar[d]_{\beta}\ar[r]^h&K\ar[d]_{\gamma}\\
E\ar[r]_f&F\ar[r]_j&F\cup_f CE}$$
where $\alpha$, $\beta$, $\gamma$ are homotopy equivalences.
\end{definition}
\begin{proposition} {\em 1)} In the sequence
$$\xymatrix@C=30pt{E\ar[r]^f&F\ar[r]^j&F\cup_f CE\ar[r]^{\kappa'}&E\wedge S^1\ar[r]^{f\wedge 1}&F\wedge S^1}$$
each pair of consecutive maps forms a cofibre sequence.

{\em 2)} Given a homotopy commutative diagram of spectra and maps, as the following
$$\xymatrix{G\ar[d]_{\alpha}\ar[r]^g&H\ar[d]_{\beta}\ar[r]^h&K\ar[d]_{\gamma}\ar[r]^\kappa&G\wedge S^1\ar[d]_{\alpha\wedge 1}\\
G'\ar[r]^{g'}&H'\ar[r]^{h'}&K'\ar[r]_{\kappa'}&G'\wedge S^1}$$
where the rows are cofibre sequences, we can find a map
$\gamma:K\to K'$ such that the resulting diagram is homotopy
commutative.

{\em 3)} If
$\xymatrix{G\ar[r]^g&H\ar[r]^h&K}$ is a cofibre
sequence, then for any spectrum $E$ the sequences
$$\begin{array}{l}
   \xymatrix{[E,G]\ar[r]^{g_*}&[E,H]\ar[r]^{h_*}&[E,K]}\\
   \xymatrix{[G,E]&[H,E]\ar[l]^{g^*}&[K,E]\ar[l]^{h^*}}\\
  \end{array}$$
are exact.\end{proposition}

\begin{theorem}[(Co)homology theories associated with any spectrum]\label{co-homology-theories-associated-with-any-spectrum}
Let $\mathcal{W}'_\bullet$ be the category of pointed CW-complexes where $Hom_{\mathcal{W}'_\bullet}((X,x_0);(Y,y_0))=[(X,x_0),(Y,y_0)]$ is the set of all homotopy classes of pointed maps $(X,x_0)\to(Y,y_0)$.
For each $(X,x_0)\in Ob(\mathcal{W}'_\bullet)$
and $n\in\mathbb{Z}$, we have
$$\begin{array}{l}
   E_n(X)=\pi_n(E\wedge X)=[\Sigma^nS^0,E\wedge X]\\
   E^n(X)=[E(X),\Sigma^n E]\cong[\Sigma^{-n}S^0\wedge X,E].\\
  \end{array}$$
These define (co)homology theories on $\mathcal{W}'_\bullet$ that satisfy the wedge axiom.\footnote{For definitions of generalized (co)homology theories see, e.g., \cite{PRA3, SWITZER}.}
The coefficient groups of the homology theory $E_\bullet$ are
$$E_n(S^0)=\pi_n(E\wedge S^0)=\pi_n(E),\hskip 5pt n\in\mathbb{Z}.$$
The coefficient groups of the cohomology theory $E^\bullet$ are
$$E^n(S^0)=[E(S^0),\Sigma^nE]=[S^0,\Sigma^nE]\cong[\Sigma^{-n}S^0,E]=\pi_{-n}(E),\hskip 5pt n\in\mathbb{Z}.$$
Furthermore, any map $f:E\to F$ of spectra induces natural transformations $T_\bullet(f):E_\bullet\to F_\bullet$,
$T^\bullet(f):E^\bullet\to F^\bullet$ of homology and cohomology theories respectively. If $f$ is a homotopy equivalence, then $T_\bullet(f)$ and
$T^\bullet(f)$ are natural equivalences. This is the case iff $f_\bullet:\pi_n(E)\to\pi_n(F)$ is an isomorphism for all $n\in\mathbb{Z}$,
i.e., $T_\bullet(f)$ is a natural equivalence iff it is an isomorphism on the coefficient groups.
\end{theorem}

\begin{proof} For $f:(X,x_0)\to(Y,y_0)$ we take $E_n(f)=(1\wedge f)$ and $E^n(f)=E(f)$.\footnote{Here we have used the fact that $ E(SX)$ is a cofinal subspectrum of $\Sigma E(X)$ and hence the inclusion $ i:E(SX)\to\Sigma E(X)$ induces an isomorphism $ i^*$.}
We define $\sigma_n:E_n(X)\to E_{n+1}(SX)$ to be the composite
$$\scalebox{0.8}{$\xymatrix{E_n(X)=[\Sigma^nS^0,E\wedge X]\ar@{=}[d]\ar[r]^(0.55){\Sigma}_(0.55){\cong}&[\Sigma^{n+1}S^0,\Sigma E\wedge X]\ar[r]_(0.4){\cong}&[\Sigma^{n+1}S^0,E\wedge S^1\wedge X]= E_{n+1}(SX)\ar@{=}[d]\\
E_n(X)\ar[rr]_{\sigma_n}&& E_{n+1}(SX)}$}$$ Then
$\sigma_n$ is a natural equivalence. Furthermore, we define
$\sigma^n:E^{n+1}(SX)\to E^{n}(X)$ to be the composite
$$\scalebox{0.8}{$\xymatrix{E^{n+1}(SX)=[E(SX),\Sigma^{n+1}E\wedge X]\ar@{=}[d]&[\Sigma E(X),\Sigma^{n+1}E]\ar[l]^(0.35){i^*}_(0.35){\cong}\ar[r]^(0.45){\Sigma^{-1}}_(0.45){\cong}&[E(X),\Sigma^{n}E]=E^{n}(X)\ar@{=}[d]\\ E^{n+1}(SX)\ar[rr]_{\sigma^n}&&E_{n}(X)}$}$$ Then $\sigma^n$ is a natural equivalence too. Let
$(X,A)$ be any pointed CW-pair. Since
$$E_n\wedge(X\cup CA)\cong(E_n\wedge X)\cup C(E_n\wedge A), \hskip 5pt n\in\mathbb{Z},$$
we see that
$$\xymatrix@C=60pt{E\wedge A\ar[r]^{1\wedge i}&E\wedge X\ar[r]^{1\wedge j}&E\wedge( X\cup CA)}$$
is a cofibre sequence. Therefore,
$$\xymatrix@C=60pt{[\Sigma^nS^0,E\wedge A]\ar[r]^{(1\wedge i)_*}&[\Sigma^nS^0,E\wedge X\ar[r]^{(1\wedge j)_*}
&[\Sigma^nS^0,E\wedge( X\cup CA)]}$$
is exact; but this is just the sequence
$$\xymatrix@C=60pt{E_n(A)\ar[r]^{i_*}&E_n(X)\ar[r]^{j_*}&E_n(X\cup CA)}.$$
Thus $E_\bullet$ is a homology theory on $\mathcal{W}'_\bullet$. Since $S^n(X\cup CA)\cong S^nX\cup C(S^nA)$, $n\in\mathbb{Z}$, we see that
$$\xymatrix@C=60pt{E(A)\ar[r]^{E(i)}&E(X)\ar[r]^{E(j)}&E(X\cup CA)}$$
is a cofibre sequence. Hence
$$\xymatrix{[E(A),\Sigma^n(E)]&[E(X),\Sigma^n(E)]\ar[l]^{E(i)^*}&[E(X\cup CA),\Sigma^n(E)]\ar[l]^{E(j)^*}}$$
is exact; but this is just the sequence
$$\xymatrix@C=60pt{E^n(A)&E^n(X)\ar[l]^{i^*}&E^n(X\cup CA)\ar[l]^{j^*}}.$$
Thus $E^\bullet$ is a cohomology theory on $\mathcal{W}'_\bullet$. Since for any collection $\{X_\alpha:\alpha\in A\}$ of CW-complexes we have
$S^n(\bigvee_\alpha X_\alpha)\cong\bigvee_\alpha S^nX_\alpha$ and hence
$E(\bigvee_\alpha X_\alpha)\cong\bigvee_\alpha E(X_\alpha)$ we conclude that
$\{i^*_\alpha\}:E^n(\bigvee_\alpha X_\alpha)\to\prod_\alpha E^n(X_\alpha)$ is an isomorphism for all $n\in\mathbb{Z}$. In other words $E^\bullet$ satisfies the wedge axiom.
One can also prove that $E_\bullet$ satisfies the wedge axiom.\end{proof}

\begin{cor} For any spectrum $E$ and any filtration $\{X^n\}$ of a CW-complex $X$ we have an exact sequence
$$\xymatrix{0\ar[r]&\lim^1E^{q-1}(X^n)\ar[r]& E^q(X)\ar[r]^(0.4){\{i^*_n\}}&\lim^0E^q(X^n)\ar[r]& 0}.$$
\end{cor}

\begin{proposition} We can extend the cohomology theory $E^\bullet$ to a cohomology theory on the category $\mathcal{S}'$ by simply taking
$$E^n(F)=[F,\Sigma^n(E)],\hskip 5pt n\in\mathbb{Z},\hskip 2pt F\in Ob(\mathcal{S}').$$
Furthermore, if $T^\bullet:E^\bullet\to F^\bullet$ is a natural equivalence of cohomology theories on $\mathcal{S}'$, we can show $T^\bullet=T^\bullet(f)$ for some map $f:E\to F$.
\end{proposition}
\begin{proof} In fact $E^\bullet$ is a cohomology theory in the sense that we have natural equivalences
$$\xymatrix{E^{n+1}(F\wedge S^1)\ar@{=}[d]\ar[r]& E^{n+1}(\Sigma F)\ar[r]^{\cong}&E^n(F)\ar@{=}[d]\\ E^{n+1}(F\wedge S^1)\ar[rr]_{\sigma^n}&&E^n(F)}$$
for all $n\in\mathbb{Z}$, $F\in Ob(\mathcal{S}')$. Furthermore, $E^\bullet$ satisfies the following exactness axiom: For any cofibre sequence $F\mathop{\to}\limits^{f}G\mathop{\to}\limits^{g}H$, the sequence
$$\xymatrix{E^n(F)\ar[r]^{f^*}&E^n(G)\ar[r]^{g^*}&E^n(H)}$$
is exact. (This axiom is equivalent to the usual one over $\mathcal{W}'_\bullet$.)\end{proof}

\begin{proposition} {\em 1)} A possible extension of the homology theory $E^\bullet$ to a homology theory on the category $\mathcal{S}'$ is the following
$$E_n(G)=\pi_n(E\wedge G)\equiv[\Sigma^nS^0,E\wedge G].$$
In this case, however, it is not assured that a natural
transformation $T_\bullet:E_\bullet\to F_\bullet$ on $\mathcal{S}'$
is of the form  $T^\bullet=T^\bullet(f)$ for some map $f:E\to F$.

{\em 2)} If $E$ is an $\Omega$-spectrum, then for every CW-complex
$(X,x_0)$ we have a natural isomorphism
$E^n(X)\cong[X,x_0;E_n,*]$.\end{proposition}

\begin{example}[Example of (co)homology theories associated to spectra] Let us consider
the {\em sphere spectrum} $S^0\equiv E(S^0)$. The associated
homology theory:
$$S^0_\bullet(X)=\pi_\bullet(S^0\wedge X)=\{\mathop{\lim}\limits_{\overrightarrow{k}}\pi_{n+k}(S^k X,*)\equiv\pi_n^s(X)\},$$
 is called {\em stable homotopy} (of $X$).
Furthermore, the associated cohomology theory:
$$(S^0)^\bullet(X)=\pi^\bullet_s(X)\equiv\{
\mathop{\lim}\limits_{\overrightarrow{k}}\pi_{n+k}(S^k\wedge X)\},$$
 is called {\em stable cohomotopy} (of $X$).
For any $n\ge 2$ we have the natural map
$$i_0:\pi_n(X,x_0)\to \mathop{\lim}\limits_{\overrightarrow{k}}\pi_{n+k}(S^kX,*)=\pi^s_n(X),\hskip 2pt
i_0(x)=\{x\}.$$
We can also define $i_0$ as follows: any map $f:(S^n,s_0)\to(X,x_0)$ defines a function $E(f):E(S^n)\to E(X)$. Since $E(S^n)$ is a cofinal subspectrum of $\Sigma S^0$, we get a map $\{E(f)\}:\Sigma^nS^0\to E(X)$, and
$$i_0[f]=[\{E(f)\}]\in[\Sigma^nS^0,E(X)]=\pi^s_n(X).$$
This definition of $i_0$ applies even for $n=0$ or $1$. $i_0$ is a
homomrphism for $n\ge 1$. The coefficient groups
$\pi^s_n(S^0)=\mathop{\lim}\limits_{\overrightarrow{k}}\pi_{n+k}(S^k,s_0)$ are called {\em stable
homotopy groups} or {\em $n$-stems}, and denoted by $\pi^s_n$.
These groups are known only through a finite range of $n>0$. In
particular, one has: $\pi^s_n=0$, $n<0$, $\pi^s_0\cong\mathbb{Z}$.\end{example}

\begin{proposition}
Let denote $\mathcal{T}'_{op,\bullet}$, (resp. $\mathcal{T}^{2'}_{op,\bullet}$, resp. $\mathcal{T}'_{op}$), the category of topological spaces, (resp. pointed topological spaces, resp. couples of pointed topological spaces), with morphisms homotopy classes of maps structures preserving.
For every spectrum $E$ we can define a reduced
homology theory $E_\bullet(-)$ and a reduced cohomology theory
$E^\bullet(-)$ on $\mathcal{T}'_{op,\bullet}$, $\mathcal{T}^{2'}_{op,\bullet}$ and $\mathcal{T}'_{op}$ respectively.\end{proposition}
\begin{proof}
In fact, for $X\in Ob(\mathcal{T}'_{op,\bullet})$ we have the
following reduced homology theory: $\tilde {E}_\bullet(X)\equiv
E_\bullet(X')=\pi_\bullet(E\wedge X')$, where $X'$ is any
CW-substitute for $X$. Furthermore, for any $(X,A)\in Ob(\mathcal{T}^{2'}_{op,\bullet})$ we have the following reduced homology
theory: $E_\bullet(X,A)\equiv\tilde E_\bullet(X^+\cup CA^+)$.
Finally for any space $X\in Ob(\mathcal{T}'_{op})$ we  have
$E_\bullet(X)=E_\bullet(X,\varnothing)$. Furthermore, $E^n(X)=[X,E_n]$. The coefficients of these theories are the groups $E^\bullet(*)\cong E_\bullet(*)=\pi_\bullet(E)$.\end{proof}

The calculation of generalized homology theories can be made easier by using spectral sequences. Relations between such structures are given by the following two theorems.

\begin{theorem}[Atiyah-Hirzebruch-Whitehead]\label{atiyah-hirzebruch-whitehead-theorem} Suppose
$\{E_n\}$ be a spectrum and $X$ a space. Then, there is a spectral
sequence $\{E_r^{\bullet,\bullet},d_r\}$ with
$$E_2^{p,q}\cong H^p(X;E^q(*))$$
converging to $E^\bullet(X)$. Furthermore, there is also a spectral sequence $\{E^r_{\bullet,\bullet},d^r\}$ with
$$E^2_{p,q}\cong H_p(X;E^q(*))$$
converging to $E_\bullet(X)$.\footnote{Let $ \mathcal{U}$ be an abelian category. A {\em differential object}
in $\mathcal{U}$ is a pair $ (A,d)$ where
$ A\in Ob(\mathcal{U})$ and $ d\in Hom_\mathcal{U}(A;A)$ such that $ d^2=0$. Let $ \mathcal{D}(\mathcal{U})$ be the category of differential objects in
$ \mathcal{U}$. We call {\em homology} the
additive functor $ H:\mathcal{D}(\mathcal{U})\to\mathcal{U}$,
given by $ H(A,d)=\ker(d)/\IM(d)=Z(A)/B(A)$, where
$ Z(A)$ is the set of {\em cycles} of
$ A$ and $ B$ is the set of {\em boundaries} of $ A$. $ H(A)$ is the
{\em homology} of $ (A,d)$. Then, a {\em spectral sequence} in the category $ \mathcal{U}$ is
a sequence of differential objects of $ \mathcal{U}$:
$ \{E_n,d_n\}$, $ N=1,2,\cdots$, such that
$ H(E_n,d_n)=E_{n+1}$, $ n=1,2,\cdots$.
(See, e.g.,  \cite{MCCLEARY, PRA000}.)} Here $E_\bullet(-)$ (resp.
$E^\bullet(-)$) is the homology (resp. cohomology) associated to
the spectrum $\{E_n\}$.\end{theorem}

\begin{theorem}[Leray-Serre] If
$E_\bullet$ is a homology theory with products satisfying the
wedge axiom for CW-complexes and the WHE axiom, then for every
fibration $p:E\to B$ orientable with respect to
$E_\bullet$,\footnote{For the homological definition of
orientability see next Remark \ref{fundamental-homology-class-manifold}.} and with $B$ {}$0$-connected,
there is a spectral sequence $\{E^r_{p,q},d^r\}$ converging to
$E_\bullet(E)$ and having
$$E^2_{p.q}\cong H_p(B;E_q(F)).$$
The spectral sequence is natural with respect to a fibre map.
\end{theorem}
\begin{remark} The problem of extension of maps and sections of fiber
bundles is related to (co)homology theories. In fact we have the
following theorems.\end{remark}

\begin{theorem} Let $K$ be a cell complex
and let $L\subset K$ be a subcomplex. Let $X$ be a
simply-connected topological space (or at least that is
homotopy-simple in the sense that $\pi_1(X)$ is abelian and
acts trivially on all the groups $\pi_i(X)$, $i>1$.) A given map
$f:L\to X$, can be extended from the subcomplex $L\bigcup K^{i-1}$
to $L\bigcup K^{i}$, if $\pi_{i-1}(X)=0$.
\end{theorem}
\begin{proof} In fact the
obstruction to such an extension is determined by an element
$\alpha_f$ of the relative cohomology group
$H^i(K,L;\pi_{i-1}(X))$. The vanishing of $\alpha_f$ in this group
suffices for the map to be extendible. In particular, the
extension is assured if $\pi_{i-1}(X)=0$. (For more details see
e.g. Refs. \cite{DUB-FOM-NOV, PRA3} and works quoted there.)\end{proof}

\begin{theorem} Let $f,g:K\to X$ be two maps which coincide on the
$(q-1)$-skeleton $K^{q-1}$ of $K$. On each cell $\sigma^q\subset
K^q$ the two maps $f$ and $g$ give rise, via their restrictions,
to two maps $f,g:\sigma^q\to X$ coinciding on the boundary:
$f|_{\partial\sigma^q}=g|_{\partial\sigma^q}$, and therefore
yielding in combination a map $S^q\to X$, determining what is
called a ''distinguishing element'' of $\pi_q(X)$, i.e., for each
$q$-cell $\sigma^q$ of $K$, we have a {\em difference cochain}
$\alpha(\sigma^q,f,g)\in\pi_q(X)$. Then, the difference cochain
may be regarded as belonging to the cohomology group
$H^q(K;\pi_q(X))$.\end{theorem}

\begin{theorem} If $X=K(G,n)$ is a {\em
Eilenberg-MacLane space} then there is a natural one-to-one
correspondence $[K,X]\leftrightarrow H^n(K;G)$. In the case $n=1$,
the elements of $H^1(K;G)$ and $[K,X]$ are determined by the
homomorphisms $\pi_1(K)\to G$. (This theorem remains true even if
$G$ is non-abelian.)
\end{theorem}

\begin{example} One has a natural one-to-one
correspondence $[K^n,S^n]\leftrightarrow H^n(K^n;\mathbb{Z})$, where
$K^n$ is an $n$-dimensional complex.\end{example}

\begin{theorem} Let $\pi:E\to B$
be a fibre bundle with base $B$ given as a simplicial (or cell)
complex and fibre $F$. We shall assume that $B$ is
simply-connected (or at least that $\pi_1(B)$ acts trivially on
the groups $\pi_i(F)$. We shall assume also that the fibre $F$ is
simply-connected (or at least homotopy-simple). Suppose
$s:B^{q-1}\to E$ be a croos-section of the fibre bundle above the
$(q-1)$-skeleton $B^{q-1}\subset B$. An obstruction to extending a
cross-section may be regarded as an element of
$H^q(B;\pi_{q-1}(F))$. In particular, if the fibre is the
$(q-1)$-sphere $S^{q-1}$, then the obstruction $\alpha\in
H^q(B;\pi_{q-1}(F))$ is an {\em Euler characteristic class} of
the fibre bundle.\end{theorem}

\begin{proof} Let $\sigma^q$ be any $q$-simplex of
$B$. Above the simplex $\sigma^q$ the fibre bundle is canonically
identifiable with the direct product:
$\pi^{-1}(\sigma^q)\cong\sigma^q\times F$. As on the boundary
$\partial\sigma^q\cong S^{q-1}$ the cross-section
$s:\partial\sigma^q\to\partial\sigma^q\times F$ is by assumption,
already given. Hence via the projection map onto $F$ we obtain a
map $S^{q-1}\to F$, defining an element
$\alpha(\sigma^q,s)\in\pi_{q-1}(F)$ for each $q$-simplex
$\sigma^q\subset B^q$. Therefore an {\em obstruction cocycle}
$\alpha$ to the attempted extension of the cross-section $s$ to
the $q$-skeleton $B^q$, belongs to $H^q(B;\pi_{q-1}(F))$.\end{proof}

\begin{theorem} Let $\varphi_i:B\to E$, $i=1,2$, be two cross-sections agreeing
on the $(q-1)$-skeleton $B^{q-1}\subset B$. The obstruction to a
homotopy between the cross-sections $\varphi_1$ and $\varphi_2$,
$\alpha(\varphi_1,\varphi_2)\in H^q(B;\pi_q(F))$.
\end{theorem}

\begin{cor}\label{existence-cross-sections-contractible-fibre}
If the fibre is contractible, ($\pi_i(F)=0$ for all $i$), then it follows
that cross-sections always exist, and moreover that all
cross-sections are homotopic.
\end{cor}
\begin{example} This is the situation for
the fiber bundle of positive definite quadratic forms, where the
cross-sections are Riemannian metrics. So over a manifold $M$
Riemannian metrics always exist and are homotopic, i.e., any two
Riemannian metrics are continuously deformable one into the other.
For indefinite metrics of type $(p,q)$ with $p+q=n$, this results
is not more valid. In fact in these cases one has
$\pi_i(F)=\pi_i(GL(n;\mathbb{R})/O(p,q))\not=0$.\end{example}

\begin{example} Connections on a fibre bundle $E\to B$, with fibre $F$, always
exist. In fact, such connections can be identified with sections
of the fibre bundle of horizontal directions over any point $x\in
B$.\end{example}
\begin{remark}[Fundamental homology class of manifold]\label{fundamental-homology-class-manifold}
Let $\Lambda$ be a commutative ring. Let $M$ be a
fixed $n$-dimensional manifold, not necessarily compact. Let
$K\subset M$ denote a compact subset of $M$. If $K\subset L\subset
M$, one has a natural homomorphism $\rho_K:H_i(M,M\setminus
L;\Lambda)\to H_i(M,M\setminus K;\Lambda)$. If $a\in
H_i(M,M\setminus L;\Lambda)$, then we call $\rho_K(a)$ the {\em
restriction} of $a$ to $K$. The groups $H_i(M,M\setminus
K;\Lambda)$ are zero for $i>n$. A homology class $a\in
H_n(M,M\setminus K;\Lambda)$ is zero iff the restriction
$\rho_x(a)\in H_n(M,M\setminus x;\Lambda)$ is zero for each $x\in
K$. Let us, now, take $\Lambda=\mathbb{Z}$. Then
$$H_i(M,M\setminus x;\mathbb{Z})\cong
H_i(\mathbb{R}^n,\mathbb{R}^n\setminus \{0\};\mathbb{Z})=\left\{\begin{array}{l}
                                                                0,\hskip 2pt i\not=n\\
                                                                \hbox{infinite cyclic},\hskip 2pt i=n.\\
                                                                \end{array}\right.$$
A {\em local orientation}  $\mu_x$ for $M$ at $x$ is a choice of one of two possible
generators for $H_n(M,M\setminus x;\mathbb{Z})$. Note that such a
$\mu_x$ determines local orientations $\mu_y$ for all points $y$
in a small neighborhood of $x$. In fact, if $B$ is a ball about
$x$, then for each $y\in B$ the isomorphisms
$$\xymatrix{H_\bullet(M,M\setminus x;\mathbb{Z})\ar[r]^{\rho_x}&H_\bullet(M,M\setminus B;\mathbb{Z})&
H_\bullet(M,M\setminus y;\mathbb{Z})\ar[l]^{\rho_y}}$$
determine a local orientation $\mu_y$. An {\em orientation}
for $M$ is a function which assigns to each $x\in M$ a local
orientation $\mu_x$ which continuously depends on $x$, i.e., for
each $x$ there should exist a compact neighborhood $N$ and a
class $\mu_N\in H_n(M,M\setminus N;\mathbb{Z})$ so that
$\rho_y(\mu_N)=\mu_y$ for each $y\in N$. An {\em oriented
manifold} is a manifold $M$ endowed with an orientation. For any
oriented manifold $M$ and any compact $K\subset M$, there is one
and only one $\mu_K\in H_n(M;M\setminus K;\mathbb{Z})$ which
satisfies $\rho_x(\mu_K)=\mu_x$ for each $x\in K$. In particular,
if $M$ is compact, then there is one an only one $\mu_M\in
H_n(M;\mathbb{Z})$ with the required property. This class
$\mu\equiv\mu_M$ is called the {\em fundamental homology
class} of $M$. As $H_n(M;\mathbb{Z})\cong\mathbb{Z}^r$, for oriented
manifold, with $r$ the number of connected components of $M$, it
follows that $\mu_M=(1,\cdots_r\cdots,1)$ is the basis of the
$\mathbb{Z}$-module $H_n(M;\mathbb{Z})$. For any coefficient domain
$\Lambda$, the unique homomorphism $\mathbb{Z}\to\Lambda$ gives rise
to a class in $H_n(M,M\setminus K;\Lambda)$ that will also be
denoted by $\mu_K$. For example, we can take $\Lambda\equiv\mathbb{Z}_2$, so that $\mu_K\in H_n(M,M\setminus K;\mathbb{Z}_2)$. This
homology class can be constructed directly for any $n$-dimensional
manifold, without making any assumption of orientability. In
particular, if $M$ is a non-orientable compact manifold of
dimension $n$, with $r$ connected components, one has $H_n(M;\mathbb{Z}_2)\cong(\mathbb{Z}_2)^r$. Similar considerations apply to an
oriented manifold with boundary. For each compact subset $K\subset
M$, there exists a unique class $\mu_K\in H_n(M,(M\setminus
K)\cup\partial M;\mathbb{Z})$ with the property that
$\rho_x(\mu_K)=\mu_x$ for each $x\in K\cap(M\setminus\partial M)$.
In particular, if $M$ is compact, then there is a unique
fundamental homology class $\mu_M\in H_n(M,\partial M;\mathbb{Z})$
with the required property. Then, the connecting homomorphism
$\partial:H_n(M,\partial M;\mathbb{Z})\to H_n(\partial M;\mathbb{Z})$
maps $\mu_M$ to the fundamental homology class of $\partial
M$.\end{remark}

\begin{remark}[Stiefel-Whitney characteristic classes and
Stiefel-Whitney characteristic numbers] Given a vector bundle
$p:E\to B$,
 fibre $\mathbb{R}^n$, and bundle group $G=O(n)$, we can form the associated bundle $p_k:E_k\to B$ of orhonormal $k$-frames with fibre
$F_k\equiv V_{n,k}$, the Stiefel manifold of orthonormal $k$-frames in $\mathbb{R}^n$.\footnote{In particular, for
$ k=n$, $ F_k\cong O(n)$, and for $ k=1$, $ F_1\cong S^{n-1}$.}
As
$$\pi_i(V_{n,k})=\left\{\begin{array}{l}
        0,\hskip 5pt i<n-k\\
        \mathbb{Z},\hskip 5pt n-k=2r+1, \hskip 2pt \hbox{or $k=1$}\\
        \mathbb{Z}_2,\hskip 5pt n-k=2r,\\
                        \end{array}\right.$$ it follows that for each
$k=1,\cdots,n$ the obstruction to the existence of a cross-section
of the fibre bundle $p_k:E_k\to B$ will be an element $\alpha_k\in
H^{n-k+1}(B;\pi_{n-k}(V_{n,k}))$. The cohomology class $\alpha_k$,
considered module 2, is called the {\em $k$th Stiefel-Whitney
class of the vector bundle $p:E\to B$},\footnote{By the
Stiefel-Whitney classes of an $ n$-dimensional
manifold $ M$, one means the corresponding classes of
$ TM$.} and we write
$$\left\{\begin{array}{l}
w_0=1,\\
w_q\equiv\alpha_{n-q+1}\hskip 2pt\MOD 2\hskip 2pt\in
H^q(B;\mathbb{Z}_2),\hskip 5pt q=1,\cdots,n.,\\
w_q=0,\hskip 5pt q>n.\\
\end{array}\right.$$ The polynomial $w(t)\equiv
w_0+w_1t+\cdots+w_qt^q+\cdots+w_nt^n$ is called the {\em
Stiefel-Whitney polynomial} of the vector bundle $p:E\to B$.
We call $w(E)=1+w_1+\cdots+w_n$ the {\em Stiefel-Whitney class} of
$p:E\to B$. If the manifold $M$ ($\dim M=n$) is orientable, one
has $w_1=0$. In fact, the natural mapping $j:BSO(n)\to BO(n)$,
which ''forgets'' the orientation on the oriented $n$-dimensional
planes representing the points in $\hat
G_{\infty,n}=BSO(n)$,\footnote{universal classifying space for
the group $ SO(n)$.} induces an epimorphism
$j^\bullet:H^\bullet(BO(n);\mathbb{Z}_2)\to H^\bullet(BSO(n);\mathbb{Z}_2)$, with kernel $<w_1>$, the ideal generated by the first
Stiefel-Whitney class $w_1\in H^1(BO(n);\mathbb{Z}_2)$. Therefore,
for any fiber bundle $W_O$ and $W_{SO}$, over a manifold $X$, with
structure groups $O(n)$ and $SO(n)$ respectively, we get the
following commutative diagram:
$$\xymatrix{0\ar[r]&<w_1>\ar[r]&H^\bullet(BO(n);\mathbb{Z}_2)\ar[d]_{\cong}\ar[r]^{\pi}&
H^\bullet(BSO(n);\mathbb{Z}_2)\ar[d]_{\cong}\ar[r]& 0\\ 0\ar[r]&<w_1|_X>\ar[r]&Kar^\bullet(W_O;\mathbb{Z}_2)\ar@{^{(}->}[d]\ar[r]^{\bar\pi}&Kar^\bullet(W_{SO};\mathbb{Z}_2)
\ar@{^{(}->}[d]\ar[r]& 0\\
 &&H^\bullet(X;\mathbb{Z}_2)&H^\bullet(X;\mathbb{Z}_2)&}$$ Therefore $W_O$ admits the
reduction to $W_{SO}$ if $\bar\pi$ is injective, i.e., if
$w_1|_X=0$. In particular, $X$, $\dim X=n$, is orientable iff $TX$
admits the reduction to $SO(n)$, i.e., its first Stiefel-Whitney
class is zero. We get also the following propositions:

{\em(i)} $w_n=\chi(X)$ mod $2$, where $\chi(X)$ is the Euler characteristic of $X$.

{\em(ii)} For the direct product $(E_1\times E_2, p_1\times p_2,
B_1\times B_2)$ of vector bundles, one has
$w(t)=\mathop{w}\limits^1(t)\mathop{w}\limits^2(t)$, where
$\mathop{w}\limits^i(t)$, $i=1,2$, are the Stiefel-Whitney
polynomials of the factors.

{\em(iii)} Let $E_1\bigoplus E_2$ be the Whitney sum of two real vector bundles over the same base, then $w(E_1\bigoplus E_2)=w(E_1)w(E_2)$,
i.e., $w_k(E_1\bigoplus E_2)=\sum_{0\le i\le k}w_i(E_1)w_{k-i}(E_2)$.

{\em(iv)} One has the following isomorphism: $H^\bullet(M;\mathbb{Z}_2)\cong\mathbb{Z}_2[w_1,\cdots,w_n]$, $\dim M=n$.

Let $G_{N,k}$ be the Grassmannian manifold that represents the set
of $k$-dimensional vector spaces of $\mathbb{R}^N$. Then
$G_{\infty,k}=BO(k)$ is the universal classifying space for the
orthogonal group $O(k)$. Let $f:M\to\mathbb{R}^N$ be an embedding of
a $k$-dimensional manifold $M$ into $\mathbb{R}^N$, (for enough large
$N$). Then we have the following map ({\em generalized Gauss
map})
$$\tau_M:M\to G_{\infty,k},\hskip 5pt x\mapsto T_xM\hookrightarrow T_{f(x)}\mathbb{R}^N\subset T_{f(x)}\mathbb{R}^\infty.$$
This induces the tangent bundle $TM$ from the universal bundle,
$V_{\infty,k}\to G_{\infty,k}$, of orthonormal tangent $k$-frames on $M$, with respect to the induced metric. So we have the following commutative diagram:
$$\xymatrix{\tau_M^*V_{\infty,k}\equiv TM\ar[d]\ar[r]&V_{\infty,k}\ar[d]\\
M\ar[r]_{\tau_M}&G_{\infty,k}}$$ Each element
$w\in H^s(G_{\infty,k};\mathbb{Z}_2)$ determines a corresponding mod
$2$ {\em characteristic class} $w(M)\equiv\tau_M^*w$, and the
{\em stable} \MOD $2$ {\em characteristic classes}
 of $M$ (with respect to the group $O(k)$) are those determined by elements $w\in H^s(BO(k);\mathbb{Z}_2)$ which are pull-backs of elements $\bar w\in H^s(BO(k+1);\mathbb{Z}_2)$
via the natural embeddings $\lambda:BO(k)\to BO(k+1)$, induced by the standard embedding $\lambda:O(k)\to O(k+1)$: $w=\lambda^*\bar w$.
So if $w(M)$ is a stable mod $2$ characteristic class of $M$, then $w(M)=\tau_M^*\lambda^*\bar w$, for some
$\bar w\in H^s(BO(k+1);\mathbb{Z}_2)$.

Let us assume, now, that $M=\partial W$, $dim W=k+1$, we have: $\bar w(W)=\tau^*_W\bar w$, and taking into account the inclusion map $i:M\to W$, the restriction to $M$ of the map $\tau_W:W\to BO(k+1)$, satisfies
$$\begin{array}{ll}
  \tau_W|_M&=\tau_W\circ i=\lambda\circ\tau_M\\
  &=\tau_M\bigoplus 1:M\to BO(k+1),\\
  \end{array}$$
inducing the Whitney sum $TM\bigoplus_MT^0_0M$. Therefore $w(M)=i^*\bar w(W)$. Now since $M=\partial W$ it follows that, for the fundamental homology class $[M]$, we have $i_*[M]=0$. Therefore, assuming
$w(M)\in H^k(M,\mathbb{Z}_2)$, its evaluation on $[M]$ gives:
$$<w(M),[M]>=<i^*\bar w,[M]>=<\bar w,i_*[M]>=<\bar w,0>=0.$$
Since $[M]$ generate $H_k(M;\mathbb{Z}_2)$\footnote{If $ M$ is a closed and connected manifold of dimension
$ k$, admitting a finite triangulation, then $ H_k(M;\mathbb{Z}_2)\cong\mathbb{Z}_2$. The fundamental class of
$ M$ is $[M]=\sum_iv_i^k$, i.e., the sum of all $ k$-simplexes.} it follows that the {\em Stiefel-Whitney numbers} of $M$, i.e., the values taken on $[M]$ by its mod $2$ stable characteristic classes of dimension $k$, are zero.\end{remark}

We have the following theorem.

\begin{theorem}[Pontrjagin] If $B$ is a smooth compact
$(n+1)$-dimensional manifold with boundary $M\equiv\partial B$,
then the Stiefel-Whitney numbers of $M$ are all zero.
\end{theorem}
\begin{proof}
Here, let us give, also, another direct proof  to this important
theorem . Let us denote the fundamental homology class of the pair
$(B,\partial B)$ by $\mu_B\in H_{n+1}(B,\partial B;\mathbb{Z}_2)$.
Then, the natural homomorphism $$\partial:H_{n+1}(B,\partial B;\mathbb{Z}_2)\to H_n(\partial B;\mathbb{Z}_2)$$ maps $\mu_B$ to $\mu_{\partial
B}$. For any class $v\in H^n(M;\mathbb{Z}_2)$ one has: $<v,\partial
\mu_B>=<\delta v,\mu_B>$, where $\delta$ is the natural
homomorphism $\delta:H^n(\partial B;\mathbb{Z}_2)\to
H^{n+1}(B,\partial B;\mathbb{Z}_2)$. (There is not sign since we are
working mod $2$.) Consider the tangent bundles $TB|_{\partial B}$
and $T(\partial B)\subset TB|_{\partial B}$. Choosing a Euclidean
metric on $TB$, there is a unique outward normal vector field
along $\partial B$, spanning a trivial line bundle $\epsilon^1$,
and it follows that $TB|_{\partial B}\cong T(\partial
B)\bigoplus\epsilon^1$. Hence the Stiefel-Whitney classes of
$TB|_{\partial B}$ are precisely equal to the Stiefel-Whitney
classes $w_j$ of $T(\partial B)$. Using the exact sequence
$$\xymatrix{H^n(B;\mathbb{Z}_2)\ar[r]^{i^*}&H^n(\partial B;\mathbb{Z}_2)\ar[r]^(0.4){\delta}
&H^{n+1}(B,\partial B;\mathbb{Z}_2)}$$
it follows that $\delta(w_1^{r_1}\cdots w_n^{r_n})=0$ and therefore
$$<w_1^{r_1}\cdots w_n^{r_n},\partial\mu_B>=<\delta(w_1^{r_1}\cdots w_n^{r_n}),\mu_B>=0.$$
As $\partial\mu_B=\mu_{\partial B}$, we can conclude that all Stiefel-Whitney numbers of $\partial B$ are zero.\end{proof}

\begin{definition} In the category of closed smooth (resp. oriented) manifolds of
dimension $n$, we can define, by means of bordism properties, an
equivalence relation. More precisely, we say that $X_1\sim X_2$
iff $X_1\sqcup X_2=\partial W$, where $W$ is a smooth manifold of
dimension $n+1$. The corresponding set $\Omega_n$ (resp.
${}^+\Omega_n$) of equivalences classes is called the {\em
$n$-dimensional bordism group} (resp. {\em oriented
$n$-dimensional bordism group}).\end{definition}

Now, the nullity of the
Stiefel-Whitney numbers is also a sufficient condition to bording.
In fact we have the following.

\begin{theorem}[Pontrjagin-Thom] A closed
$n$-dimensional smooth manifold $V$, belonging to the category of
smooth differentiable manifolds, is bordant in this category,
i.e., $V=\partial M$, for some smooth $(n+1)$-dimensional manifold
$M$, iff the Stiefel-Whitney numbers $<w_{i_1}\cdots
w_{i_p},\mu_V>$ are all zero, where $i_1+\cdots +i_p=n$ is any
partition of $n$ and $\mu_V$ is the fundamental class of $V$.
Furthermore, the bordism group $\Omega_n$ of $n$-dimensional
smooth manifolds is a finite abelian torsion group of the form
$$\Omega_n\cong\underbrace{\mathbb{Z}_2\oplus\cdots\oplus\mathbb{Z}_2}_q,$$ where $q$ is the number of nondyadic partitions of
$n$.\footnote{A partition $(i_1,\cdots,i_r)$ of
$ n$ is nondyadic if none of the $
i_\beta$ are of the form $ 2^s-1$.} Two smooth closed
$n$-dimensional manifolds belong to the same bordism class iff all
their corresponding Stiefel-Whitney numbers are equal.
Furthermore, the bordism group ${}^+\Omega_n$ of closed
$n$-dimensional oriented smooth manifolds is a finitely generated
abelian group of the form $${}^+\Omega_n\cong\mathbb{Z}\bigoplus\cdots\bigoplus\mathbb{Z}\bigoplus\mathbb{Z}_2\bigoplus\cdots\bigoplus\mathbb{Z}_2,$$ where infinite cyclic
summands can occur only if $n\equiv 0$ $\MOD 4$. Two smooth closed
oriented $n$-dimensional manifolds belong to the same bordism
class iff all their corresponding Stiefel-Whitney and Pontrjagin
numbers are equal.\footnote{{\em Pontrjagin numbers}
are determined by means of homonymous characteristic classes
belonging to $ H^\bullet(BG,\mathbb{Z})$, where
$ BG$ is the classifying space for $
G$-bundles, with $ G=S_p(n)$.}
The bordism groups $\Omega_p$, (resp. ${}^+\Omega_p$), by disjoint union and topological product of manifolds induce addition and multiplication operators with respect to which the cobordism classes form a graded ring, the {\em bordism ring} $\Omega_\bullet\equiv\bigoplus_{p\ge 0}\Omega_p$, (resp. the {\em oriented bordism ring} ${}^+\Omega_\bullet\equiv\bigoplus_{p\ge 0}{}^+\Omega_p$) that is a polynomial ring over $\mathbb{Z}_2$.
\end{theorem}

\begin{proof} See, e.g., \cite{PONT, STONG, THOM, WALL1}.\end{proof}

\begin{theorem}[Dold \cite{DOLD}]\label{dold-theorem}
Let us call {\em Dold manifold} $P(m,n)$, the bundle over $\mathbb{RP}^n$ with fibre $\mathbb{CP}^n$, defined by the following $P(m,n)\equiv(S^m\times\mathbb{CP}^n)/\tau$, where $\tau$ is the involution mapping $(x,[y])\mapsto(-x,[\bar y])$, where $\bar y=(\bar y_0,\cdots,\bar y_n)$ for $y=(y_0,\cdots,y_n)$. The {\em bordism ring} $\Omega_\bullet\equiv\bigoplus_{p\ge 0}\Omega_p$ is a polynomial ring over $\mathbb{Z}_2$:
$$\Omega_\bullet\cong\mathbb{Z}_2[x_2,x_4,x_5,x_6,x_8,\cdots, x_i,\cdots], \quad i\not=2^k-1$$
where the polynomial generators $x_i$ are given by Dold manifolds. More precisely one has:\footnote{\hskip 2pt$\mathbb{RP}^k$ are orientable manifolds iff $k\in\mathbb{N}$ is odd. $P(2^r-1,s2^r)$ are orientable manifolds. One has $\dim P(m,n)=m+2n$.}
$$\left\{\begin{array}{l}
    \hbox{\rm For $i$ even}\hskip 2pt x_i=[P(i,0)]=[\mathbb{RP}^i]\\
     \hbox{\rm For $i=2^r(2s+1)-1$ }\hskip 2pt x_i=[P(2^r-1,s2^r)].\\
  \end{array}\right.$$
\end{theorem}

\begin{theorem}[Wall \cite{WALL1, WALL2}]\label{wal-theorem}
There is a natural map $r:{}^+\Omega_\bullet\to\Omega_\bullet$, obtained by ignoring orientation, and a polynomial subalgebra ${}^\blacksquare\Omega_\bullet\subset\Omega_\bullet$, containing $r({}^+\Omega_\bullet)$, and a map $\partial:{}^\blacksquare\Omega_\bullet\to{}^+\Omega_\bullet$, such that the following diagram is commutative and exact.
$$\xymatrix{{}^+\Omega_\bullet\ar[rr]^{2}&&{}^+\Omega_\bullet\ar[dl]_{r}\\
0\ar[r]&{}^\blacksquare\Omega_\bullet\ar[lu]^{\partial}\ar[r]&\Omega_\bullet}$$

${}^\blacksquare\Omega_\bullet$ is defined as the subset of $\Omega_\bullet$ of classes containing a manifold $M$ such that the first Stiefel-Whitney class $w_1$ is the restriction of an integer class, and thus corresponds to a map $f:M\to S^1$.

${}^\blacksquare\Omega_\bullet$ contains:

{\em(i)} Dold manifolds representing the classes $x_i$, $i\not=2k-1$, in $\Omega_\bullet$;

{\em(ii)} manifolds $M_{2k}$ with $w_{2k}(M_{2k})=1$;

{\em(iii)} spaces $(\mathbb{CP})^{2^n}$.

Since by a computation with Stiefel-Whitney numbers $\mathbb{CP}^n$ and $(\mathbb{RP}^n)^2$ are cobordant, all these just generate the polynomial subalgebra ${}^\blacksquare\Omega_\bullet\subset\Omega_\bullet$.
\end{theorem}

\begin{definition} Let a {\em$k$-cycle} of $M$ be a couple $(N,f)$, where $N$ is
a $k$-dimensional closed (oriented) manifold and $f:N\to M$ is a
differentiable mapping. A {\em group of cycles} $(N,f)$ of an
$n$-dimensional manifold $M$ is the set of formal sums
$\sum_i(N_i,f_i)$, where $(N_i,f_i)$ are cycles of $M$. The
quotient of this group by the cycles equivalent to zero, i.e., the
boundaries, gives the {\em bordism groups}
${\underline{\Omega}}_s(M)$. We define {\em relative bordisms}
${\underline{\Omega}}_s(X,Y)$, for any pair of manifolds $(X,Y)$,
$Y\subset X$, where the boundaries are constrained to belong to
$Y$. Similarly we define the {\em oriented bordism groups}
${}^+{\underline{\Omega}}_s(M)$ and ${}^+{\underline{\Omega}}_s(X,Y)$.
\end{definition}

\begin{proposition} One has ${\underline{\Omega}}_s(*)\cong\Omega_s$
and ${}^+{\underline{\Omega}}_s(*)\cong {}^+\Omega_s$.\end{proposition}

\begin{proposition}
For bordisms, the theorem of invariance of homotopy is valid.
Furthermore, for any CW-pair $(X,Y)$, $Y\subset X$, one has the
isomorphisms: ${\underline{\Omega}}_s(X,Y)\cong{\Omega}_s(X/Y)$,
$s\ge 0$.\end{proposition}

\begin{theorem} One has a natural group-homomorphism
${\underline{\Omega}}_s(X)\to H_s(X;\mathbb{Z}_2)$. This is an
isomorphism for $s=1$. In general, ${\underline{\Omega}}_s(X)\not=
H_s(X;\mathbb{Z}_2)$.\end{theorem}

\begin{proof} In fact one has the following lemma.

\begin{lemma}[Quillen]\cite{QUILLEN} One has the canonical
isomorphism:
$${\underline{\Omega}}_p(X)\cong \bigoplus_{r+s=p} H_r(X;\mathbb{Z}_2)\otimes_{\mathbb{Z}_2}\Omega_s.$$ In particular, as
$\Omega_0=\mathbb{Z}_2$ and $\Omega_1=0$, we get
${\underline{\Omega}}_1(X)\cong H_1(X;\mathbb{Z}_2)$. Note that for
contractible manifolds, $H_s(X)=0$, for $s>0$, but
${\underline{\Omega}}_s(X)$ cannot be trivial for any $s>0$. So, in
general, ${\underline{\Omega}}_s(X)\not= H_s(X;\mathbb{Z}_2)$.
\end{lemma}
After these results and remarks, the proof of the theorem follows
directly.\end{proof}

\begin{definition} Let $B$ be a closed differential connected manifold and let
$\xi\equiv(p:E\to B,F\equiv\mathbb{R}^n,G)$ be a vector bundle over
$B$ with fibre $\mathbb{R}^n$ and structure group $G=O(n)$, $SO(n)$,
$U(n)$, $SU(n)$ or $S_p(n)$. Let $\tilde E\to B$ be the subbundle
of $\xi$ defined by the vectors in the fibers with length $\le 1$.
The fiber $F'$ of $\tilde E$ is $F'\equiv D^n\subset \mathbb{R}^n$.
The boundary $\partial\tilde E$ is a fiber bundle with fiber
$S^{n-1}$. The {\em Thom complex} of the vector bundle $\xi$
is the quotient complex $M(\xi)=\tilde E/\partial\tilde E$. So
$M(\xi)$ is the compactified to a point of $E$: $M(\xi)\equiv
E\cup\{\infty\}\equiv E^+$.\end{definition}

\begin{example} If $B=BG$, the base space of
the universal $G$-bundle, with fibre $\mathbb{R}^n$, we denote by
$MG$ the corresponding Thom complex. In particular, for $G=O(n)$,
$SO(n)$, $U(n/2)$, $SU(n/2)$, or $S_p(n/4)$, we denote the
corresponding Thom complexes by $MO(n)$, $MSO(n)$, $MU(n/2)$,
$MSU(n/2)$ and $MS_p(n/4)$ respectively. In some cases the
complexes
 $MO(s)$, $MSO(s)$ are Eilenberg-MacLane complexes of type $K(G,n)$. Tab. \ref{mos-and-msos-as-kgn-complexes} resumes such cases.
\begin{table}
\caption{$MO(s)$ and $ MSO(s)$ as $ K(G,n)$-complexes.}
\label{mos-and-msos-as-kgn-complexes}
\begin{tabular}{|l|l|}
\hline
{\rm{\footnotesize $ MO(1)\cong\mathbb{R}P^\infty\cong
K(\mathbb{Z}_2,1)$}}&{\rm{\footnotesize $\pi_j=0,\hskip 2pt j>1$}}\\
\hline
{\rm{\footnotesize $ MSO(1)\cong
Me\cong S^1\cong K(\mathbb{Z},1)$}}&{\rm{\footnotesize $\pi_j=0,\hskip 2pt j>1$}}\\
\hline
{\rm{\footnotesize $MSO(2)\cong\mathbb{C}P^\infty\cong K(\mathbb{Z},2)$}}&{\rm{\footnotesize $ \pi_j=0,\hskip 2pt j\not=2$}}\\
\hline
\multicolumn{2}{l}{\rm{\footnotesize$\phi(1)=u\in H^n(MG)$ is the
fundamental class of $K(G,n)$.}}\\
\multicolumn{2}{l}{\rm{\footnotesize(See Lemma \ref{lemma-natural-isomorphisms-algebraic-topology-a}.)}}\\
\end{tabular}
\end{table}
 The Thom complexes
$M(\xi)$ are simply connected for $n>1$. Their homotopy groups are
reported in Tab. \ref{homotopy groups-of-mxi}.
\begin{table}
\caption{Homotopy groups of $ M(\xi)$.}
\label{homotopy groups-of-mxi}
\begin{tabular}{|l|l|}
\hline
{\rm{\footnotesize $\pi_j(M(\xi))$}}&{\rm{\footnotesize Conditions}}\\
\hline
{\rm{\footnotesize $ 0$}}&{\rm{\footnotesize $
1\le j<n$}}\\
\hline
{\rm{\footnotesize $ \mathbb{Z}_2$}}&{\rm{\footnotesize $ j=n$,\hskip 2pt
non-orientable fiber bundle}}\\
\hline
{\rm{\footnotesize $ \mathbb{Z}$}}&{\rm{\footnotesize $ j=n$,\hskip 2pt
orientable fiber bundle}}\\
\hline
\end{tabular}
\end{table}
\end{example}

\begin{theorem}\label{realizations-cycles-theorems}
{\em 1)} A cycle $x\in H_s(M;\mathbb{Z}_2)$, $dim M=n+s$, is
realized by means of a closed $s$-dimensional submanifold
$N\subset M$, iff there exists a mapping $f:M\to MO(n)$ such that
$f^*u=Dx$, where $u\in H^n(MO(n);\mathbb{Z}_2)$ is a fundamental
class and $D:H_s(M;\mathbb{Z}_2)\to H^n(M;\mathbb{Z}_2)$ is the
Poincar\'e duality operator.

{\em 2)}  Let $M$ be an
$(n+s)$-dimensional oriented manifold. A cycle $x\in H_s(M;\mathbb{Z})$ is realized by means of a closed oriented submanifold
$N\subset M$ iff there exists a mapping $f:M\to MSO(n)$ such that
$f^*u=Dx$. A cycle $x\in H_s(M;\mathbb{Z})$ is realized by means of a
closed oriented submanifold $N\subset M$ of trivial normal bundle
(i.e., defined by means of a family of nonsingular equations
$\psi_1=0,\cdots,\psi_k=0$, in $M$) iff there exists a mapping
$f:M\to Me\cong S^n$ such that $f^*u=Dx$.

{\em 3)}  Similar theorems hold in the cases of realizations of cycles by means of
submanifolds with normal bundles endowed with structural groups
$U(n/2)$, $SU(n/2)$, $S_p(n/4)$. A mapping $M\to MU(n/2)$, $M\to
MSU(n/2)$ and $MS_p(n/4)$ generates such restrictions.\end{theorem}

\begin{proof} Let
us consider the following definitions and lemmas.

\begin{lemma}\label{lemma-natural-isomorphisms-algebraic-topology-a} One has the natural isomorphisms:
$$\begin{array}{l}
 \phi:H_i(B;A)\to H_{n+i}(M(\xi);A),\hskip 2pt \phi:H^i(B;A)\to H^{n+i}(M(\xi);A),\\
 z\mapsto\phi(z)\equiv p^*z\hskip 2pt(\MOD\hskip 2pt\partial\tilde E),\hskip 2pt i\ge 0,\hskip 2pt n=\dim F,\\
  \end{array}$$
where £ì$A\equiv\mathbb{Z}_2$ if $G=O(n)$, $A\equiv\mathbb{Z}$ if $G=SO(n)$, $A\equiv\mathbb{Q}$ if $G=U(n)$, $S_p(n)$. More precisely
$\phi=D_{\tilde E}\circ D_B$, where $D_X$ are the following duality operators:
$$\begin{array}{l}
D_B:H_q(B)\to H^{m-q}(B),\hskip 2pt \dim B=m,\\
D_{\tilde E}:H_{m-q}(\tilde E)\cong H^{m-q}(B)\to
H_{n+m-(m-q)}(\tilde E,\partial\tilde E)\cong
H_{q+n}(M(\xi)),\hskip 2pt q>0.\\
 \end{array}$$
 One has a {\em fundamental
class} in the {\em cohomology of Thom}  of $\xi$, i.e.,
$\phi(1)\in H^n(M(\xi))$. Furthermore, the following
identifications hold: $M(\xi)/B\equiv\xi$, $M(\xi)\setminus
B\cong\{*\}$, where $B$ is identified with a submanifold of
$M(\xi)$ by means of the zero section.\end{lemma}

\begin{lemma} The
Stiefel-Whitney class $w_i\in H^i(B;\mathbb{Z}_2)$ of a vector bundle
$\xi$ with base $B$ is related to the Thom complex $M(\xi)$ by the
following relation: $w_i=\phi^*s^i_q(\phi(1))$, where
$\phi:H^q(B;\mathbb{Z}_2)\to H^{n+q}(M(\xi);\mathbb{Z}_2)$ and $s^i_q$
are Steenrod squares, i.e., homomorphisms
$s^i_q:H^n(M(\xi);\mathbb{Z}_2)\to H^{n+i}(M(\xi);\mathbb{Z}_2)$. (See Theorem \ref{steenrod-algebra-and-stiefel-whitney-classes} and Tab. \ref{properties-steenrod-squares} for informations on Steenrod squares.)
\end{lemma}

\begin{definition} Let $X\subset Y$ be a smooth submanifold of a smooth manifold
$Y$, of codimension $k$. Let $M$ be another smooth manifold. Then
a smooth map $f:M\to Y$ is said to be {\em transversally
regular}  on $X$ if the rank of the map $Df(x):T_xM\to
T_{f(x)}Y/T_{f(x)}X$ is $k$ whenever $f(x)\in X$. In such a case
we write $f(M)\pitchfork X$.\end{definition}

The group $O(n)$ contains a subgroup of diagonal matrices $D(n)\subset O(n)$
such that $D(n)\cong\mathbb{Z}_2\times\cdots\times\mathbb{Z}_2$.
Furthermore, one has a canonical mapping between classifying
spaces
$$BD(n)\cong\mathbb{R}P_1^\infty\times\cdots\times\mathbb{R}P_n^\infty\mathop{\longrightarrow}\limits^i BO(n)$$
and the induced cohomology mapping
$$i^*:H^\bullet(BO(n);\mathbb{Z}_2)\to H^\bullet(BD(n);\mathbb{Z}_2).$$
One can see that $i^*$ is a monomorphism and that $\IM(i^*)$ is
the set of symmetric polynomials in $\chi_1,\cdots,\chi_n$, where
$0\not =\chi_i\in H^1(\mathbb{R}P^\infty;\mathbb{Z}_2)$.
Moreover, the Stiefel-Whitney classes are elementary symmetric
polynomials:
$i^*(w_q)=\sum_{i_1<\cdots<i_q}\chi_{i_1}\cdots\chi_{i_q}$. Recall
that the mapping $j:BSO(n)\to BO(n)$ induces an epimorphism
$$j^\bullet:H^\bullet(BO(n);\mathbb{Z}_2)\to H^\bullet(BSO(n);\mathbb{Z}_2),$$
with $\ker(j^\bullet)$ generated, as ideal, by the element $w_1\in
H^1(BO(n);\mathbb{Z}_2)$. Now, let $M(\xi)$ be the Thom complex of
the universal bundle $\xi$ on $BO(n)$. By means of the zero
section we can identify a mapping $f:BO(n)\to M(\xi)$, hence we
get also the following morphism
$$f^\bullet:H^\bullet(M(\xi);\mathbb{Z}_2)\to H^\bullet(BO(n);\mathbb{Z}_2),$$
that is a ring monomorphism and $\IM(f^\bullet)$ is the set of
polynomials in $w_i$ divisible by $w_n\in H^n(BO(n);\mathbb{Z}_2)$,
where $i^*w_n=\chi_1\cdots\chi_n$. Furthermore,
$$\begin{array}{l}
   f^*\phi(1)=w_n,\\
   f^*\phi(w_i)=s^i_q(w_n)=w_iw_n.\\
  \end{array}$$
In general $f^*\phi(x)=xw_n$. Similar results can be obtained for $H^\bullet(BSO(n);\mathbb{Z}_2)$.

\begin{lemma}\label{lemma-important-bordism-groups} The bordism groups
$\Omega_s$, ${}^+\Omega_s$, $\Omega_s^U$, $\Omega_s^{SU}$,
$\Omega_s^{Sp}$, are canonically isomorphic to the stable homotopy
groups. See Tab. \ref{bordism-groups-and-stable-homotopy-groups}.
\end{lemma}
\begin{table}
\caption{Bordism groups and stable homotopy groups.}
\label{bordism-groups-and-stable-homotopy-groups}
\begin{tabular}{|l|}
\hline
{\rm{\footnotesize $\Omega_s\cong\pi_{n+s}(MO(n))$}}\\
{\rm{\footnotesize ${}^+\Omega_s\cong\pi_{n+s}(MSO(n))$}}\\
{\rm{\footnotesize $\Omega_s^U\cong\pi_{n+s}(MU(n/2))$}}\\
{\rm{\footnotesize $\Omega_s^{SU}\cong\pi_{n+s}(MSU(n/2))$}}\\
{\rm{\footnotesize $\Omega_s^{Sp}\cong\pi_{n+s}(MSp(n/4))$}}\\
\hline
\end{tabular}
\end{table}

After above results let us prove theorem for $G=O(n)$. Let
$N\subset M$ be a $s$-dimensional closed submanifold of $M$, $\dim
M=n+s$. The normal bundle on $N$ defines the following commutative
diagram:
$$\xymatrix@C=70pt{M\ar[r]^f&MO(n)\\
N\ar@{^{(}->}[u]\ar[r]&BO(n)\ar@{^{(}->}[u]}$$ where the bottom mappping
is the classifying map. The complement of the neighborhood of $N$
in $M$ fully reduces to a point $\sigma^0$ by means of the
contraction of $\partial\tilde E$ in the construction of
$M(\xi)=MO(n)$. Then, we get $f^*\phi(1)\equiv f^*u=D[N]$, where
$\phi:H^0(BO(n);\mathbb{Z}_2)\to H^n(MO(n);\mathbb{Z}_2)$ (see Lemma
\ref{lemma-natural-isomorphisms-algebraic-topology-a}). So if $x\in H_s(M;\mathbb{Z}_2)$ is representable by a
submanifold $N\subset M$, then we should have $f^*\phi(1)=Dx$.
Conversely, if we assign a mapping $f:M\to MO(n)$ transversally
regular along $BO(n)\subset MO(n)$, the reciprocal image of
$N\equiv f^{-1}(BO(n))$ is such that $f^*u=D[N]$.\end{proof}

\begin{theorem} Any cycle $x\in H_n(M;\mathbb{Z})$, $\dim M=n+1$, can be realized
with a closed submanifold. Furthermore, any cycle $x\in H_n(M;\mathbb{Z})$, $n+1\le \dim M\le n+2$, can be realized with an orientable closed submanifold. If $s<n/2$, for any cycle $x\in H_s(M;\mathbb{Z})$, $\dim M=n$, there exists a $\lambda\not=0$ such that the
cycle $\lambda x$ is represented by an $s$-dimensional submanifold
$N\subset M$.\end{theorem}

\begin{theorem} Let $M$ be any finite cell complex. For any
cycle $x\in H_s(M;\mathbb{Z})$ there exists an $\lambda\not=0$ such
that $\lambda x$ is the image of an $s$-dimensional manifold $N$,
$\phi:N\to M$, $\phi_*[N]=\lambda x$. One has the natural
epimorphism
$${\underline{\Omega}}_s^{SO}(X,Y)\bigotimes\mathbb{Q}\to H_s(X,Y;\mathbb{Q}).$$
\end{theorem}

\begin{definition} An {\em $X$-structure} on a manifold $V$ is a
homotopy class of cross-sections of the bundle of geometric
objects with fiber $X$ over $V$. A {\em $X$-manifold} is a
manifold $V$ together with an $X$-structure on $V$.\end{definition}

\begin{proposition} If
$V$ is an $X$-manifold, then so is $\partial V$.\end{proposition}

\begin{definition} Given
any closed $X$-manifold $V$ one can define a second $X$-manifold
$-V$ such that $\partial(V\times I)\cong V\sqcup(-V)$. Thus one can
define a bordism group for the class of $n$-dimensional
$X$-manifolds, denoted by ${\underline{\Omega}}^X_n$ and called the
{\em $n$-th $X$-bordism group}.\end{definition}

Spectra are also related to
the bordism groups. In fact, one has the following.

\begin{theorem}[R.Thom] One has the following isomorphism:
${\underline{\Omega}}^X_\bullet\cong\pi_\bullet(MX)$, where $MX$ is
the spectrum {\em(Thom spectrum)} associated to the
$X$-structure.\end{theorem}

\begin{proof} See, e.g., \cite{SWITZER} and works quoted there.\end{proof}

\begin{example} In
particular, for $X=BO$, $SO$, $U$, $SU$ and $S_p$, we get
the bordism groups considered in Theorem \ref{realizations-cycles-theorems} and Lemma \ref{lemma-important-bordism-groups} and reported in Tab. \ref{bordism-groups-and-stable-homotopy-groups}.\end{example}

\begin{definition} A {\em singular $X$-manifold} in the space $Y$ is a continuous map
({\em singular simplex}) $f:M\to Y$, where $M$ is a closed
$X$-manifold. Two singular $X$-manifolds $(M,f)$, $(M',f')$ in $Y$
are called {\em $X$-bordant} if there is a pair $(W,g)$ such
that $W$ is a compact $X$-manifold with boundary, $\partial W\cong
M\sqcup(-M')$, and $g$ is a continuous map $g:W\to Y$ such that
$g|_M=f$, $g|_{M'}=f'$. The corresponding bordism group, for
$n$-dimensional $X$-manifolds contained in $Y$, is denoted by
${\underline{\Omega}}^X_n(Y)$.\end{definition}

\begin{theorem} One has the following
isomorphisms:
$${\underline{\Omega}}^X_n(Y)\cong MX_n(Y^+)\cong\pi_n(MX\wedge Y^+),$$
where $Y^+\equiv Y\cup\{\infty\}\equiv Y/\emptyset $.
One has the following isomorphism:
$${\underline{\Omega}}^X_n(Y)\cong E_n(Y^+),$$
where $E_\bullet(Y^+)$ is the homology induced by the spectrum $MX$.
\end{theorem}

\begin{proof} See, e.g., \cite{QUILLEN, SWITZER} and references quoted there.\end{proof}

\begin{example}[Framed cobordism and Pontrjagin-Thom construction]\label{framed-cobordism-pontrjagin-thom-construction}
Let $\Omega_\bullet^{fr}$ denote the graded ring of framed cobordism and let $\Omega_\bullet^{fr}(X)$ denote the graded ring of framed cobordism classes of maps $f:B\to X$, where $B$ are framed manifolds without boundary and $X$ a fixed topological space. One has the isomorphisms reported in {\em(\ref{isomorphisms-framed-cobordisms-stable-homotopy-groups})}.
\begin{equation}\label{isomorphisms-framed-cobordisms-stable-homotopy-groups}
    \left\{
    \begin{array}{l}
      \Omega_\bullet^{fr}\cong\pi^s_\bullet,\hskip 3pt(\pi_n^s\equiv\pi^s_n(S^0)=\mathop{\lim}\limits_{k}\pi_{n+k}(S^k,s_0))\\
      \Omega_\bullet^{fr}(X)\cong\pi^s_\bullet(X^+),\hskip 3pt(\pi_n^s(X^+)\equiv\mathop{\lim}\limits_{k}\pi_{n+k}(S^kX^+,\ast)).\\
\end{array}\right.
\end{equation}
Let $B$ has a stably trivial normal bundle, i.e., let $i:B\to \mathbb{R}^{n+r}$ be an embedding for enough large $r$ and its normal bundle $\nu(B,i)$ is trivial, $\nu(B,i)\cong B\times\mathbb{R}^{r}$. Then, there is a canonical homeomorphism $M(\nu(B,i))\thickapprox\Sigma^r(B^+)=S^r\wedge B^+$, between the Thom complex of $\nu(B,i)$ and the $r$-fold suspension of $B^+$, called {\em framing} of $\nu(B,i)$. The {\em Pontrjagin-Thom construction} is the mapping given in {\em(\ref{pontrjagin-thom-construction-a})}.
\begin{equation}\label{pontrjagin-thom-construction-a}
\xymatrix{S^{n+r}\equiv\mathbb{R}^{n+r}\bigcup\left\{\infty\right\}\equiv(\mathbb{R}^{n+r})^+\ar@/_1pc/[rr]_{\tau}\ar[r]&
M(\nu(B,i))\thickapprox\Sigma^r(B^+)=S^r\wedge B^+\ar[r]&S^r}
\end{equation}
The homotopy class of the map $\tau$ defines an element of $\pi_{n+r}(S^r)$. This construction induces an isomorphism of graded rings $\Omega^{fr}_{\bullet}\cong\pi^s_\bullet(S^0)$. This construction can be generalized to maps $f:B\to X$ obtaining the mapping given in {\em(\ref{pontrjagin-thom-construction-b})}.
\begin{equation}\label{pontrjagin-thom-construction-b}
\xymatrix{S^{n+r}\ar@/_1pc/[rr]_{\tau}\ar[r]&
M(\nu(B,i))\thickapprox\Sigma^r(B^+)\ar[r]&S^r(X^+)}
\end{equation}
Taking into account that there is a stable homotopy equivalence $X^+\approxeq X\vee S^0$ and a non-canonical isomorphism $\pi_\bullet^s(X^+)=\mathop{\lim}\limits_{\to}(\Sigma^r(X^+))\cong\pi_\bullet^s(X)\bigoplus\pi_\bullet^s(S^0)$, we get the other isomorphism in {\em(\ref{isomorphisms-framed-cobordisms-stable-homotopy-groups})}.\footnote{In Tab. \ref{spectra-and-generalized-co-homology-theories} are resumed relations between spectra, generalized (co)homologies, and some distinguished examples. Let us emphasize the relation with Brown's representable theorem. A functor $F:(\mathcal{W}^{\bullet'})^{op}\to \mathcal{S}_{et}$ is {\em representable}, i.e., $F$ is equivalent to $Hom_{\mathcal{W}^{\bullet'}}(-;C)$ for some CW-complex $C$, iff the following conditions are satisfied. (i)(Wedge axiom). $F(\vee_\alpha X_\alpha)\cong\prod_\alpha F(X_\alpha)$; (ii)(Mayer-Vietoris axiom). For any CW complex $W$ covered by two subcomplexes $U$ and $V$, and any elements $u\in F(U)$, $v\in F(V)$, such that $u$ and $v$ restrict to the same element of $F(U\bigcap V)$, there is an element $w\in F(W)$ restricting to $u$ and $v$, respectively. In the particular case of singular cohomology, one has $H^n(X;A)\cong Hom(X;K(A,n))$, i.e., the singular cohomology functor is representable. Thanks to extended versions of Brown's representable theorem one can prove that all homology theories come from spectra, or by considering multiplicative operations, all homology theories come from ring spectra with multiplication $\mu: E\wedge E\to E$ and the unity $\eta:E(S^0)\to E$.}
\end{example}

\begin{table}
\caption{Spectra $E=\{E_n\}$ and generalized (co)homology theories.}
\label{spectra-and-generalized-co-homology-theories}
\scalebox{0.8}{$\begin{tabular}{|l|l|l|}
\hline
{\rm{\footnotesize Name}}&{\rm{\footnotesize Isomorphism}}&{\rm{\footnotesize Spectra}}\\
\hline
{\rm{\footnotesize generalized homology}}&{\rm{\footnotesize $E_n(X)=\pi_n(E\wedge X)=[\Sigma^nS^0,E\wedge X]$}}&{\rm{\footnotesize $E$}}\\
{\rm{\footnotesize $E_\bullet(X)$}}&&\\
\hline
{\rm{\footnotesize generalized cohomology}}&{\rm{\footnotesize $E^n(X)=[(E(X),\Sigma^nE]\cong[\Sigma^{-n}S^2\wedge X,E]$}}&{\rm{\footnotesize $E$}}\\
{\rm{\footnotesize $E^\bullet(X)$}}&&\\
\hline
{\rm{\footnotesize stable homotopy}}&{\rm{\footnotesize $\pi^s_n(X)=\mathop{\lim}\limits_{k}\pi_{n+k}(S^kX,\ast)=\pi_n(S^0\wedge X)=S^0_n(X)$}}&{\rm{\footnotesize sphere spectrum}}\\
{\rm{\footnotesize $\pi_\bullet^s(X)$}}&&{\rm{\footnotesize $S^0\equiv E(S^0)$}}\\
\hline
{\rm{\footnotesize stable cohomotopy}}&{\rm{\footnotesize $\pi_s^n(X)=\mathop{\lim}\limits_{k}\pi_{n+k}(S^k\wedge X)=(S^0)^\bullet(X)$}}&{\rm{\footnotesize sphere spectrum}}\\
{\rm{\footnotesize $\pi^\bullet_s(X)$}}&&{\rm{\footnotesize $S^0\equiv E(S^0)$}}\\
\hline
{\rm{\footnotesize singular cohomology}}&{\rm{\footnotesize $H^n(X;A)\cong[X;K(A,n)]$}}&{\rm{\footnotesize Eilenberg-MacLane}}\\
{\rm{\footnotesize $H^n(X;A)$}}&&{\rm{\footnotesize $K(A,n)$}}\\
\hline
{\rm{\footnotesize topological K-theory}}&{\rm{\footnotesize $K^0(X)$}}&{\rm{\footnotesize $E_0=\mathbb{Z}\times BU$ }}\\
{\rm{\footnotesize $K^n(X)$}}&{\rm{\footnotesize $K^1(X)$}}&{\rm{\footnotesize $E_1=U$}}\\
&{\rm{\footnotesize $K^{2n}(X)\cong K^0(X)$}}&{\rm{\footnotesize $E_{2n}=\mathbb{Z}\times BU$}}\\
&{\rm{\footnotesize $K^{2n+1}(X)\cong K^1(X)$}}&{\rm{\footnotesize $E_{2n+1}=U$}}\\
\hline
{\rm{\footnotesize $X$-bordism}}&{\rm{\footnotesize $\underline{\Omega}_\bullet^X\cong\pi_\bullet(MX)$}}&{\rm{\footnotesize Thom spectrum}}\\
{\rm{\footnotesize $\underline{\Omega}_\bullet^X$}}&&{\rm{\footnotesize $MX$}}\\
\hline
{\rm{\footnotesize singular $X$-bordism}}&{\rm{\footnotesize $\underline{\Omega}_n^X(Y)\cong\pi_n(MX\wedge Y^+)\cong MX_n(Y^+)$}}&{\rm{\footnotesize Thom spectrum}}\\
{\rm{\footnotesize $\underline{\Omega}_\bullet^X(Y)$}}&&{\rm{\footnotesize $MX$}}\\
\hline
{\rm{\footnotesize framed bordism}}&{\rm{\footnotesize $\Omega^{fr}_n=\pi^s_n$}}&{\rm{\footnotesize sphere spectrum}}\\
{\rm{\footnotesize $\Omega^{fr}_\bullet$}}&&\\
\hline
{\rm{\footnotesize singular framed bordism}}&{\rm{\footnotesize $\Omega^{fr}_n(X)=\pi^s_n(X^+)$}}&{\rm{\footnotesize sphere spectrum}}\\
{\rm{\footnotesize $\Omega^{fr}_\bullet(X)$}}&&\\
\hline
\multicolumn{3}{l}{\rm{\footnotesize $A$=abelian group.}}\\
\multicolumn{3}{l}{\rm{\footnotesize $U$=infinite unitary group and $BU$ its classifying space.}}\\
\multicolumn{3}{l}{\rm{\footnotesize $K^0(X)$=Groethendieck group of complex vector bundles over $X$.}}\\
\multicolumn{3}{l}{\rm{\footnotesize $K^1(X)$=Groethendieck group of vector bundles over $SX$.}}\\
\multicolumn{3}{l}{\rm{\footnotesize coefficient groups of generalized homology theories: $E_n(S^0)=\pi_n(E\wedge S^0)=\pi_n(E)$.}}\\
\multicolumn{3}{l}{\rm{\footnotesize coefficient groups of generalized cohomology theories:}}\\
\multicolumn{3}{l}{\rm{\footnotesize $E^n(S^0)=[E(S^0),\Sigma^nE]\cong[S^0,\Sigma^nE]\cong[\Sigma^{-n }S^0,E]=\pi_{-n}(E)$.}}\\
\end{tabular}$}
\end{table}

\begin{theorem}[Steenrod algebra and Stiefel-Whitney classes]\label{steenrod-algebra-and-stiefel-whitney-classes}
There exists a graded algebra $\mathcal{A}^\bullet\equiv\mathcal{A}^\bullet(\mathbb{F}_p)$ ({\em Steenrod algebra}) such that $H^{\bullet}(X;\mathbb{Z}_p)$ has a natural structure of graded module over $\mathcal{A}^\bullet$.\footnote{These module structures $\mathcal{A}^\bullet\times H^\bullet(X;\mathbb{Z}_p)\to H^\bullet(X;\mathbb{Z}_p)$, allow us to understand that there are strong constraints just on the space $X$ in order to obtain cohomology spaces $H^\bullet(X;\mathbb{Z}_p)$ with a prefixed structure. For example, do not exist spaces $X$ with $H^\bullet(X;\mathbb{Z}_p)=\mathbb{Z}[\alpha]$, unless $\alpha$ has dimension $2$ or $4$, where there examples $\mathbb{C}P^\infty$ and $\mathbb{H}P^\infty$.}

In particular, $\mathcal{A}^\bullet$ over the prime field $\mathbb{F}_p$ has the interpretation as $\mathbb{Z}_p$-cohomology of the Eilenberg-MacLane spectrum $K(\mathbb{F}_p)$. By the base change $\mathcal{A}^\bullet\bigotimes_{\mathbb{F}_p}\mathbb{F}_q$  can be considered the $\mathbb{F}_q$-cohomology of the Eilenberg-MacLane spectrum $K(\mathbb{F}_q)$. By including $K(\mathbb{F}_p)$ into $K(\mathbb{F}_q)$ we can view the elements of $\mathcal{A}^\bullet\bigotimes_{\mathbb{F}_p}\mathbb{F}_q$ as defining stable cohomology operations in $\mathbb{F}_q$-cohomology. This allows us to interpret elements of ${\mathbb S}^{\bullet}(\mathbb{F}_q)$ as stable cohomology operations acting on the $\mathbb{F}_p$-cohomology of a topological space.\footnote{A {\em cohomology operations} is a natural transformation between cohomology functors. One says that a cohomology operation is {\em stable} if it commutes with the suspension functor $S$. For example the {\em cup product squaring operator} $H^n(X;R)\to H^{2n}(X;R)$, $x\mapsto x\cup x$, where $R$ is a ring and $X$ a topological space, is an instable cohomology operation. Instead, are stable the following Steenrod operations: $Sq^i:H^n(X;\mathbb{Z}_2)\to H^{n+i}(X;\mathbb{Z}_2)$ and $P^i:H^n(X;\mathbb{Z}_2)\to H^{n+2i(p-1)}(X;\mathbb{Z}_2)$. (In Tab. \ref{properties-steenrod-squares} are resumed some fundamental properties of $Sq^i$.)}

Furthermore, with respect this structure on $H^{\bullet}(X;\mathbb{Z}_2)$, the Stiefel-Whitney classes are generated by $w_{2^i}$.
\end{theorem}

\begin{table}[h]
\caption{Properties of the Steenrod squares $Sq^i:H^n(X;\mathbb{Z}_2)\mapsto H^{n+i}(X;\mathbb{Z}_2)$.}
\label{properties-steenrod-squares}
\begin{tabular}{|l|l|}
  \hline
  \hfil{\rm{\footnotesize Name}}\hfil& \hfil{\rm{\footnotesize Properties}}\hfil\\
\hline
{\rm{\footnotesize naturality}}& {\rm{\footnotesize $f^\bullet(Sq^i(x))=Sq^i(f^\bullet(x)),\hskip 3pt f:X\to Y$}}\\
\hline
{\rm{\footnotesize additivity}}& {\rm{\footnotesize $Sq^i(x+y)=Sq^i(x)+Sq^i(y)$}}\\
\hline
{\rm{\footnotesize Cartan formula}}& {\rm{\footnotesize $Sq^i(x\cup y)=\sum_{r+s=i}(Sq^r(x))\cup(Sq^s(y))$}}\\
\hline
{\rm{\footnotesize stability}}& {\rm{\footnotesize $S\circ Sq^i-Sq^i\circ S=0$}}\\
\hline
{\rm{\footnotesize cup square}}& {\rm{\footnotesize $Sq^i(x)=x\cup x,\hskip 3pt \DEG(x)=i$}}\\
\hline
{\rm{\footnotesize $Sq^0$}}& {\rm{\footnotesize $Sq^0=1$}}\\
\hline
{\rm{\footnotesize $Sq^1$}}& {\rm{\footnotesize $Sq^1$= Bockstein homomorphism of the exact sequence}}\\
& {\rm{\footnotesize $\xymatrix{0\ar[r]&\mathbb{Z}_2\ar[r]&\mathbb{Z}_4\ar[r]&\mathbb{Z}_2\ar[r]&0}$}}\\
\hline
{\rm{\footnotesize Serre-Cartan basis}}& {\rm{\footnotesize $\{Sq^I\equiv Sq^{i_1}\cdots Sq^{i_k}\}_{i_j\ge 2i_{j+1}}$}}\\
\hline
{\rm{\footnotesize Adem's relations}}& {\rm{\footnotesize $\{Sq^iSq^j-I^{ij}=0\}_{i,j>0,\hskip 2pt i<2j}$}}\\
& {\rm{\footnotesize $I^{ij}\equiv \sum_{0\le k\le[i/2]}\binom {j-k-1} {i-2k} Sq^{i+1}Sq^{k}$.}}\\
\hline
 \multicolumn{2}{l}{\rm{\footnotesize $\mathcal{A}^\bullet(\mathbb{F}_2)=<Sq^i>/<Sq^iSq^j-I^{ij}>$.}}\\
\multicolumn{2}{l}{\rm{\footnotesize For any $p\ge 2$, $\mathcal{A}^\bullet(\mathbb{F}_p)$ is generated by $P^i$ and  the Bockstein operator $\beta$ associated to}}\\
\multicolumn{2}{l}{\rm{\footnotesize the short exact sequence $\scalebox{0.8}{$\xymatrix{0\ar[r]&\mathbb{Z}_p\ar[r]&\mathbb{Z}_{p^2}\ar[r]&\mathbb{Z}_p\ar[r]&0}$}$.}}\\
\end{tabular}
\end{table}

\begin{proof}
Let $V$ be a $n$-dimensional $\mathbb{F}_q$-vector space over the Galois field $\mathbb{F}_q$ of size $q=p^k$, with prime $p$ and positive integer $k\in\mathbb{N}_0$.\footnote{A {\em Galois field} (or {\em finite field}) is a field that contains only finitely many elements. These are classified by $q=p^k$ if they contains $q$ elements. Each Galois field with $q$ elements is the splitting field of the polynomial $x^q-x$. Recall that the {\em splitting field} of a polynomial $p(x)$ over a field $K$ is a field extension $L$ of $K$ over which $p(x)$ factorizes into linear factors $x-a_i$, and such that $a_i$ generates $L$ over $K$, i.e, $L=K(a_i)$. (In Tab. \ref{properties-field-extension} are reported some properties of field extensions useful in the paper.) The extension $L$ of minimal degree over $K$ in which $p$ splits exists and is unique up to isomorphism, identified by the {\em Galois group} of $p$. (In Tab. \ref{examples-galois-group} are reported some fundamental properties and examples of Galois groups.) } Let us consider the controvariant functor $\mathbb{F}_q[-]$, identified  by the correspondence $V \rightsquigarrow \mathbb{F}_q[V]\cong S^\bullet(V^*)$, where $S^\bullet(V^*)$ is the graded commutative symmetric algebra on the dual space $V^*$ of $V$. Let us define the $\mathbb{F}_q$-algebra homomorphism $P(\xi):\mathbb{F}_q\to\mathbb{F}_q[V][[\xi]]$, with the formula
$P(\xi)(\alpha)=\alpha+\alpha^q\xi\in \mathbb{F}_q[V][[\xi]]$, $\forall \alpha\in V^*$. Then we get the formulas in (\ref{definitions-steenrod-operations}).
\begin{equation}\label{definitions-steenrod-operations}
    P(\xi)(f)=\left\{\begin{array}{l}
                    \sum_{0\le i\le \infty}P^i(f)\xi^i,\hskip 3pt q\not=2\\
                    \sum_{0\le i\le \infty}Sq^i(f)\xi^i,\hskip 3pt q=2\\
                     \end{array}\right\}\hskip 3pt \forall\:  f\in\mathbb{F}_q[V].
\end{equation}
Equation (\ref{definitions-steenrod-operations}) defines the $\mathbb{F}_q$-linear maps $P^i,Sq^i:\mathbb{F}_q[V]\to\mathbb{F}_q[V]$. $P^i$ are called {\em Steenrod reduced power operations} and $Sq^i$ are called {\em Steenrod squaring operations}. For abuse of notation can be all denoted by $P^i$ and called {\em Steenrod operations}. These operations satisfy the conditions ({\em unstability conditions}), reported in (\ref{unstability-conditions-steenrod-operations}).
\begin{equation}\label{unstability-conditions-steenrod-operations}
    P^i(f)=\left\{\begin{array}{l}
                    f^q,\hskip 3pt i=\DEG(f)\\
                    0,\hskip 3pt i>\DEG(f)\\
                     \end{array}\right\}\hskip 3pt \forall\:  f\in\mathbb{F}_q[V], \hskip 3pt i,j,k\in\mathbb{N}_0.
\end{equation}
Moreover, one has the derivation {\em Cartan formulas} reported in (\ref{derivation-cartan-formulas}).
\begin{equation}\label{derivation-cartan-formulas}
    P^k(fg)=\sum_{i+j=k}P^i(f)P^j(g),\hskip 3pt f,g\in\mathbb{F}_q[V].
\end{equation}
Furthermore, one has the relations ({\em Adem-Wu relations} \cite{WU, ADEM, HCARTAN, SERRE}) reported in (\ref{adem-wu-relations}).
\begin{equation}\label{adem-wu-relations}
    P^iP^j=\sum_{0\le k\le[\frac{i}{q}]}(-1)^{(i-qk}\binom {(q-1)(j-k)-1} {i-qk}P^{i+j-k}P^k,\hskip 3pt \forall\:  i,j\ge 0, \: i<qj.
\end{equation}

For any Galois field $\mathbb{F}_q$ the coefficients are in the prime subfield $\mathbb{F}_p\vartriangleleft\mathbb{F}_q$.

Then the {\em Steenrod algebra} is the free associative $\mathbb{F}_q$-algebra generated by the reduced power operations $Pi$, modulo the Adem-Wu relations. The admissible monomials are an $\mathbb{F}_q$-basis for the Steenrod algebra.

The Steenrod algebra has a natural structure of Hopf algebra \cite{MILNOR2, MILNOR-MOORE, SMITH}.\footnote{$\mathcal{A}^\bullet$ has a natural structure of Hopf algebra with commutative, associative comultiplication $\psi:\mathcal{A}^\bullet\to\mathcal{A}^\bullet\otimes\mathcal{A}^\bullet$, given by $\psi(\mathcal{A}^k)=\sum_{i+j=k}\mathcal{A}^i\otimes\mathcal{A}^j$. Let us denote by $\mathcal{A}_\bullet\equiv Hom_{\mathbb{_q}}(\mathcal{A}^\bullet;\mathbb{F}_q)=\oplus_{n}\mathcal{A}_n=\oplus_{n} Hom_{\mathbb{F}_q}(\mathcal{A}^n;\mathbb{F}_q)$ the dual vector space to $\mathcal{A}^\bullet$. One has the canonical evaluation pairing $<,>:\mathcal{A}^\bullet\times\mathcal{A}_\bullet\to\mathbb{F}_q$, $<f,\alpha>=\alpha(f)$. One has the following isomorphism of $\mathbb{F}_q$-Hopf algebras $\mathcal{A}_\bullet\cong\mathbb{F}_q[\xi_1,\xi_2,\xi_3,\cdots,\xi_k,\cdots]$, where $\deg(\xi_k)=q^{k-1}$, $k\in\mathbb{N}$, and comultiplication $\phi:\mathcal{A}_\bullet\to\mathcal{A}_\bullet\otimes\mathcal{A}_\bullet$, given by $\phi(\xi_k)=\sum_{i+j=k}\xi_i^{q^j}\otimes\xi_j$.}

Set $H(V)\equiv \mathbb{F}[V]\bigotimes\Lambda^\bullet(V^*)$. One has two embeddings of $V^*$ into $H(V)$, given in (\ref{embeddings-dual-space-graded-space}).

\begin{equation}\label{embeddings-dual-space-graded-space}
\begin{array}{ll}
\scalebox{0.7}{$\xymatrix{&V^*\bigotimes\mathbb{F}\subset\mathbb{F}[V]\bigotimes_{\mathbb{F}}\mathbb{F}\cong\mathbb{F}[V]\subset H(V)\\
    V^*\ar[ur]^(0.4){a}\ar[dr]_(0.4){b}&\\
    &\mathbb{F}\bigotimes_{\mathbb{F}}V^*\cong V^*\subset\Lambda^\bullet(V^*)}$}&\scalebox{0.7}{$
    \xymatrix{&a(z)\equiv z\in V^*\subset\mathbb{F}[V]\\
    z\mapsto \ar[ur]^(0.4){a}\ar[dr]_(0.4){b}&\\
    &b(z)\equiv dz\in V^*\subset\Lambda^\bullet(V^*)}$}\\
                     \end{array}
                     \end{equation}
Let $\beta:H(V)\to H(V)$ be the unique derivation with the property that for an alternating linear form $dz$ one has $\beta(dz)=z$, and for any polynomial linear form $z$, one has $\beta(z)=0$. This derivation is called {\em Bockstein operator}.\footnote{The (co)homological interpretation of the Bockstein operator is associated to a short exact sequence $\scalebox{0.7}{$\xymatrix{0\ar[r]&A_\bullet\ar[r]^{\alpha}&B_\bullet\ar[r]^{\beta}&C_\bullet\ar[r]&0}$}$ of chain complexes in an abelian category. In fact to such a sequence there corresponds a long exact sequence $\scalebox{0.7}{$\xymatrix{\cdots\ar[r]&H_{n+1}(A_\bullet)\ar[r]^{\alpha_*}&H_{n+1}(B_\bullet)\ar[r]^{\beta_*}&
H_{n+1}(C_\bullet)\ar[r]^{\delta_{n+1}}&H_{n}(A_\bullet)\ar[r]^{\alpha_*}&H_n(B_\bullet)\ar[r]^{\beta_*}&
H_n(C_\bullet)\ar[r]^{\delta_n}&\cdots}$}$. The boundary maps $\delta_{n+1}:H_{n+1}(C_\bullet)\to H_n(A_\bullet)$ are just the {\em Bockstein homomorphisms}. In particular, if $\scalebox{0.7}{$\xymatrix{0\ar[r]&A\ar[r]^{\alpha}&B\ar[r]^{\beta}&C\ar[r]&0}$}$ is a short exact sequence of abelian groups and $A_\bullet=E_\bullet\otimes A$, $B_\bullet=E_\bullet\otimes B$, $C_\bullet=E_\bullet\otimes C$, with $E_\bullet$ a chain complex of free, or at least torsion free, abelian groups, then the Bockstein homomorphisms are induced by the corresponding short exact sequence $\scalebox{0.7}{$\xymatrix{0\ar[r]&E_\bullet\otimes A\ar[r]^{1\otimes\alpha}&E_\bullet\otimes B\ar[r]^{1\otimes\beta}&E_\bullet\otimes C\ar[r]&0}$}$. Similar considerations hold for cochain complexes. in such cases the Bockstein homomorphism increases the degree, i.e., $\beta:H^n(C^\bullet)\to H^{n+i}(A^\bullet)$.} Then the {\em full Steenrod algebra}, $\mathcal{A}^\bullet(\mathbb{F}_q)$, of the Galois field $\mathbb{F}_q$ is generated by $P^i$, $i\in\mathbb{N}_0$, and the Bockstein operator $\beta$. This a subalgebra of the algebra of endomorphisms of the functor $V\rightsquigarrow H(V)$.

Then the relation between Stiefel-Whitney classes and Steenrod squares is given by the relation ({\em Wu's relation}) reported in (\ref{wu-relation}).
\begin{equation}\label{wu-relation}
    Sq(\nu)=w\hskip 3pt\left\{
    \begin{array}{l}
      Sq^k(x)=\nu_k\cup x \\
      <Sq^k(x),\mu>=<\nu_k\cup x,\mu> \\
    \end{array}
    \right\}.
\end{equation}
This means that the total Stiefel-Whitney class $w$ is the Steenrod square of the total {\em Wu class} $\nu$ that is implicitly defined by the relation (\ref{wu-relation}). The natural short exact sequence $\xymatrix{\mathbb{Z}\ar[r]&\mathbb{Z}_2\ar[r]&0}$ induces the Bockstein homomorphism $\beta:H^i(X;\mathbb{Z}_2)\to H^{i+1}(X;\mathbb{Z})$. $\beta(w_i)\in  H^{i+1}(X;\mathbb{Z})$ is called the {$(i+1)$-integral Stiefel-Whitney class}.\footnote{The third integral Stiefel-Whitney class is the obstruction to a $spin^c$-structure on $X$.}  Thus, over the Steenrod algebra, the Stiefel-Whitney classes $w_{2^i}$ generate all the Stiefel-Whitney classes and satisfy the formula ({\em Wu's formula}) reported in (\ref{wu-formula}).
\begin{equation}\label{wu-formula}
    Sq^i(w_j)=\sum_{0\le k\le i}\binom {j+k-i-1} {k}w_{i-k}w_{j+k}.
\end{equation}

\end{proof}

\begin{table}[h]
\caption{Properties of field extension $L/K$.}
\label{properties-field-extension}
\scalebox{0.8}{$\begin{tabular}{|l|l|}
  \hline
  \hfil{\rm{\footnotesize Name}}\hfil& \hfil{\rm{\footnotesize Properties}}\hfil\\
\hline
{\rm{\footnotesize intermediate of $L/K$}}& {\rm{\footnotesize any extension $L/H$}}\\
& {\rm{\footnotesize such that $H/K$ is an extension field}}\\
\hline
{\rm{\footnotesize adjunction of subset $S\subset L$}}& {\rm{\footnotesize $K(S)$= smallest subfield containing $K$ and $S$.}}\\
\hline
{\rm{\footnotesize simple extension}}& {\rm{\footnotesize $L=K(\{s\})$, $s\in L$, $s$=primitive element}}\\
\hline
{\rm{\footnotesize degree of the extension}}& {\rm{\footnotesize $[L:K]=\dim_K(L)$}}\\
\hline
{\rm{\footnotesize trivial extension}}& {\rm{\footnotesize $[L:K]=1$, i.e., $L=K$}}\\
\hline
{\rm{\footnotesize quadratic (cubic) extension}}& {\rm{\footnotesize $[L:K]=2$, ($[L:K]=3$)}}\\
\hline
{\rm{\footnotesize finite (infinite) extension}}& {\rm{\footnotesize $[L:K]<\infty$, ($[L:K]=\infty$)}}\\
\hline
{\rm{\footnotesize Galois extension}}& {\rm{\footnotesize $L/K$ such that:}}\\
& {\rm{\footnotesize (a) (normality): $L$ is the splitting field}}\\
& {\rm{\footnotesize  of a family of polynomials in $K[x]$;}}\\
& {\rm{\footnotesize (b) (separability):}}\\
& {\rm{\footnotesize For every $\alpha\in L$, the minimal polynomial of $\alpha$ in $K$}}\\
& {\rm{\footnotesize is a sparable polynomial, i.e., has distinct roots.}}\\
\hline
{\rm{\footnotesize $[\mathbb{C}:\mathbb{R}]=2$}}& {\rm{\footnotesize This is a simple, Galois extension:}}\\
& {\rm{\footnotesize $\mathbb{C}=\mathbb{R}(i)$; $[\mathbb{C}:\mathbb{R}]=|Aut(\mathbb{C}/\mathbb{R})|=2$.}}\\
& {\rm{\footnotesize $\mathbb{C}/\mathbb{R}\cong \mathbb{R}[x]/(x^2+1)$.}}\\
\hline
{\rm{\footnotesize $[\mathbb{R}:\mathbb{Q}]=\mathfrak{c}$}}& {\rm{\footnotesize  This is an infinite extension.}}\\
& {\rm{\footnotesize  $\mathfrak{c}$=cardinality of the continuum.}}\\
\hline
{\rm{\footnotesize ($\clubsuit$): $H/\mathbb{Q}$}}& {\rm{\footnotesize  Splitting field of $p(x)=x^3-2$ over $\mathbb{Q}$.}}\\
{\rm{\footnotesize $H=\mathbb{Q}(a_1,a_2)\subset \mathbb{C}$}}& {\rm{\footnotesize  }}\\
{\rm{\footnotesize $\{a_1=2^{1/3}\in \sqrt[3]{2},a_2=-\frac{1}{2}+i\frac{\sqrt{3}}{2}\in\sqrt[3]{1}\}$}}& {\rm{\footnotesize  }}\\
\hline
\multicolumn{2}{l}{\rm{\footnotesize Artin's theorem Galois extension: For a finite extension $L/K$}}\\
\multicolumn{2}{l}{\rm{\footnotesize the following statements are equivalent.}}\\
\multicolumn{2}{l}{\rm{\footnotesize (i). $L/K$ is a Galois extension.}}\\
\multicolumn{2}{l}{\rm{\footnotesize (ii). $L/K$ is a normal extension and a separable extension.}}\\
\multicolumn{2}{l}{\rm{\footnotesize (iii). $L$ is the splitting field of a separable polynomial with coefficients in $K$.}}\\
\multicolumn{2}{l}{\rm{\footnotesize (iv). $[L:K]=|Aut(L/K)|$=order of $Aut(L/K)$.}}\\
\end{tabular}$}
\end{table}

\begin{table}[h]
\caption{Examples of Galois group of extension field $L/K$: $Gal(L/K)\equiv Aut(L/K)=\{\alpha\in Aut(L)\: |\: \alpha(x)=x,\: \forall x\in K\}$ ($\star$).}
\label{examples-galois-group}
\begin{tabular}{|l|l|}
  \hline
  \hfil{\rm{\footnotesize Examples}}\hfil& \hfil{\rm{\footnotesize Remarks}}\hfil\\
\hline
{\rm{\footnotesize $Gal(L/L)=\{1\}$}}& {\rm{\footnotesize }}\\
\hline
{\rm{\footnotesize $Gal(\mathbb{C}/\mathbb{R})=\{1,i\}$}}& {\rm{\footnotesize }}\\
\hline
{\rm{\footnotesize $Aut(\mathbb{R}/\mathbb{Q})=\{1\}$}}& {\rm{\footnotesize }}\\
\hline
{\rm{\footnotesize $Aut(\mathbb{C}/\mathbb{Q})$}}& {\rm{\footnotesize infinite group}}\\
\hline
{\rm{\footnotesize Galois group of polynomial $p(x)=x^3-2$ \ref{properties-field-extension}($\clubsuit$)}}& {\rm{\footnotesize   $Gal(p(x))=\{1,f,f^2,g,gf,gf^2\}$}}\\
& {\rm{\footnotesize  $f,g\in Aut(H)$}}\\
&{\rm{\footnotesize  $f(a_1)=a_1a_2$, $f(a_2)=a_2$,}}\\
&{\rm{\footnotesize   $g(a_1)=a_1$, $g(a_2)=a_2^2$.}}\\
\hline
\multicolumn{2}{l}{\rm{\footnotesize ($\star$) $Gal(L/K)$ does not necessitate to be an abelian group.}}\\
\multicolumn{2}{l}{\rm{\footnotesize Fundamental theorem Galois theory: Let $L/K$ be a finite and Galois field extension.}}\\
\multicolumn{2}{l}{\rm{\footnotesize Then there are bijective correspondences between }}\\
\multicolumn{2}{l}{\rm{\footnotesize its intermediate fields H and subgroups of its Galois group.}}\\
\multicolumn{2}{l}{\rm{\footnotesize For any subgroup $G_H\lhd Gal(L/K) \rightsquigarrow H=\{x\in L\: |\: \alpha(x)=x,\: \forall \alpha\in G_H\}\lhd L$.}}\\
\multicolumn{2}{l}{\rm{\footnotesize For any intermediate field $H$ of $L/K$,}}\\
\multicolumn{2}{l}{\rm{\footnotesize $H \rightsquigarrow G_H=\{\alpha\in Gal(L/K)\: | \: \alpha(x)=x,\: \forall x\in H\}\lhd Gal(L/K)$.}}\\
\multicolumn{2}{l}{\rm{\footnotesize In particular $L\: \rightsquigarrow Gal(L/K)$ and  $K\: \rightsquigarrow Gal(L/K)$.}}\\
\end{tabular}
\end{table}

\begin{table}[h]
\caption{Whitney-Stiefel classes $w(E)\in H^\bullet(X;\mathbb{Z}_2)$ properties.}
\label{}
\begin{tabular}{|l|l|}
  \hline
  \hfil{\rm{\footnotesize Name}}\hfil& \hfil{\rm{\footnotesize Properties}}\hfil\\
\hline
{\rm{\footnotesize naturality}}& {\rm{\footnotesize $w(f^\star E)=f^*w(E),\hskip 3pt f:Y\to X$}}\\
\hline
{\rm{\footnotesize zero-degree}}& {\rm{\footnotesize $w_0(E)=1\in H^0(X;\mathbb{Z}_2)=\mathbb{Z}_2$}}\\
\hline
{\rm{\footnotesize normalization}}& {\rm{\footnotesize $w_1(\gamma)=1\in\mathbb{Z}_2=H^1(\mathbb{R}P^1;\mathbb{Z}_2),\hskip 3pt \gamma$=canonical line bundle}}\\
\hline
{\rm{\footnotesize Whitney addition formula}}& {\rm{\footnotesize $w(E\oplus F)=w(E)\cup w(F)$}}\\
\hline
{\rm{\footnotesize Linearly independent}}& {\rm{\footnotesize iff $w_{n-r+1}(E)=\cdots=w_{n}(E)=0$}}\\
{\rm{\footnotesize sections $s_1,\cdots,s_r$}}&\\
\hline
{\rm{\footnotesize orientable bundle}}& {\rm{\footnotesize iff $w_1(E)=0$}}\\
{\rm{\footnotesize orientable manifold $X$}}& {\rm{\footnotesize iff $w_1(TX)=0$}}\\
\hline
{\rm{\footnotesize $spin$ structure on $E$}}& {\rm{\footnotesize iff $w_1(E)=w_2(E)=0$}}\\
{\rm{\footnotesize $spin$ structure on $X$}}& {\rm{\footnotesize iff $w_1(TX)=w_2(TX)=0$}}\\
\hline
{\rm{\footnotesize $spin^c$ structure on $X$}}& {\rm{\footnotesize iff $w_1(TX)=0$ and $w_2$ belongs to the image}}\\
& {\rm{\footnotesize $H^2(X;\mathbb{Z})\to H^2(X;\mathbb{Z}_2)$}}\\
\hline
{\rm{\footnotesize $X=\partial Y$}}& {\rm{\footnotesize iff $<w,[X]>=0$}}\\
\hline
\multicolumn{2}{l}{\rm{\footnotesize $w:[X;Gr_{n}]\cong\mathbb{V}_n(X)\to H^\bullet(X;\mathbb{Z}_2)$}}\\
\multicolumn{2}{l}{\rm{\footnotesize $Gr_n\equiv Gr_n(\mathbb{R}^\infty)$, $\mathbb{V}_n(X)$= set of real $n$-vector bundles over $X$.}}\\
\end{tabular}
\end{table}

\section{SPECTRA IN PDE's}\label{spectra-pdes-section}

In this section we give an explicit relation between integral
bordism groups for admissible integral manifolds of PDE's bording by
means of smooth solutions, singular solutions and weak solutions
respectively. In particular we shall relate such integral bordism groups with suitable spectra.
Analogous relations for the corresponding Hopf
algebras of PDE's, are considered. Then important spectral sequences, useful to characterize conservation laws and (co)homological properties of PDE's, are related to their integral bordism groups.\footnote{Let us also emphasize that we can recognize webs on PDE's, by looking inside the geometric structure
of PDE's. By means of such webs, we can solve (lower
dimensional) Cauchy problems. This is important in order to decide
about the ''admissibility '' of integral manifolds in integral
bordism problems. However these aspects are not explicitly considered in this paper. They are studied in some details in other previous works about the PDE's algebraic topology by A. Pr\'astaro \cite{AG-PRA1}. For complementary informations on geometry of PDE's, see, e.g., Refs. \cite{AG-PRA2, B-C-G-G-G, CARTAN, HCARTAN, GOLD1, GOLD2, GROMOV, LYCH-PRAS, PRA00, PRA01, PRA1, PRA3, PRA4, PRA5, PRA6}.}

\begin{remark}
Let us shortly recall some definitions about integral bordism groups
in PDE's as just considered in some companion previous works by Pr\'astaro. Let
$\pi:W\to M$ be a smooth fiber bundle between smooth manifolds of
dimension $m+n$ and $n$ respectively. Let us denote by $J^k_n(W)$
the $k$-jet space for $n$-dimensional submanifolds of $W$. Let
$E_k\subset J^k_n(W)$ be a partial differential equation (PDE). Let
$N_i\subset E_k$, $i=1,2$, be two $(n-1)$-dimensional compact closed
admissible integral manifolds. Then, we say that they are $E_k$-{\em
bordant} if there exists a solution $V\subset E_k,$ such that
$\partial V=N_1\sqcup N_2$ (where $\sqcup$ denotes disjoint union). We
write $N_1\sim_{ E_k }N_2$. The empty set\hskip 3pt $\varnothing$ will be
regarded as a $p$-dimensional compact closed admissible integral
manifold for all $p\ge 0$. $\sim_{ E_k }$ is an equivalence
relation. We will denote by $\Omega^{ E_k }_{n-1}$ the set of all $
E_k $-bordism classes $[N]_{ E_k }$ of $(n-1)$-dimensional compact
closed admissible integral submanifolds of $ E_k $. The operation of
taking disjoint union defines a sum $+$ on  $\Omega^{ E_k }_{n-1}$
such that it becomes an Abelian group. We call $\Omega^{ E_k
}_{n-1}$ the {\em integral bordism group} of $ E_k $.  A {\em
quantum bord} of $ E_k $ is a solution $V\subset J^k_n(W)$ such that
$\partial V$ is a $(n-1)$-dimensional compact admissible integral
manifold of $ E_k$. The quantum bordism is an equivalence relation.
The set of quantum bordism classes is denoted by $\Omega_{n-1}( E_k
)$.\footnote{In other words the quantum bordism group of $E_k$ is
the integral bordism group of $J^k_n(W)$ relative to $E_k$. (This
language reproduces one in algebraic topology for couples $(X,Y)$ of
differentiable manifolds, where $Y\subset X$.)} The operation of
disjoint union makes $\Omega_{n-1}( E_k )$ into an Abelian group. We
call $\Omega_{n-1}( E_k )$ the {\em quantum bordism group} of $ E_k
$. Similar definitions can be made for any $0\le p<n-1$. For an
''admissible'' $p$-dimensional, $p\in\{0,\cdots,n-1\}$, integral
manifold $N\subset E_k\subset J^k_n(W)$ we mean a $p$-dimensional
smooth submanifold of $E_k$, contained into a solution $V\subset
E_k$, that can be deformed into $V$, in such a way that the deformed
manifold $\tilde N$ is diffeomorphic to its projection $\tilde
X\equiv\pi_{k,0}(\tilde N)\subset W$. In such a case $\tilde
X{}^{(k)}=\tilde N$. Note that the $k$-prolongation, $X^{(k)}$, of a
$p$-dimensional submanifold $X\subset Y$, where $Y$ is a
$n$-dimensional submanifold of $W$, is given by: $
X^{(k)}=\{[Y]^k_a\hskip 2pt\vert a\in X\}\subset
Y^{(k)}\equiv\{[Y]^k_b\hskip 2pt\vert\hskip 2pt b\in
 Y\}$. Here $[Y]^k_a$ denotes the equivalence class of
 $n$-dimensional submanifolds of $W$, having in $b\in W$ a contact
 of order $k$ with the $n$-dimensional submanifold $Y\subset W$,
 passing for $p$.
 The existence of admissible $p$-dimensional manifolds is obtained
 solving Cauchy problems of order $p\in\{0,\cdots,n-1\}$, i.e.,
 finding $n$-dimensional admissible integral manifolds (solutions) of
 a PDE $E_k\subset J^k_n(W)$, that contains some fixed integral
 manifolds of dimension $p<n$. We
 call {\em low dimension Cauchy problems}, Cauchy problems of
 dimension
 $0\le p\le  n-2$. We simply say {\em Cauchy problems}, Cauchy problems of
 dimension
 $p=n-1$.

In a satisfactory theory of PDE's it is necessary to consider in a
systematic way also {\em weak solutions}, i.e., solutions $V$,
where the set $\Sigma(V)$ of singular points of $V$, contains also
discontinuity points, $q,q'\in V$, with
$\pi_{k,0}(q)=\pi_{k,0}(q')=a\in W$, or
$\pi_{k}(q)=\pi_{k}(q')=p\in M$. We denote such a set by
$\Sigma(V)_S\subset\Sigma(V)$, and, in such cases we shall talk
more precisely of {\em singular boundary} of $V$, like $(\partial
V)_S=\partial V\setminus\Sigma(V)_S$. However for abuse of
notation we shall denote $(\partial V)_S$, (resp. $\Sigma(V)_S$),
simply by $(\partial V)$, (resp. $\Sigma(V)$), also if no
confusion can arise. Solutions with such singular points are of
great importance and must be included in a geometric theory of
PDE's too.
\end{remark}

\begin{definition}
Let $\Omega_{n-1}^{E_k}$, (resp. $\Omega_{n-1,s}^{E_k}$, resp.
$\Omega_{n-1,w}^{E_k}$), be the integral bordism group for
$(n-1)$-dimensional smooth admissible regular integral manifolds
contained in $E_k$, bounding smooth regular integral
manifold-solutions,\footnote{This means that
$N_1\in[N_2]\in\Omega_{n-1}^{E_k}$, iff
$N_1^{(\infty)}\in[N^{(\infty)}_2]\in\Omega_{n-1}^{E_{\infty}}$.
(See Refs.\cite{PRA4, PRA14} for notations.)} (resp. piecewise-smooth or singular
solutions, resp. singular-weak solutions), of $E_k$.
\end{definition}

\begin{theorem}
Let $\pi:W\to M$ be a fiber bundle with $W$ and $M$ smooth
manifolds, respectively of dimension $m+n$ and $n$. Let $E_k\subset
J^k_n(W)$ be a PDE for $n$-dimensional submanifolds of $W$.  One has
the following exact commutative diagram relating the groups
$\Omega_{n-1}^{E_k}$, $\Omega_{n-1,s}^{E_k}$ and
$\Omega_{n-1,w}^{E_k}$:

\begin{equation}
\xymatrix{
&0\ar[d]&0\ar[d]&0\ar[d]&\\
0\ar[r]&K^{E_k}_{n-1,w/(s,w)}\ar[d]\ar[r]&
K^{E_k}_{n-1,w}\ar[d]\ar[r]&K^{E_k}_{n-1,s,w}\ar[d]\ar[r]&0\\
0\ar[r]&K^{E_k}_{n-1,s}\ar[d]\ar[r]&
\Omega^{E_k}_{n-1}\ar[d]\ar[r]&\Omega^{E_k}_{n-1,s}\ar[d]\ar[r]&0\\
&0\ar[r]&\Omega^{E_k}_{n-1,w}\ar[d]\ar[r]&\Omega^{E_k}_{n-1,w}\ar[d]\ar[r]&0\\
&&0&0&}
\end{equation}

and the canonical isomorphisms: $K^{E_k}_{n-1,w/(s,w)}\cong
K^{E_k}_{n-1,s}$; $\Omega^{E_k}_{n-1}/K^{E_k}_{n-1,s}\cong
\Omega^{E_k}_{n-1,s}$;
$\Omega^{E_k}_{n-1,s}/K^{E_k}_{n-1,s,w}\cong\Omega^{E_k}_{n-1,w}$;
$\Omega^{E_k}_{n-1}/K^{E_k}_{n-1,w}\cong\Omega^{E_k}_{n-1,w}$. In
particular, for $k=\infty$, one has the following canonical
isomorphisms: $ K^{E_\infty}_{n-1,w}\cong K^{E_\infty}_{n-1,s,w}$;
$K^{E_\infty}_{n-1,w/(s,w)}\cong K^{E_\infty}_{n-1,s}\cong 0$;
$\Omega^{E_\infty}_{n-1}\cong \Omega^{E_\infty}_{n-1,s}$;
$\Omega^{E_\infty}_{n-1}/K^{E_\infty}_{n-1,w}\cong\Omega^{E_\infty}_{n-1,s}/K^{E_\infty}_{n-1,s,w}\cong
\Omega^{E_\infty}_{n-1,w}$. If $E_k$ is formally integrable then
one has the following isomorphisms:
$\Omega^{E_k}_{n-1}\cong\Omega^{E_\infty}_{n-1}\cong\Omega^{E_\infty}_{n-1,s}$.
\end{theorem}

\begin{proof}
The proof follows directly from the definitions and standard
results of algebra. \end{proof}

\begin{theorem}\label{formal-integrability-integral-bordism-groups}
Let us assume that $E_k$ is formally integrable and completely
integrable, and such that $\dim E_k\ge 2n+1$. Then, one has the
following canonical isomorphisms: $
\Omega^{E_k}_{n-1,w}\cong\oplus_{r+s=n-1}H_r(W;{\mathbb Z
}_2)\otimes_{{\mathbb Z}_2}\Omega_s
\cong\Omega^{E_k}_{n-1}/K^{E_k}_{n-1,w}\cong
\Omega^{E_k}_{n-1,s}/K^{E_k}_{n-1,s,w}$. Furthermore, if
$E_k\subset J^k_n(W)$, has non zero symbols: $g_{k+s}\not=0$,
$s\ge 0$, (this excludes that can be $k=\infty$), then
$K^{E_k}_{n-1,s,w}=0$, hence
$\Omega^{E_k}_{n-1,s}\cong\Omega^{E_k}_{n-1,w}$.
\end{theorem}

\begin{proof}
It follows from above theorem and results in \cite{PRA4}.
Furthermore, if $g_{k+s}\not=0$, $s\ge 0$, we can always connect two
branches of a weak solution with a singular solution of $E_k$.
\end{proof}

\begin{definition}
The {\em full space of $p$-conservation laws}, (or {\em full
$p$-Hopf algebra}), of $E_k$ is the following algebra: ${\bf
H}_p(E_k)\equiv{\mathbb R}^{\Omega_p^{E_k}}$.\footnote{This is, in
general, an extended Hopf algebra. (See Refs. \cite{PRA1, PRA2}.)} We call
{\em full Hopf algebra}, of $E_k$, the following algebra: ${\bf
H}_{n-1}(E_\infty)\equiv{\mathbb
R}^{\Omega_{n-1}^{E_\infty}}$.\end{definition}

\begin{definition}
The {\em space of (differential) conservation laws} of $E_k\subset
J^k_n(W)$, is ${\frak C}ons(E_k)={\frak I}(E_\infty)^{n-1}$, where
$${\frak I}(E_k)^{q}
\equiv\frac{\Omega^q(E_k)\cap
d^{-1}(C\Omega^{q+1}(E_k))}{d\Omega^{q-1}(E_k)\oplus\{C\Omega^q(E_k)\cap
d^{-1}(C\Omega^{q+1}(E_k))\}}$$ is the {\em space of characteristic
integral $q$-forms} on $E_k$ Here, $\Omega^q(E_k)$ is the space of
smooth $q$-differential forms on $E_k$ and $C\Omega^q(E_k)$ is the
space of Cartan $q$-forms on $E_k$, that are zero on the Cartan
distribution ${\bf E}_k$ of $E_k$. Therefore, $\beta\in
C\Omega^q(E_k)$ iff $\beta(\zeta_1,\cdots,\zeta_q)=0$, for all
$\zeta_i\in C^\infty({\bf E}_k)$.\footnote{${\frak C}ons(E_k)$ can
be identified with the spectral term $E_1^{0,n-1}$ of the spectral
sequence associated to the filtration induced in the graded algebra
$\Omega^\bullet(E_\infty)\equiv\oplus_{q\ge 0}\Omega^q(E_\infty)$,
by the subspaces $C\Omega^q(E_\infty)\subset\Omega^q(E_\infty)$.
(For abuse of language we shall call ''conservation laws of
$k$-order'', characteristic integral $(n-1)$-forms too. Note that
$C\Omega^0(E_k)=0$. See also Refs. \cite{PRA000, PRA01, PRA1, PRA2, PRA4}.)}
\end{definition}

\begin{theorem}{\em\cite{PRA4}}
The space of conservation laws of $E_k$ has a canonical
representation in ${\bf H}_{n-1}(E_\infty)$, (if the integral
bordism considered is not for weak-solutions).
\end{theorem}

\begin{theorem}
Set: $\mathbf{H}_{n-1}(E_k)\equiv
\mathbb{R}^{\Omega^{E_k}_{n-1}}$, $\mathbf{H}_{n-1,s}(
E_k)\equiv \mathbb{R}^{\Omega^{E_k}_{n-1,s}}$, $\mathbf{
H}_{n-1,w}(E_k)\equiv \mathbb{R}^{\Omega^{E_k}_{n-1,w}}$.
One has the exact and commutative diagram reported in {\em(\ref{commutative-exact-diagram-conservation-laws})}, that define the following spaces: $\mathbf{K}^{E_k}_{n-1,w/(s,w)}$, $\mathbf{K}^{E_k}_{n-1,w}$,
$\mathbf{K}^{E_k}_{n-1,s,w}$, $\mathbf{K}^{E_k}_{n-1,s}$.

\begin{equation}\label{commutative-exact-diagram-conservation-laws}
 \xymatrix{&0&0&0&\\
0&\ar[l]\mathbf{K}^{E_k}_{n-1,w/(s,w)}\ar[u]& \ar[l]\mathbf{K}^{E_k}_{n-1,w}\ar[u]& \ar[l]\mathbf{K}^{E_k}_{n-1,s,w}\ar[u]& \ar[l]0\\
0&\ar[l]\mathbf{K}^{E_k}_{n-1,s}\ar[u]& \ar[l]\mathbf{H}_{n-1}(E_k)\ar[u]&
\ar[l]\mathbf{H}_{n-1,s}(E_k)\ar[u]& \ar[l]0\\
 &0\ar[u]&\ar[l]\mathbf{H}_{n-1,w}(E_k)\ar[u]&\ar[l]\mathbf{H}_{n-1,w}(E_k)\ar[u]& \ar[l]0\\
&&0\ar[u]&0\ar[u]&}
\end{equation}

More explicitly, one has the following canonical isomorphisms:
\begin{equation}
\left\{
\begin{array}{l} \mathbf{K}^{
E_k}_{n-1,w/(s,w)}\cong\mathbf{K}^{K^{
E_k}_{n-1,s}};\\
\mathbf{K}^{E_k}_{n-1,w}/\mathbf{K}^{
E_k}_{n-1,s,w}\cong\mathbf{K}^{K^{E_k}_{n-1,w/(s,w)}};\\
\mathbf{H}_{n-1}(E_k)/\mathbf{H}_{n-1,s}(E_k)\cong\mathbf{
K}^{E_k}_{n-1,s};\\
\mathbf{H}_{n-1}(E_k)/\mathbf{H}_{n-1,w}(E_k)\cong\mathbf{
K}^{E_k}_{n-1,w} \\
\cong\mathbf{H}_{n-1,s}(E_k)/\mathbf{H}_{n-1,w}(
E_k)\cong\mathbf{K}^{E_k}_{n-1,s,w}.\\
\end{array}\right.\end{equation}

If $E_k$
is formally integrable one has: ${\bf H}_{n-1}(E_\infty)\cong{\bf
H}_{n-1}(E_k)\cong {\bf H}_{n-1,s}(E_\infty)$.
\end{theorem}

\begin{theorem}
Let us assume the same hypotheses considered in Theorem \ref{formal-integrability-integral-bordism-groups}. If
$N'\in[N]_{E_k}\in\Omega^{E_k}_{n-1,w}$, then there exists a
$n$-dimensional integral manifold (solution) bording $N'$ with
$N$, without discontinuities, i.e., a singular solution, iff all
the integral characteristic numbers of order $k$ of $N'$ are equal
to the integral characteristic numbers of the same order of $N$.
\end{theorem}

\begin{proof}
In fact we can consider a previous theorem given in Refs. \cite{PRA01, PRA1},
where it is proved that $N$ bounds with $N'$ a smooth
integral manifold iff the respective integral characteristic numbers
of order $k$ are equal.
\end{proof}

\begin{theorem}
Under the same hypotheses of Theorem \ref{formal-integrability-integral-bordism-groups}, and with
$g_{k+s}\not=0$, $s\ge 0$, one has the following canonical
isomorphism: ${\bf H}_{n-1,s}(E_k)\cong{\bf H}_{n-1,w}(E_k)$.
Furthermore, we can represent differential conservation laws of
$E_k$, coming from ${\frak I}(E_k)^{n-1}$, in ${\bf
H}_{n-1,w}(E_k)$.
\end{theorem}

\begin{proof}
Let us note that ${\frak I}(E_k)^{n-1}\subset{\frak
I}(E_\infty)^{n-1}$. If $j:{\frak C}ons(E_k)\to{\bf
H}_{n-1}(E_\infty)$, is the canonical representation of the space
of the differential conservation laws in the full Hopf algebra of
$E_k$, (corresponding to the integral bordism groups for regular
smooth solutions), it follows that one has also the following
canonical representation $j|_{{\frak I}(E_k)^{n-1}}:{\frak
I}(E_k)^{n-1}\to{\bf H}_{n-1,s}(E_k)\cong {\bf H}_{n-1,w}(E_k)$.
In fact, for any
$N'\in[N]_{E_k,s}\in\Omega_{n-1,s}^{E_k}\cong\Omega_{n-1,w}^{E_k}$,
one has $\int_{N'}\beta=\int_{N}\beta$, $\forall[\beta]\in {\frak
I}(E_k)^{n-1}$, i.e., the integral characteristic numbers of $N$
and $N'$ coincide.
\end{proof}

\begin{theorem}
Let $E_k\subset J^k_n(W)$ be a formally integrable and completely
integrable PDE, with $\dim E_k\ge 2n+1$. Let $\pi:W\to M$, be an
affine fiber bundle, over a $4$-dimensional affine space-time $M$.
Let us consider admissible only the closed $3$-dimensional
time-like smooth regular integral manifolds $N\subset E_k$. We
consider admissible only ones $N$ with zero all the integral
characteristic numbers. Then, there exists a smooth time-like
regular integral manifold-solution $V$, such that $\partial V=N$.
\end{theorem}

\begin{proof}
In fact, $E_k$ is equivalent, from the point of view of the
regular smooth solutions, to $E_\infty$. On the other hand, we
have:
$$\Omega^{E_\infty}_{3}/K^{E_\infty}_{3,w}\cong\Omega^{E_\infty}_{3,w}\cong
\bigoplus_{r+s=3}H_r(W;{\mathbb Z }_2)\otimes_{{\mathbb
Z}_2}\Omega_s=0.$$ Therefore, $\Omega^{E_\infty}_{3}\cong
K^{E_\infty}_{3,w}$. This means that any closed smooth time-like
regular integral manifold $N\subset E_k$, is the boundary of a weak
solution in $E_\infty$. On the other hand, since we have considered
admissible only such manifolds $N$ with zero integral characteristic
numbers, it follows that one has: $ \Omega^{E_\infty}_{3}=0$.
\end{proof}

\begin{definition} In Tab. \ref{important-spaces-associated-to-pde} we define some important spaces associated to a PDE
$E_k\subset J^k_n(W)$.\end{definition}

\begin{table}
\caption{Important spaces associated to PDE $E_k$.}
\label{important-spaces-associated-to-pde}
\begin{tabular}{|l|}
\hline
{\rm{\footnotesize (Space of characteristic $q$-forms, $q=1,2,\cdots$)}}\\
{\rm{\footnotesize $Ch\Omega^q(E_k)\equiv\{\beta\in\Omega^q(E_k) |
\beta(\zeta_1,\cdots,\zeta_q)(p)=0, \zeta_i(p)\in \mathbf{C}har(E_k)_p \hskip 2pt ,\hskip 2pt\forall p\in
E_k\};$}}\\
{\rm{\footnotesize $Ch\Omega^0(E_k)=0$.}}\\
\hline
{\rm{\footnotesize (Space of Cartan $q$-forms, $q=1,2,\cdots$)}}\\
{\rm{\footnotesize $C\Omega^q(E_k)\equiv\{\beta\in\Omega^q(E_k) |
\beta(\zeta_1,\cdots,\zeta_q)(p)=0, \zeta_i(p)\in (\mathbf{E}_k)
\hskip 2pt , \hskip 2pt\forall p\in E_k\};$}}\\
{\rm{\footnotesize $C\Omega^0(E_k)=0$.}}\\
\hline
{\rm{\footnotesize (Space of $p$-characteristic $q$-forms, $q=1,2,\cdots$)}}\\
{\rm{\footnotesize $Ch^p\Omega^q(E_k)\equiv\{\beta\in\Omega^q(E_k) |
\beta(\zeta_1,\cdots,\zeta_q)=0$, with condition $(\spadesuit)\}$.}}\\
{\rm{\footnotesize ($\spadesuit$): (If at least $q-p+1$ of the fields $\zeta_1,\cdots,\zeta_q$ are characteristic).}}\\
\hline
{\rm{\footnotesize (Space of $p$-Cartan $q$-forms, $q=1,2,\cdots$)}}\\
{\rm{\footnotesize $C^p\Omega^q(E_k)\equiv\{\beta\in\Omega^q(E_k) |
\beta(\zeta_1,\cdots,\zeta_q)=0$, with condition $(\clubsuit)\}$.}}\\
{\rm{\footnotesize ($\clubsuit$): (If at least $q-p+1$ of the fields $\zeta_1,\cdots,\zeta_q$ are Cartan).}}\\
\hline
\end{tabular}
\end{table}

\begin{remark} If the fiber dimension of $\mathbf{C}har(E_k)_p$ is $s$ one has:
$Ch\Omega^q(E_k)=\Omega^q(E_k)$, $q>s$.
If the fiber dimension of $\mathbf{E}_k$ is $r$ one has: $C\Omega^q(E_k)=\Omega^q(E_k),\quad q>r$.
If $k=\infty$ one has: $Ch\Omega^q(E_k)=C\Omega^q(E_\infty)$, $Ch^p\Omega^q(E_k)=C^p\Omega^q(E_\infty)$, $C\Omega^q(E_\infty)=\Omega^q(E_\infty)=Ch\Omega^q(E_k)$, $q>n$.
$\alpha\in Ch\Omega^q(E_k)$, iff $\alpha|_V=0$, for all the characteristic integral manifolds of $E_k$. If $E_\infty\subset J^k_n(W)$, then $\mathbf{C}har(E_\infty)=
\mathbf{E}_\infty$, and $Ch\Omega^q(E_\infty)=C\Omega^q(E_\infty)$. Furthermore, even if for any $p\in E_\infty$, one has
an infinity number of maximal integral manifolds (of dimension $n$) passing for $p$, one has that all these integral manifolds have at $p$ the same tangent space $(\mathbf{E}_\infty)_p$. Hence, a differential $q$-form on $E_\infty$ is Cartan iff it is zero on all the integral manifolds of $E_\infty$. One has the following natural differential complex:
\begin{equation}\label{characteristic-differential-complex-pde}
\begin{array}{l}
\xymatrix{0\ar[r]& Ch\Omega^1(E_k)\ar[r]^{d}&Ch\Omega^2(E_k)\ar[r]^{d}&\cdots}\\
\xymatrix{\cdots\ar[r]& Ch\Omega^s(E_k)\ar[r]^{d}&Ch\Omega^{s+1}(E_k)\ar[r]^{d}&\cdots\ar[r]&
Ch\Omega^r(E_k)\ar[r]^{d}&0},\\
\end{array}
\end{equation}
where $s=$ fiber dimension of $\mathbf{C}har(E_k)$, and $r=\dim E_k$, with
$k\le\infty$. In particular, if   $k=\infty$ we can write above
complex by fixing $Ch\Omega^q(E_\infty)=C\Omega^q(E_\infty)$. One
has: $d:Ch^p\Omega^q(E_k)\to Ch^p\Omega^{q+1}(E_k)$, $k\le\infty$.
In particular, for $k=\infty$ we can write
$d:C^p\Omega^q(E_\infty)\to C^p\Omega^{q+1}(E_\infty)$. One has
the following filtration compatible with the exterior
differential:
$$\begin{array}{l}
Ch^0\Omega^q(E_k)\equiv\Omega^q(E_k)\supset Ch^1\Omega^q(E_k)\equiv
Ch\Omega^q(E_k)\supset Ch^2\Omega^q(E_k)\supset\cdots\\
\cdots\supset Ch^q\Omega^q(E_k)\supset 0,\\
\end{array}$$

for $k\le\infty$. As a consequence
we have associated a spectral sequence ({\em characteristic
spectral sequence}\index{characteristic spectral sequence} of $E_k$): $\{E_r^{p,q}(E_k), d_r^{p,q}\}$. In
particular, if   $E_\infty$ is the infinity prolongation of a PDE
$E_k\subset J^k_n(W)$ above spectral sequence applied to
$E_\infty$ coincides with the $\mathcal{C} $-spectral sequence of
$E_k$ \cite{PRA1, PRA5}. Of particular importance is the following
spectral term:
$$E_1^{0,q}(E_k)={{\Omega^q(E_k)\cap d^{-1}Ch\Omega^{q+1}(E_k)}\over
{d\Omega^{q-1}(E_k)\oplus Ch\Omega^q(E_k)}}$$
that for $k=\infty$ can be also written
$$E_1^{0,q}(E_k)={{\Omega^q(E_k)\cap d^{-1}C\Omega^{q+1}(E_k)}\over
{d\Omega^{q-1}(E_k)\oplus C\Omega^q(E_k)}}.$$ Set:
$\bar\Omega^q(E_k)=\Omega^q(E_k)/Ch\Omega^q(E_k)$, $\le\infty$.
Then, one has the following differential complex associated to
$E_k$ ({\em bar de Rham complex}\index{bar de Rham complex} of $E_k$):
$$\begin{array}{l}
\xymatrix{0\ar[r]&\bar\Omega^0(E_k)\ar[r]^{\bar d}&
\bar\Omega^1(E_k)\ar[r]^{\bar d}&\bar\Omega^2(E_k)\ar[r]^{\bar
d}&\cdots0}\\
\xymatrix{\cdots\ar[r]&\bar\Omega^s(E_k)\ar[r]^{\bar d}&0}.\\
\end{array}$$

We call {\em bar de Rham cohomology}\index{bar de Rham cohomology} of $E_k$ the corresponding homology
$\bar H^q(E_k)$. One has the following canonical isomorphism:
$E_1^{0,q}(E_k)\cong \bar H^q(E_k)\hskip 2pt,\quad k\le\infty$.
\end{remark}

\begin{definition} Set
$$\begin{array}{ll}
  \mathcal{I}(E_k)^p&\equiv{{\Omega^p(E_k)\cap d^{-1}(C\Omega^{p+1}(E_k))}\over
{d\Omega^{p-1}(E_k)\oplus\{C\Omega^p(E_k)\cap d^{-1}(C\Omega^{p+1}(E_k))\}}},\\
\mathcal{Q}^k_n(W)^p&\equiv\mathcal{I}(J^k_n(W))^p.\\
  \end{array}$$
\end{definition}

\begin{remark} For $k=\infty$ one has: $\mathcal{I}(E_\infty)^p \cong E_1^{0,p}(E_\infty)\cong \bar
H^p(E_\infty)$.
\end{remark}

\begin{theorem} \cite{PRA01, PRA1}\label{necessary-sufficient-conditions-existence-integral-quantum-solutions-bording} {\em 1)} Let us assume that
$\mathcal{I}(E_k)^p\not=0$. One has a natural group homomorphism:
$$\begin{array}{l}
j_p:\Omega^{E_k}_p\to(\mathcal{I}(E_k)^p)^*\\
       {[N]_{E_k}\mapsto j_p([N]_{E_k}),\quad j_p([N]_{E_k})([\alpha])=
       \int_N\alpha\equiv<[N]_{E_k},[\alpha]>.}\\
       \end{array}$$
We call $i[N]\equiv <[N]_{E_k},[\alpha]>$ {\em integral characteristic numbers} of $N$ for all $[\alpha]\in \mathcal{I}(E_k)^p$. Then a necessary condition that $N'\in [N]_{E_k}$ is
the following
\begin{equation}\label{necessary-condition-belonging-same-integral-bordism-class}
i[N']=i[N] \hskip 2pt , \hskip 2pt\forall [\alpha]\in \mathcal{I}(E_k)^p.\end{equation}
Above condition is also sufficient for $k=\infty$ in
order to identify elements belonging to the same singular integral bordism
classes of $\Omega^{E_\infty}_{p,s}$.  In fact, one has the following exact
commutative diagram:
$$\xymatrix{&&0\ar[d]&\\
      &\Omega^{E_\infty}_p\ar[d]_{j_p}\ar[r]^(0.4){i_p}&\Omega^{E_\infty}_{p,s}\equiv\bar H_p(E_\infty;\mathbb{R})\ar[d]&\\
      0\ar[r]&(\mathcal{I}(E_\infty)^p)^*\ar[r]&\bar H^p(E_\infty;\mathbb{R})^*\ar[r]&0}$$

{\em 2)} For any $k\le \infty$ one has the following exact commutative diagram:
$$\xymatrix{&&&0\ar[d]&\\
      0\ar[r]&\overline{K^{E_k}_p}\ar[r]&\Omega^{E_k}_p\ar[dr]_{j_p}\ar[r]&\overline{\Omega^{E_k}_p}\ar[d]\ar[r] &0\\
      &&&(\mathcal{I}(E_k)^p)^*&}$$
Therefore, we can write
$$\begin{array}{l}
 \overline{K^{E_k}_p}\equiv\{ [N]_{E_k} |<[\alpha],[N]_{E_k}>=0,
\forall [\alpha]\in \mathcal{I}(E_k)^p \}\\
N'\in\overline{[N]_{E_k}}\in\overline{\Omega^{E_k}_p}\Leftrightarrow
\int_{N'}\alpha=\int_N\alpha \hskip 2pt , \hskip 2pt\forall [\alpha]\in\mathcal{I}(E_k)^p.\\
  \end{array}$$

{\em 3)} Let us assume that $\mathcal{Q}^k_n(W)^p\not=0$. One has a natural group
homomorphism
$$\begin{array}{l}
\bar j_p:\Omega_p(E_k)\to (\mathcal{Q}^k_n(W)^p)^*\\
      { [N]_{\overline{E_k}}\mapsto \bar j_p([N]_{\overline{E_k}})}\\
       {\bar j_p([N]_{\overline{E_k}})([\alpha]) =\int_N\alpha
       \equiv<[N]_{\overline{E_k}},[\alpha]>.}\\
       \end{array}$$
We call $q[N]\equiv<[N]_{\overline{E_k}},[\alpha]>$ {\em
quantum characteristic numbers} of $N$, for all
$[\alpha]\in\mathcal{Q}^k_n(W)^p$. Then, a necessary condition that
$N'\in [N]_{\overline{E_k}}$ is that
\begin{equation}\label{necessary-condition-belonging-same-quantum-bordism-class}
q[N']=q[N] \hskip 2pt , \hskip 2pt\forall [\alpha]\in\mathcal{Q}^k_n(W)^p.\end{equation}

{\em 4) (Criterion in order condition (\ref{necessary-condition-belonging-same-integral-bordism-class}) should be sufficient).} Let us assume
that $E_k\subset J^k_n(W)$ is such that all its $p$-dimensional
compact closed admissible integral submanifolds are orientable and
$\mathcal{I}(E_k)^p\not=0$.\footnote{It is important to note that
can be $\mathcal{I}(E_k)^p\not=0$ even if $E_k$ is $p$-cohomologic trivial, i.e.,
$H^p(E_k;\mathbb{R})=0$. This, for example, can happen
if $E_k$ is contractible to a point.} Then,
$\ker(j_p)=0$, i.e.,
$$N'\in[N]_{E_k}\Leftrightarrow \int_{N'}\alpha=\int_N\alpha \hskip 2pt , \hskip 2pt \forall
[\alpha]\in\mathcal{I}(E_k)^p.$$
In particular, for $k=\infty$, one has $\Omega^{E_\infty}_p\cong\Omega^{E_\infty}_{p,s}$ as  $\mathcal{I}(E_\infty)^p\cong\bar H^p(E_\infty)$.

{\em 5)} Under the same hypotheses of above theorem one has
$$N'\in[N]_{\overline{E_k}}\Leftrightarrow \int_{N'}\alpha=\int_N\alpha
\hskip 2pt,\quad
\forall [\alpha]\in\mathcal{Q}^k_n(W)^p.$$
\end{theorem}

\begin{proof} See \cite{PRA01, PRA1}.\end{proof}

\begin{remark} In above criterion $\Omega^{E_k}_p$ (resp. $ \Omega_p(E_k)$) does not
necessarily coincides with the oriented version of the integral (resp. quantum)
bordism groups. In fact, the M\"obius band\index{moebius@M\"{o}bius band} is an example of non orientable
manifold $B$ with $\partial  B\cong S^1$, that, instead, is an orientable
manifold.\end{remark}

\begin{remark} The oriented version of integral and quantum bordism can be similarly obtained by substituting the groups $\Omega_p$ with the corresponding groups ${}^+\Omega_p $ for oriented manifolds. We will not go in to details.\end{remark}

Let us give, now, a full characterization of singular integral and quantum (co)bordism groups by means of suitable characteristic numbers.

\begin{definition} {\em 1)} Let $E_k\subset J^k_n(W)$ be a PDE. We call {\em bar
singular chain complex}\index{bar singular chain complex}, {\em with coefficients into an abelian
group $G$}, of $E_k$ the chain complex $\{\bar
C_p(E_k;G),\bar\partial\}$, where $\bar C_p(E_k;G)$ is the
$G$-module of formal linear combinations, with coefficients in
$G$, $\sum \lambda_i c_i$, where $c_i$ is a singular $p$-chain
$f:\bigtriangleup ^p\to E_k$ that extends on a neighborhood
$U\subset\mathbb{R}^{p+1}$, such that $f$ on $U$ is differentiable
and $Tf(\bigtriangleup^p)\subset\mathbf{E}_k$. Denote by $\bar
H_p(E_k;G)$ the corresponding homology ({\em bar singular
homology with coefficients in $G$})\index{bar singular homology} of $E_k$. Let $\{\bar
C^p(E_k;G)\equiv Hom_{\mathbb{Z}}(\bar C_p(E_k;\mathbb{Z});G),
\bar\delta\}$ be the corresponding dual complex and $\bar
H^p(E_k;G)$ the associated homology spaces ({\em bar singular
cohomology}\index{bar singular cohomology}, {\em with coefficients into $G$} of $E_k$).

{\em 2)} A {\em $G$-singular $p$-dimensional integral manifold}\index{G@$G$-singular integral manifold} of
$E_k\subset J^k_n(W),$ is a bar singular $p$-chain $V$ with $p\le
n$, and coefficients into an abelian group $G$, such
that $V\subset E_k$.

{\em 3)} Set $\bar B_{\bullet}(E_k;G)\equiv \IM(\bar\partial)$, $\bar Z_{\bullet}(E_k;G)\equiv \ker(\bar\partial)$. Therefore, one has the following exact commutative diagram:
$$\xymatrix{&&0\ar[d]&0\ar[d]&&\\
&0\ar[r]&\bar B_{\bullet}(E_k;G)\ar[d]\ar[r]&\bar Z_{\bullet}(E_k;G)\ar[d]\ar[r]&\bar H_{\bullet}(E_k;G)\ar[r]&0\\
&&\bar C_{\bullet}(E_k;G)\ar[d]\ar@{=}[r]&\bar C_{\bullet}(E_k;G)\ar[d]&&\\
0\ar[r]&{}^G\Omega^{E_k}_{\bullet,s}\ar[r]&\bar Bor_{\bullet}(E_k;G)\ar[d]\ar[r]&\bar Cyc_{\bullet}(E_k;G)\ar[d]\ar[r]&0&\\
&&0&0&&}$$
where $\bar Bor_{\bullet}(E_k;G)\equiv$ {\em bordism group};
$b\in{}^G[a]_{E_k}\in\bar Bor_{\bullet}(E_k;G)\Rightarrow\exists
c\in \bar C_{\bullet}(E_k;G): \bar\partial c=a-b$; $\bar
Cyc_{\bullet}(E_k;G)\equiv$ {\em cyclism group};
$b\in{}^G[a]_{E_k}\in\bar Cyc_{\bullet}(E_k;G)\Rightarrow
\bar\partial(a-b)=0$; ${}^G\Omega^{E_k}_{\bullet,s}\equiv$ {\em closed bordism group};
$b\in{}^G[a]_{E_k}\in{}^G\Omega^{E_k}_{\bullet,s}\Rightarrow
\left\{\begin{array}{l}
\bar\partial a=\bar\partial b=0\\
a-b=\bar\partial c\\
       \end{array}\right\}$.
\end{definition}

\begin{theorem} {\em 1)} One has the following canonical isomorphism: ${}^G\Omega^{E_k}_{\bullet,s}\cong\bar H_{\bullet}(E_k;G)$.

{\em 2)} If ${}^G\Omega^{E_k}_{\bullet,s}=0$ one has: $\bar Bor_{\bullet}(E_k;G)\cong\bar Cyc_{\bullet}(E_k;G)$.

{\em 3)} If $\bar Cyc_{\bullet}(E_k;G)$ is a free $G$-module, then the bottom horizontal exact sequence, in above diagram, splits and one has the isomorphism:
$$\bar Bor_{\bullet}(E_k;G)\cong {}^G\Omega( E_k)_{\bullet,s}\bigoplus\bar Cyc_{\bullet}(E_k;G).$$
\end{theorem}

\begin{remark} In the following we shall consider only closed
bordism groups $ {}^G\Omega (E_k)_{s,\bullet}$. So, we will omit
the term "closed". Similar definitions and results can be
obtained in dual form by using the cochain complex $ \{\bar
C^\bullet(E_k;G);\bar\delta\}$.\end{remark}

\begin{definition} A {\em $G$-singular $p$-dimensional quantum manifold}\index{G@$G$-singular quantum manifold} of
$E_k$ is a bar singular $p$-chain $V\subset J^k_n(W)$, with $p\le
n$, and coefficients into an abelian group $G$, such that
$\partial  V\subset E_k$. Let us denote by ${}^G\Omega_{p,s}(E_k)$
the corresponding (closed) bordism groups in the singular case.
Let us denote also by ${}^G[N]_{\overline{E_k}}$ the equivalence
classes of quantum singular bordisms respectively.
\end{definition}

\begin{remark} In the following, for $G=\mathbb{R}$ we will omit the apex $G$
in the symbols ${}^G\Omega_{p,s}^{E_k}$, ${}^G\Omega^{p,s}_{E_k}$
and ${}^G\Omega_{p,s}(E_k)$.\end{remark}

\begin{theorem}[Bar de Rham theorem for PDEs]\index{bar de Rham theorem for PDEs} One has a natural
bilinear mapping: $<,>:\bar C_p(E_k;\mathbb{R})\times\bar
C^p(E_k;\mathbb{R})\to\mathbb{R}$ such that: ({\em bar Stokes
formula}) $<\bar\delta\alpha,c>+(-1)^p<\alpha,\bar\partial c>=0$.
One has the canonical isomorphism: $\bar H^p(E_k;\mathbb{R})\cong
Hom_\mathbb{R}(\bar H_p(E_k;\mathbb{R});\mathbb{R})\equiv\bar H_p(E_k;\mathbb{R})^*$, and a nondegenerate mapping: $$<,>:\bar H_p(E_k;\mathbb{R})\times \bar H^p(E_k;\mathbb{R})\to \mathbb{R}.$$ Hence one has the
following short exact sequence
$$\xymatrix{0\ar[r]&\bar H_p(E_k;\mathbb{R})\ar[r]&\bar H^p(E_k;\mathbb{R})^*}.$$
This means that if $c$ is a
$\bar\partial$-closed bar singular $p$-chain ($\bar\partial c=0$)
of $E_k$, $c$ is the boundary of a bar-singular $(p+1)$-chain $c'$
of $E_k$ ($\bar\partial c'=c$) , iff $<c,\alpha>=0$, for all the
$\bar\delta$-closed bar singular $p$-cochains $\alpha$ of $E_k$. Furthermore, if   $\alpha$ is a $\bar\delta$-closed bar singular
$p$-cochain of $E_k$, $\alpha$ is $\bar\delta$-exact,
($\alpha=\bar\delta\beta$) iff $<c,\alpha>=0$, for all the
$\bar\partial$-closed bar singular $p$-chains $c$ of $E_k$.
\end{theorem}

\begin{proof} The full proof has been given in \cite{PRA01, PRA1}.\end{proof}

\begin{remark} {\em 1)} Similarly to the classical case, we can also define the
{\em relative (co)homology spaces}\index{relative homology spaces ! relative cohomology spaces} $\bar H^p(E_k,X;\mathbb{R})$ and
$\bar H_p(E_k,X;\mathbb{R})$, where $X\subset E_k$ is a bar singular
chain.

{\em 2)}  One has the following exact sequence:
$$\begin{array}{l}
   \xymatrix{\cdots\ar[r]&\bar H_p(X;\mathbb{R})\ar[r]& \bar H_p(E_k;\mathbb{R})\ar[r]&\bar H_p(E_k,X;\mathbb{R})
\ar[r]&\bar H_{p-1}(X;\mathbb{R})\ar[r]&\cdots}\\
\xymatrix{\cdots\bar H_0(X;\mathbb{R})\ar[r]&\bar H_0(E_k;\mathbb{R})\ar[r]&\bar H_0(E_k,X;\mathbb{R})\ar[r]& 0}.\\
\end{array}$$

{\em 3)} One has the following isomorphisms:
$\bar H_p(E_k,*;\mathbb{R})\cong \bar H_p(E_k;\mathbb{R})$, with $ p>0$; $ \bar H_0(E_k,*;\mathbb{R})=0$
if $E_k$ is arcwise connected.\end{remark}

\begin{theorem} Let us assume that $E_k\subset J^k_n(W)$ is a formally integrable PDE.

{\em 1)} As $\pi_\infty:E_\infty\to E_k$ is surjective, one has the following
short exact sequence of chain complexes:
$$\begin{array}{l}
\xymatrix{\bar C_{\bullet}(E_\infty;\mathbb{R})\ar[r]&\bar C_{\bullet}(E_k;\mathbb{R})\ar[r]& 0},\\
\xymatrix{\bar C^\bullet(E_\infty;\mathbb{R})&\bar C^\bullet(E_k;\mathbb{R})\ar[l]&0\ar[l]}.\\
\end{array}$$
These induce the following homomorphisms of vector spaces:
$\bar H_p(E_\infty;\mathbb{R})\to\bar H_p(E_k;\mathbb{R})$,
       $\bar H^p(E_\infty;\mathbb{R})\leftarrow\bar H^p(E_k;\mathbb{R})$.

{\em 2)} One has the following isomorphisms:
$\Omega^{E_k}_{p,s}\cong\bar H_p(E_k;\mathbb{R}),\quad k\le\infty$,
       $\Omega_{p,s}(E_k)\cong\bar H_p(J^k_n(W),E_k;\mathbb{R})$.

{\em 3)} One has the following exact sequences of vector spaces:
$$\begin{array}{l}
\xymatrix{\Omega^{E_k}_{n-1,s}\ar[r]^{a_{n-1}}&\Omega^{J^k_n(W)}_{n-1,s}
\ar[r]^{b_{n-1}}&\Omega_{n-1,s}(E_k)\ar[r]^{c_{n-1}}&\Omega^{E_k}_{n-2,s}\ar[r]^{a_{n-2}}&\cdots}\\
\xymatrix{\cdots\ar[r]^{a_{0}}&\Omega^{E_k}_{0,s}\ar[r]^{b_{0}}&\Omega^{J^k_n(W)}_{0,s}\ar[r]^{c_{0}}&\Omega_{0,s}(E_k)
\ar[r]&0}.\\
\end{array}$$
Therefore, one has unnatural splits:
$$\Omega_{p,s}(E_k)\cong\underline{\Omega^{J^k_n(W)}_{p,s}}\times
\overline{\Omega^{E_k}_{p-1,s}}\hskip 2pt ;\quad
\Omega^{J^k_n(W)}_{p,s}\cong\underline{\Omega^{E_k}_{p,s}}\times\overline{\Omega_{p,s}(E_k)},$$
where
$$\begin{array}{l}
\underline{\Omega^{J^k_n(W)}_{p,s}}\equiv \IM(b_p)\cong\ker(c_p),\\
\overline{\Omega^{E_k}_{p-1,s}}\equiv\IM(c_p)\cong\COIM(c_p)\equiv\Omega_{p,s}(E_k)/\ker(c_p)
\cong\COKER(b_p)\equiv\Omega_{p,s}(E_k)/\IM(b_p),\\
\underline{\Omega^{E_k}_{p,s}}\equiv\IM(a_p)\cong\ker(b_p),\\
\overline{\Omega_{p,s}(E_k)}\equiv\IM(b_p)\cong\COIM(b_p)\equiv
\Omega^{J^k_n(W)}_{p,s}/\ker(b_p)\cong\COKER(a_p)\equiv\Omega^{J^k_n(W)}_{p,s}/\IM(a_p).\\
\end{array}$$

{\em 4)} One has a natural homomorphism:
${\pi_{\infty,k}}_*:\Omega^{E_\infty}_{p,s}\to\Omega^{E_k}_{p,s}$.
\end{theorem}

\begin{definition}\label{singular-integral-characteristic-numbers-bar-closed-singular-integral-manifolds}
We call {\em singular integral characteristic numbers}\index{singular integral characteristic numbers} of a $p$-dimensional
$\bar\partial$-closed singular integral manifold $N\subset
E_k\subset J^k_n(W)$ the numbers $i[N]\equiv<N,\alpha>\in\mathbb{R}$,
where $\alpha$ is a $\bar\delta$-closed bar singular $p$-cochain
of $E_k$.\end{definition}

\begin{definition}\label{singular-quantum-characteristic-numbers-bar-closed-singular-integral-manifolds}
 We call {\em singular quantum characteristic numbers}\index{singular quantum characteristic numbers} of a $p$-dimensional
$\bar\partial$-closed singular integral manifold $N\subset
E_k\subset J^k_n(W),$ the numbers $q[N]\equiv<N,\alpha>\in\mathbb{R}$, where $\alpha$ is a $\bar\delta$-closed bar singular
$p$-cochain of $J^k_n(W)$.\end{definition}

\begin{theorem} {\em 1)}
$N'\in[N]^s_{E_k}\Leftrightarrow N'$ and $N$ have equal all the
singular integral characteristic numbers: $i[N']=i[N]$.

{\em 2)} $N'\in [N]^s_{\overline{E_k}}\Leftrightarrow N'$ and $N$ have
equal all the singular quantum characteristic numbers:
$q[N']=q[N]$.
\end{theorem}

 \begin{proof} It follows from the bar de Rham theorem
that one has the following short exact sequences:
$$\begin{array}{l}
\xymatrix{0\ar[r]&\Omega^{E_k}_{p,s}\ar[r]&\bar H^p(E_k;\mathbb{R})^*},\\
\xymatrix{0\ar[r]&\Omega_{p,s}(E_k)\ar[r]&\bar H^p(J^k_n(W),E_k;\mathbb{R})^*}.\\
\end{array}$$
\end{proof}

\begin{theorem} The relation between singular integral (quantum) bordism groups and
homology is given by the following exact commutative diagrams:
$$\xymatrix{&&&0\ar[d]&\\
      0\ar[r]& K\bar H_p(E_k;\mathbb{R})\ar[r]&\Omega^{E_k}_{p,s}\ar[r]&\widehat{\bar H}_p(E_k;\mathbb{R})\ar[d]\ar[r]&0\\
      &&&H_p(E_k;\mathbb{R})&}$$
where
$$\begin{array}{ll}
   K\bar H_p(E_k;\mathbb{R})&\equiv\{[N]^s_{E_k}| N=\partial V, V=\hbox{singular
       $p$-chain in $E_k$}\}\\
       &\equiv\{[N]^s_{E_k}| <[\alpha]|[N]^s_{E_k}>=0,
       \forall [\alpha]\in H^p(E_k;\mathbb{R})\}.\\
  \end{array}$$
We call $s[N]\equiv <[\alpha]|[N]^s_{E_k}>\equiv$ {\em singular
characteristic numbers}\index{singular characteristic numbers} of  $[N]^s_{E_k}$.
$$\xymatrix{&&&0\ar[d]&\\
      0\ar[r]&K\bar H_p(J^k_n(W),E_k;\mathbb{R})\ar[r]&\Omega_{p,s}(E_k)\ar[r]&{\widehat{\bar H}_p(J^k_n(W),E_k;\mathbb{R})}\ar[d]\ar[r]&0\\
      &&&H_p(J^k_n(W),E_k;\mathbb{R})&}$$
where
$$\begin{array}{l}
K\bar H_p(J^k_n(W),E_k;\mathbb{R})\\
\equiv\{[N]^s_{\overline{E_k}}| N=\partial V, V=\hbox{\rm singular$p$-chain in $J^k_n(W)$}\}\\
\equiv\{[N]^s_{\overline{E_k}}| <[\alpha]|[N]^s_{\overline{E_k}}>=0,
       \forall [\alpha]\in H^p(J^k_n(W),E_k;\mathbb{R})\}.\\
       \end{array}$$
We call {\em singular characteristic numbers}\index{singular characteristic numbers} of
$[N]^s_{\overline{E_k}}$ the numbers $s_q[N]\equiv
<[\alpha]|[N]^s_{\overline{E_k}}>$.
\end{theorem}

\begin{theorem} {\em 1)} The integral
bordism group $\Omega^{E_k}_p, 0\le p\le n-1$, is an extension of
a subgroup $\hat\Omega^{E_k}_{p,s}$ of the singular integral
bordism group $\Omega^{E_k}_{p,s}$.

{\em 2)} The quantum
bordism group $\Omega_p(E_k), 0\le p\le n-1$, is an extension of a
subgroup $\hat\Omega_{p,s}(E_k)$ of the singular quantum bordism
group $\Omega_{p,s}(E_k)$.
\end{theorem}

\begin{proof} 1) In fact, one has a
canonical group-homomorphism $j_p:\Omega^{E_k}_p\to
\Omega^{E_k}_{p,s}$, that generates the following exact commutative diagram:
$$\xymatrix{&&&0\ar[d]&&\\
      0\ar[r]&K^{E_k}_{p,s}\ar[r]&\Omega^{E_k}_p\ar[dr]_{j_p}\ar[r]^{i_p}&\hat\Omega^{E_k}_{p,s}\ar[d]\ar[r]& 0&\\
      &&0\ar[r]&\Omega^{E_k}_{p,s}\ar[r]&\bar H_p(E_k;\mathbb{R})\ar[r]&0}$$
where  $K^{E_k}_{p,s}\equiv\ker(j_p)$ and
$\hat\Omega^{E_k}_{p,s}\equiv\Omega^{E_k}_p/K^{E_k}_{p,s}$. Furthermore,
$K^{E_k}_{p,s}$ can be characterized by means of characteristic numbers. In
fact we get
$$\begin{array}{l}
K^{E_k}_{p,s}\\
=\left\{[N]^s_{E_k}|\exists (p+1)-\hbox{\rm dimensional singular
            integral submanifold $V\subset E_k$, with $\partial  V=N$}\right\}\\
=\left\{[N]^s_{E_k}| i[N]=0\hskip 2pt \hbox{\rm for all
            singular integral characteristic numbers}\right\}.\\
  \end{array}$$

2) In fact, one has a canonical group homomorphism
$\bar j_p:\Omega_p(E_k)\to\Omega_{p,s}(E_k)$, hence one has the following exact
commutative diagram:
$$\xymatrix{&&&0\ar[d]&&\\
      0\ar[r]& K_{p,s}(E_k)\ar[r]&\Omega_p(E_k)\ar[dr]_{j_p}\ar[r]&\hat\Omega_{p,s}(E_k)\ar[d]\ar[r]& 0&\\
      &&0\ar[r]&\Omega_{p,s}(E_k)\ar[r]&\bar H_p(J^k_n(W),E_k;\mathbb{R})\ar[r]& 0}$$
where $$K_{p,s}(E_k)\equiv\{[N]^s_{\overline{E_k}}
|q[N]=0, \hskip 2pt\hbox{\rm for all singular quantum characteristic
numbers}\}.$$\end{proof}

In \cite{PRA1} we have also related integral (co)bordism groups of
PDEs to some spectrum in such a way to generalize also to PDEs the
Thom-Pontrjagin construction usually adopted for bordism theories.
In fact we have the following theorem.

\begin{theorem}[Integral spectrum of PDEs]\label{integral-spectrum-pdes}

 {\em 1)} Let $E_k\subset
J^k_n(W)$ be a PDE. Then there is a spectrum $\{\Xi_s\}$ ({\em
singular integral spectrum of PDEs}), such that
$\Omega_{p,s}^{E_k}=\lim_{r\to\infty}\pi_{p+r}(E^+_k\wedge
\Xi_r)$,
$\Omega^{p,s}_{E_k}=\lim_{r\to\infty}[S^rE^+_k,\Xi_{p+r}]$,
$p\in\{0,1,\dots,n-1\}$.

{\em 2)} There exists a spectral
sequence $\{E^r_{p,q}\}$, (resp. $\{E_r^{p,q}\}$), with
$E^2_{p,q}=H_p(E_k,E_q(*))$, (resp. $E_2^{p,q}=H^p(E_k,E^q(*))$),
converging to $\Omega^{E_k}_{\bullet,s}$, (resp.
$\Omega_{E_k}^{\bullet,s}$). We call the spectral sequences
$\{E^r_{p,q}\}$ and $\{E_r^{p,q}\}$ the {\em integral singular
spectral sequences}\index{integral singular spectral sequences} of $E_k$.
\end{theorem}

\begin{proof} See \cite{PRA1}.\end{proof}

Let us, now, relate integral bordism to the spectral term $E^{0,n-1}_1$ of the $\mathcal{C} $-spectral sequence, that represents the space of conservation laws of PDEs. In fact we
represent $E^{0,n-1}_1$ into Hopf algebras that give the true full meaning of conservation laws of PDEs.

\begin{definition} We define {\em conservation law}\index{conservation law} of a PDE $E_k\subset
J^k_n(W) $, any differential $(n-1)$-form $\beta $ belonging to
the following quotient space:
$$\mathcal{C} ons(E_k)\equiv{\Omega^{n-1}(E_\infty)\cap{d^{-1}C\Omega^n(E_\infty)}\over
        {C\Omega^{n-1}(E_\infty)\bigoplus d\Omega^{n-2}(E_\infty)}}, $$
where $\Omega^q(E_\infty), q=0,1,2,\dots ,$ is the space of
differential $q$-forms on $E_\infty \hskip 2pt $,
$C\Omega^q(E_\infty)$ is the space of all Cartan $q$forms on
$E_\infty$, $q=1,2,\dots$, (see Tab. \ref{important-spaces-associated-to-pde}), and
$C\Omega^o(E_\infty)\equiv 0$,
$C\Omega^q(E_\infty)\equiv\Omega^q(E_\infty)$, for $q>n$,
$\Omega^{-1}(E_\infty)=0 $. Thus a conservation law is a
$(n-1)$-form on $E_\infty$ non trivially closed on the (singular)
solutions of $E_k $. The space of conservation laws of $E_k$
can be identified with the spectral term $E^{0,n-1}_1$ of the
$\mathcal{C} $-spectral sequence associated to $E_k$. One can see that
locally we can write
$$\mathcal{C} ons(E_k)={{\big\{\omega\in\Omega^{n-1}(E_\infty)\vert tial\omega=0\big\}}
          \over{\big\{\omega=tial\theta \vert \theta\in\Omega^{n-2}(E_\infty)\big\}}},$$
where
$$\partial \omega=\sum_{\mu_0,\dots,\mu_{n-1}}(\partial_{[\mu_0}
 \omega_{\mu_1\dots\mu_{n-1}]})dx^{\mu_0}\wedge\cdots\wedge dx^{\mu_{n-1}},$$
with
$$\omega=\sum_{\mu_1,\dots,\mu_{n-1}}\omega_{\mu_1\dots\mu_{n-1}}(x^\mu,y^j)
  dx^{\mu_1}\wedge\cdots\wedge dx^{\mu_{n-1}} mod \hskip 2pt C\Omega^{n-1}(E_\infty)$$
and
$$\partial _\mu\equiv\partial x_\mu+\sum_{i\in I}A^i_\mu(x,y)\partial y_i \hskip 2pt , \hskip 2pt \mu=1,\dots,n, $$
basis Cartan fields of $ \mathbf{E}_\infty$, where
$\{ x^\mu ,y^j\}_{1\le\mu\le k,j\in I}$ are adapted coordinates.
\end{definition}

\begin{theorem}\label{conservation-laws-and-integral-characteristic-forms} {\em 1)} One has the canonical isomorphism:
$\mathcal{I}(E_\infty)^{n-1}\cong\mathcal{C} ons(E_\infty)$.
So that integral numbers of $E_\infty$ can be considered as conserved charges of $E_k$.

{\em 2)} One has the following homomorphism of vector spaces
\begin{equation}\label{homomorphism-vector-spaces-conservation-laws-function-bordism}
j:E^{0,n-1}_1\to\mathbb{R}^{\Omega^{E_\infty}_{n-1}}.\end{equation}

Then  $E_1^{0,n-1}$ identifies a subspace $E^{0,n-1}$ of
$\mathbb{R}^{\Omega^{E_\infty}_{n-1}}$,
where
$$E^{0,n-1}\equiv im(j)=
\big\{ \phi\in \mathbb{R}^{\Omega^{E_\infty}_{n-1}} \vert
\exists\beta\in E^{0,n-1}_1 ,\hskip 2pt
\phi([N]_{E_\infty})=\int_N\beta\vert_N \big\}.$$
\end{theorem}

\begin{proof} 1) It
is a direct consequence of previous definitions and results.

2) In fact, to any conservation law $\beta :E_\infty\rightarrow
\Lambda ^o_{n-1}(E_\infty)$ we can associate a function
$j(\beta)\equiv\phi :\Omega ^{E_\infty}_{n-1}\rightarrow \mathbb{R}$,
$\phi ([N])=\int _N\beta \vert _N $. This definition has sense as
it does not depend on the representative used for
$[N]_{E_\infty}$. In fact, if   $\beta $ is a conservation law,
then $\forall V\in\Omega(E_\infty)_c $, with $\partial  V=N_0\sqcup N_1 $, we have
$$\int _{\partial V}\beta \vert _{\partial V}=
\int _Vd\beta \vert _V=0\Rightarrow
 \int _{N_0}\beta \vert _{N_0}=\int _{N_1}\beta \vert _{N_1} .$$
Furthermore, the mapping $j$ is not necessarily injective. Indeed
one has
\begin{equation}\label{kernel-homomorphism-vector-spaces-conservation-laws-function-bordism}
\ker(j)=\left\{\beta\in E^{0,n-1}_1  \left|\begin{array}{l}
   \int_N\beta\vert_N=0\\
   \hbox{\rm for all $(n-1)$-dimensional admissible}\\
   \hbox{\rm integral manifolds of $E_\infty$}\\
   \end{array}\right.\right\}.\end{equation}
So $\ker(j)$ can be
larger than the zero-class $[0]\in\mathcal{C} ons(E_k) $.\footnote{For example for the d'Alembert equation one can see that for any conservation law $\omega$ one has
$<\omega,N>=0$, where $N$ is any
admissible $1$-dimensional compact integral manifold
of $d'A)$, but $\omega\not\in[0]\in E_1^{0,n-1}$.} \end{proof}

\begin{remark}  Note that one has the following short exact sequence:
$$\xymatrix{0\ar[r]&\mathbb{R}^{\widehat{\Omega}_{n-1,s}^{E_\infty}}\ar[r]^{i_*}&\mathbb{R}^{\Omega_{n-1}^{E_\infty}}}$$
where $ i_*$ is the mapping $ i_*:\phi\mapsto\phi\circ i$,
$\forall \phi\in \mathbb{R}^{\widehat{\Omega}_{n-1,s}^{E_\infty}}$, and
$ i\equiv i_{n-1}$ is the canonical mapping defined in the
following commutative diagram:
$$\xymatrix@C=40pt{\mathbb{R}\ar@{=}[r]&\mathbb{R}&\\
        \Omega_{n-1}^{E_\infty}\ar[u]^{i_*(\phi)}\ar[r]_{i_{n-1}}&\hat\Omega_{n-1,s}^{E_\infty}\ar[u]\ar[r]^(0.5){\phi}&0}$$
As $ i_{n-1}$ is surjective it follows that $ i_*$ is injective. So any function on
$ \Omega_{n-1,s}^{E_\infty}$ can be identified with a function on
$ \Omega_{n-1}^{E_\infty}$. In particular, if $ \hat\Omega^{E_\infty}_{n-1,s}\cong
\Omega^{E_\infty}_{n-1,s}$ then any function on $ \Omega^{E_\infty}_{n-1,s}$ can be identified
with a function on $ \Omega^{E_\infty}_{n-1}$.\end{remark}

By means of Theorem \ref{conservation-laws-and-integral-characteristic-forms}  we are able to represent
$E^{0,n-1}_1$ by means of a Hopf algebra. In the following
$\Omega$ can be considered indifferently one of previously
considered "bordism groups".

\begin{lemma} Denote by $\mathbb{K}\Omega$ the free $\mathbb{K}$-module generated
by $\Omega $. Then, $\mathbb{K}\Omega$ has a natural structure of
$\mathbb{K}$-bialgebra ({\em group $\mathbb{K}$-bialgebra}\index{key@$\mathbb{K}$-bialgebra ! group}. (Here
$\mathbb{K}=\mathbb{R}$).\end{lemma}

\begin{proof} In fact define on the free $\mathbb{K}$-module $\mathbb{K}\Omega$ the multiplication
$$(\sum_{x\in\Omega}a_xx)(\sum_{y\in\Omega}b_yy)=
          \sum_{z\in\Omega}(\sum_{xy=z}a_xb_y)z.$$
Then, $\mathbb{K}\Omega$ becomes a ring. The map  $\eta_{\mathbb{K}\Omega}:\mathbb{K}\to\mathbb{K}\Omega \hskip 2pt,\quad \eta_{\mathbb{K}\Omega}(\lambda)=a1 $, where $1$ is the unit in $\Omega ,$
makes  $\mathbb{K}\Omega$ an $\mathbb{K}$-algebra. Furthermore, if we define
$\mathbb{K}$-linear maps $\triangle:\mathbb{K}\Omega\to\mathbb{K}\Omega\bigotimes_{\mathbb{K}}\mathbb{K}\Omega$, $\triangle(s)=s\otimes s $ and $\epsilon:\mathbb{K}\Omega\to\mathbb{K}$, $\epsilon(s)=1$,
then $(\mathbb{K}\Omega,\triangle,\epsilon)$ becomes a $\mathbb{K}$-coalgebra. \end{proof}

\begin{lemma} The dual linear space $(\mathbb{K}\Omega)^*$ of $\mathbb{K}\Omega$ can be identified
with the set: $\mathbb{R}^{\Omega}\equiv Map(\Omega,\mathbb{K}) $, where the dual $\mathbb{K}$-algebra structure of $\mathbb{K}\Omega$ is given by
$$\left\{\begin{array}{l}
 (f+g)(s)=f(s)+g(s)\\
(fg)(s)=f(s)g(s)\\
(af)(s)=af(s) \hskip 2pt ,\hskip 2pt \forall
f,g\in Map(\Omega,\mathbb{K}) , s\in\Omega, a\in\mathbb{K}\\
         \end{array}\right\}.$$
\end{lemma}

\begin{lemma} If $\Omega$ is a finite group $A\equiv Map(\Omega,\mathbb{K})$ has a natural structure of
$\mathbb{K}$-bialgebra $(\mu,\eta,\triangle,\epsilon) $, with
$$\begin{array}{ll}
(a)\qquad&\mu:A\bigotimes_{\mathbb{K}}A\to A \hskip 2pt,\quad \mu(f\otimes g)=f.g ;\\
        (b)\qquad&\eta:\mathbb{K}\to A \hskip 2pt,\quad \eta(\lambda)(s)=\lambda  \hskip 2pt , \hskip 2pt\forall s\in\Omega ;\\
        (c)\qquad&\triangle:A\to A\bigotimes_{\mathbb{K}}A \hskip 2pt,\quad\triangle(f)(x,y)=f(xy) ;\\
        (d)\qquad&\epsilon:A\to \mathbb{K} \hskip 2pt,\quad \epsilon(f)=f(1) .\\
  \end{array}$$
\end{lemma}
\begin{lemma}
$\mathbb{K}\Omega$ has a natural structure of $\mathbb{K}$-Hopf algebra.
\end{lemma}

\begin{proof} Define the $\mathbb{K}$-linear map  $S:\mathbb{K}\Omega\to\mathbb{K}\Omega$, $S(x)=x^{-1}$, $\forall  x\in\Omega$. Then, $(1*S)(x)=xS(x)=xx^{-1}=1=\epsilon(x)1=\eta\circ\epsilon(x) , x\in\Omega $. Then, $S$ is the antipode of $\mathbb{K}\Omega$ so that $\mathbb{K}\Omega$ becomes a $\mathbb{K}$-Hopf algebra.\end{proof}

\begin{lemma} If $\Omega$ is a finite group $A\equiv Map(\Omega,\mathbb{K})$ has a natural structure of
$\mathbb{K}$-Hopf algebra. If $ \Omega$ is not a finite group $ Map(\Omega;\mathbb{K})$ has a
structure of Hopf algebra in extended sense, i.e., an extension of an Hopf $ \mathbb{K}$-algebra $
 K$ contained into $ Map(\Omega;\mathbb{K})$. More precisely, $ K=R_\mathbb{K}(\Omega)$ is the Hopf $
 \mathbb{K}$-algebra of all the representative functions on $ \Omega$. In fact, one has the
following short exact sequence:
$$\xymatrix{0\ar[r]& R_\mathbb{K}(\Omega)\ar[r]&Map(\Omega;\mathbb{K})\ar[r]& H\ar[r]& 0},$$
 where $ H$ is the quotient
algebra. (If $ \Omega$ is a finite group then $ H=0$.) Therefore,
$ <E^{1,0}>$ is, in general, an Hopf algebra in this extended
sense.\end{lemma}

\begin{proof} In fact one can define the antipode $
S(f)(x)=f(x^{-1})$, $ \forall  f\in A , x\in\Omega$. It satisfies
the equalities: $\mu(1\otimes S)\triangle=\mu(S\otimes
1)\triangle=\eta\circ\epsilon $.\end{proof}

\begin{theorem} The space of conservation laws $E^{0,n-1}_1$ of a PDE
identifies in a natural way a $\mathbb{K}$-Hopf algebra: $<E^{0,n-1}>\subset\mathbf{H}(E_\infty)\equiv Map(\Omega_{n-1}^{E_\infty},\mathbb{R})$. If $E^{0,n-1}=0\in Map(\Omega_{n-1}^{E_\infty},\mathbb{R})$, we put for definition $<E^{0,n-1}>\equiv\mathbf{H}(E_\infty)$. We call $<E^{0,n-1}>$ the {\em Hopf algebra}\index{Hopf algebra} of $E_k$.
\end{theorem}

\begin{proof} It is an immediate consequence of Theorem \ref{conservation-laws-and-integral-characteristic-forms}
and above lemmas, and taking into
account the following commutative diagram
$$\xymatrix{{\mathbb{R}^{\Omega_{n-1}^{E_\infty}}\times\mathbb{R}^{\Omega_{n-1}^{E_\infty}}}\ar[r]
                     &{\mathbb{R}^{\Omega_{n-1}^{E_\infty}}}\\
         {E^{0,n-1}\times E^{0,n-1}}\ar[u]\ar[r]&{<E^{0,n-1}>}\ar[u]}$$
where $<E^{0,n-1}>$   is the Hopf subalgebra of $\mathbb{R}^{\Omega_{n-1}^{E_\infty}}$ generated by
$E^{0,n-1} $. We denote by $f_\alpha$ the image of the conservation law $\alpha\in E^{0,n-1}_1$
into $\mathbb{R}^{\Omega_{n-1}^{E_\infty}} $. So in $<E^{0,n-1}>$ we have the following product:
$<E^{0,n-1}>\times <E^{0,n-1}>\to <E^{0,n-1}>$, $(f_\alpha,f_\beta)\mapsto f_\alpha.f_\beta$.
Furthermore, we can explicitly write
$$\begin{array}{ll}
   \overline\eta:&\mathbb{K}\to <E^{0,n-1}>\hskip 2pt,\quad \overline\eta(\lambda)(s)=\lambda,\\
       \triangle:&<E^{0,n-1}>\to <E^{0,n-1}>\bigotimes_{\mathbb{K}}<E^{0,n-1}>
       \hskip 2pt,\quad \triangle(f)(x,y)=f(xy),\\
       \epsilon:&<E^{0,n-1}>\to\mathbb{K}\hskip 2pt,\quad \epsilon(f)=f(1), \\
       S:&<E^{0,n-1}>\to <E^{0,n-1}> \hskip 2pt,\quad S(f)(x)=f(x^{-1}).\\
  \end{array}$$
So the proof is complete.\end{proof}

\begin{definition} We call {\em full $p$-Hopf algebra} of $E_k\subset
J^k_n(W)$ the following Hopf algebra: $\mathbf{H}_p(E_k)\equiv \mathbb{R}^{\Omega_p^{E_k}}$. In particular for $p=n-1$ we write $\mathbf{H}(E_k)\equiv\mathbf{H}_{n-1}(E_k)$ and we call it {\em full Hopf
algebra}\index{full Hopf algebra} of $E_k$.

If $<E^{0,n-1}>\cong \mathbf{H}(E_\infty)\equiv\mathbb{R}^{\Omega_{n-1}^{E_\infty}} $, we say that $E_k$ is {\em wholly Hopf-bording}\index{wholly Hopf-bording}. Furthermore, we say also that  $\mathbb{R}^{\Omega_{p}^{E_k}}\equiv\mathbf{H}_p(E_k)$ is the {\em space of the full $p$-conservation laws}\index{p@$p$-conservation laws ! full ! space of} of $E_k$.\end{definition}

\begin{theorem} If $\Omega_{n-1}^{E_\infty}$ is trivial then $E_k$ is
wholly Hopf-bording. Furthermore, in such a case $E^{0,n-1}=0$.
\end{theorem}

\begin{proof} In fact, in such a case one has
$\int_N\omega=\int_{\partial V}=\int_Vd\omega=0$, $\forall
[\omega]\in E^{0,n-1}_1$, $[N]\in\Omega^{E_\infty}_{n-1}$, and
$V=n$-dimensional admissible integral manifold contained into
$E_\infty$. Hence, for definition one has $<E^{0,n-1}>\cong\mathbf{H}(E_\infty)$.\end{proof}
\begin{table}[t]
\caption{Properties of the Cartan spectral sequence $\{E^{\bullet,\bullet}_r,d_r\}$ of PDE $E_k\subset J^k_n(W)$.}
\label{c-spectral-sequences-pde}
\begin{tabular}{|l|l|l|}
  \hline
  \hfil{\rm{\footnotesize $r$}}\hfil& \hfil{\rm{\footnotesize $E^{\bullet,\bullet}_r$}}\hfil& \hfil{\rm{\footnotesize Particular cases}}\hfil\\
\hline
\hfil{\rm{\footnotesize $[2,\infty]$}}\hfil& {\rm{\footnotesize $E_r^{p,q}=0,\hskip 3pt (p>0,q\not=n)$}}& \\
& {\rm{\footnotesize $E_r^{p,q}=0,\hskip 3pt (p=0,q>n)$}}& \\
\hline
\hfil{\rm{\footnotesize $0$}}\hfil& {\rm{\footnotesize $E^{0,q}_0=\overline{\Omega}^q(E\infty), \hskip 3pt d_0=\bar d$}}&\\
& {\rm{\footnotesize $E^{p,q}_0=0, \hskip 3pt q<0$}}&\\
& {\rm{\footnotesize $E^{p,q}_0=0, \hskip 3pt q>n$}}&\\
\hline
\hfil{\rm{\footnotesize 1}}\hfil& {\rm{\footnotesize $E^{0,q}_1=\overline{H}^q(E_\infty)$}}&{\rm{\footnotesize $E_1^{0,n-1}=\overline{H}^{n-1}(E_\infty)\cong \mathcal{C}ons(E_k)$}}\\
&&{\rm{\footnotesize $E_1^{0,n}=\overline{H}^{n}(E_\infty)\cong \mathcal{L}agr(E_k)\hskip 3pt(\spadesuit)$}}\\
&{\rm{\footnotesize $E_1^{p,q}=0,\hskip 3pt(p>0,q\not=0)$}}&\\
&{\rm{\footnotesize $E_1^{0,q}=E^{0,q}_\infty,\hskip 3pt q<n)$}}&\\
\hline
\hfil{\rm{\footnotesize 2}}\hfil& {\rm{\footnotesize $E_2^{p,n}=E^{p,n}_\infty,\hskip 3pt p>0$}}&\\
& {\rm{\footnotesize $E_2^{q-n,n}=H^q(E_\infty),\hskip 3pt q\ge n$}}&\\
\hline
\multicolumn{3}{|c|}{\rm{\footnotesize $E_k\equiv M=n$-dimensional manifold with $\mathbf{E}^k_n\equiv 0\subset TM\hskip 3pt(\clubsuit)$.}}\\
\hline
\hfil{\rm{\footnotesize 0}}\hfil& {\rm{\footnotesize $E^{p,0}_0=\Omega^p(M),\hskip 3pt d_0=0$}}& \\
& {\rm{\footnotesize $E^{p,q}_0=0,\hskip 3pt q>0$}}& \\
\hline
\hfil{\rm{\footnotesize 1}}\hfil& {\rm{\footnotesize $E^{p,q}_1=E^{p,q}_0$}}& \\
& {\rm{\footnotesize $d_1=d:E_1^{p,0}=\Omega^p(M)\to E_1^{p+1,0}=\Omega^{p+1}(M)=0$}}&\\
\hline
\hfil{\rm{\footnotesize 2}}\hfil& {\rm{\footnotesize $E^{p,0}_2=H^p(M),\hskip 3pt d_2=d:H^p(M)\to H^{p+1}(M)$}}& \\
& {\rm{\footnotesize $E^{p+1,0}_2=H^{p+1}(M)$}}& \\
\hline
\multicolumn{3}{l}{\rm{\footnotesize $(\spadesuit)\hskip 2pt\mathcal{L}agr(E_k)$= space of Lagrangian densities.}}\\
\multicolumn{3}{l}{\rm{\footnotesize $(\spadesuit)\hskip 2pt d_1\omega =0$ is the Euler-Lagrange equation of $[\omega]\in \mathcal{L}agr(E_k)$.}}\\
\multicolumn{3}{l}{\rm{\footnotesize $\{E^{\bullet,\bullet}_r,d_r\}$, converges to the de Rham cohomology algebra $H^\bullet(E_\infty)$.}}\\
\multicolumn{3}{l}{\rm{\footnotesize If $E_k$ is formally integrable $H^\bullet(E_\infty)\cong H^\bullet(E_k)$.}}\\
\multicolumn{3}{l}{\rm{\footnotesize If $E_k=J^k_n(W)$ then $H^q(E_k)\cong H^q(W)$.}}\\
\multicolumn{3}{l}{\rm{\footnotesize If $\overline{\Omega}^{s+1}(E_\infty)=0$ for $s\ge 0$, then $E^{p,q}_0=0$ for $q>s$.}}\\
\multicolumn{3}{l}{\rm{\footnotesize $H^q(E_\infty)=\overline{H}^q(E_\infty)$ if $q<n$.}}\\
\multicolumn{3}{l}{\rm{\footnotesize $(\clubsuit)\hskip 3pt C\Omega^\bullet=\Omega^\bullet(M), \hskip 3pt C^k\Omega^\bullet=\oplus_{i>k}\Omega^i(M)$.}}\\
\end{tabular}
\end{table}

\begin{table}[t]
\caption{Properties of the (co)homology integral Leray-Serre spectral sequences of PDE $E_k\subset J^k_n(W)$.}
\label{leray-serre-spectral-sequences-pde}
\begin{tabular}{|l|l|l|l|}
  \hline
  \hfil{\rm{\footnotesize $r$}}\hfil& \hfil{\rm{\footnotesize $\{E_{\bullet,\bullet}^r,d^r\}$ and $\{E^{\bullet,\bullet}_r,d_r\}$}}\hfil& \hfil{\rm{\footnotesize Convergence space}}\hfil& \hfil{\rm{\footnotesize Particular cases}}\hfil\\
\hline
\hfil{\rm{\footnotesize $2$}}\hfil& {\rm{\footnotesize $E^2_{p,q}\cong H_p(E_k;H_q(F_k;G))$}}& {\rm{\footnotesize $H_\bullet(I(E_k);G)$}}&\\
\hline
\hfil{\rm{\footnotesize $2$}}\hfil& {\rm{\footnotesize $E_2^{p,q}\cong H^p(E_k;H^q(F_k;R))$}}&{\rm{\footnotesize $H^\bullet(I(E_k);R)$}}& {\rm{\footnotesize $E_2^{p,q}\cong H^p(I(E_k);\mathbb{K})\otimes H^q(F_k;\mathbb{K})$}}\\
&&& {\rm{\footnotesize if $H^q(F_k;R)$ simple and $R=\mathbb{K}$=field}}\\
\hline
\multicolumn{4}{l}{\rm{\footnotesize $G$=abelian group; $R$=commutative ring with unit.}}\\
\multicolumn{4}{l}{\rm{\footnotesize $F_k$=fibre of the fibre bundle $I(E_k)\to E_k$.}}\\
\multicolumn{4}{l}{\rm{\footnotesize Similar formulas hold for oriented fibre bundle $I^+(E_k)\to E_k$ with oriented fibre $F^+_k$.}}\\
\end{tabular}
\end{table}

\begin{theorem}[Cartan spectral sequences and integral Leray-Serre spectral sequences of PDE's]\label{c-spectral sequences-and-leray-serre-spectral-sequences-of-pdes}
Let $E_k\subset J^k_n(W)$ be a PDE on the fiber bundle $\pi:W\to M$, $\dim W=n+m$, $\dim M=n$. Let $I(E_k)\to E_k$, (resp. $I^+(E_k)\to E_k$) be the Grassmannian bundle of integral planes (resp. oriented integral planes), of $E_k$.
with fibre $F_k$ (resp. $F^+_k$). Then we can identifies two (co)homology spectral sequences: {\em(a) Cartan spectral sequences} and {\em(b) integral Leray-Serre spectral sequences}, such that if $E_k$ is formally integrable and the following conditions occur:

{\em (i)} $I(E_k)$ is path-connected;

{\em (ii)} $H^q(F_k;\mathbb{R})$ is simple;

{\em (iii)} $F_k$ is totally non-homologous to $0$ in $I(E_k)$ with respect to $\mathbb{R}$;

{\em (iv)} $H^q(F_k;\mathbb{R})=0$ if $q>0$, or $H^q(F_k;\mathbb{R})=\mathbb{R}$ if $q=0$;

then the above cohomology spectral sequences of $E_k$ converge to the same space $H^\bullet(E_\infty;\mathbb{R})\cong H^\bullet(E_k;\mathbb{R})$. (A similar theorem holds for oriented case.)

All above spectral sequences are natural with respect to fibred preserving maps and fibrations.
\end{theorem}

\begin{proof}
A detailed proof of this theorem is given in \cite{PRA00}. Here let us recall only that the Cartan spectral sequence of a PDE $E_k$ is induced by the filtration (\ref{cartan-spectral-sequence-pde-filtration}) of $\Omega^\bullet(E_\infty)\equiv C^\infty(\oplus_{i\ge 0}\Lambda^0_i(E_\infty))$.
\begin{equation}\label{cartan-spectral-sequence-pde-filtration}
    \Omega^\bullet(E_\infty)=C^0\Omega^\bullet(E_\infty)\supset C^1\Omega^\bullet(E_\infty)\supset\cdots\supset C^k\Omega^\bullet(E_\infty)\supset\cdots
\end{equation}
where $C^k\Omega^\bullet(E_\infty)$ is the $k$-th degree of differential ideal $C\Omega^\bullet(E_\infty)=\bigoplus_{i\ge 0}C\Omega^i(E_\infty)$, with $C\Omega^i(E_\infty)\equiv C\Omega^1(E_\infty)\wedge\Omega^{i-1}(E_\infty)$ and $C\Omega^1(E_\infty)$ is the annulator of the Cartan distribution $\mathbf{E}^k_n\subset TE_\infty$. Furthermore, $\overline{\Omega}^i(E_\infty)\equiv\Omega^i(E_\infty)/C\Omega^i(E_\infty)$ and $\overline{d}_i:\overline{\Omega}^i(E_\infty)\to\overline{\Omega}^{i+1}(E_\infty)$ is induced by the exterior differential $d$ on $\Omega^i(E_\infty)$. Put $\overline{H}^i(E_\infty)\equiv\ker(\overline{d}_i)/\IM(\overline{d}_{i-1})$. The Cartan spectral sequence of $E_k$ converges to the de Rham cohomology algebra $H^\bullet(E_\infty)$. If $E_k$ is formally integrable one has $H^\bullet(E_\infty)\cong H^\bullet(E_k)$. In Tab. \ref{c-spectral-sequences-pde} are resumed some remarkable properties of the Cartan spectral sequences.

The (co)homology integral Leray-Serre spectral sequences of a PDE, are obtained as (co)homology Leray-Serre spectral sequences of the fiber bundles $I(E_k)\to E_k$ (resp. $I^+(E_k)\to E_k$). In Tab. \ref{leray-serre-spectral-sequences-pde} are resumed some remarkable properties of the Cartan spectral sequences.
\end{proof}

\section{SPECTRA IN EXOTIC PDE's}\label{spectre-in-exotic-pdes-section}

In this last section results of the previous two sections are utilized to obtain some new results in the geometric theory of PDE's. More precisely, it is introduced the new concept of {\em exotic PDE}, i.e., a PDE having Cauchy integral manifolds with exotic differential structures. Integral (co)bordism groups for such exotic PDE's are characterized by suitable spectra, and local and global existence theorems are obtained. The main result is Theorem \ref{integral-h-cobordism-in-Ricci-flow-pdes} characterizing global solutions in Ricci flow equation, that extends some previous results in \cite{PRA17} also to dimension $n=4$. In fact, we have proved that the smooth Poincar\'e conjecture is true. As a by-product we get also that the smooth $4$-dimensional h-cobordism theorem holds. This extends to the category of smooth manifolds, the well-known result by Freedman obtained in the category of topological manifolds.

\begin{definition}[Exotic PDE's]\label{exotic-pdes}
Let $E_k\subset J^k_n(W)$ be a $k$-order PDE on the fiber bundle $\pi:W\to M$, $\dim W=m+n$, $\dim M=n$. We say that $E_k$ is an {\em exotic PDE} if it admits Cauchy integral manifolds $N\subset E_k$, $\dim N=n-1$, such that one of the following two conditions is verified.

{\em(i)} $\Sigma^{n-2}\equiv\partial N$ is an exotic sphere of dimension $(n-2)$, i.e. $\Sigma^{n-2}$ is homeomorphic to $S^{n-2}$, ($\Sigma^{n-2}\thickapprox S^{n-2}$) but not diffeomorphic to $S^{n-2}$, ($\Sigma^{n-2}\not\cong S^{n-2}$).

{\em(ii)} $\varnothing=\partial N$ and $N\thickapprox S^{n-1}$, but $N\not\cong S^{n-1}$.\footnote{The following Refs. \cite{BOTT-MILNOR, CERF, CHEEGER, FERRY-RANICKI-ROSENBERG, HIRSCH1, HIRSCH2, KAWA, KERVAIRE-MILNOR, KIRBY-SIEBENMAN, KLING, MAZUR, MILNOR1, MILNOR2, MILNOR3, MILNOR3, MOISE1, MOISE2, MUNKRES, NASH, OHKAWA, RADO, SCHOEN-YAU, SMALE1, SMALE2, SMALE3, SULLIVAN, TOGNOLI, TUSCH, WHITEHEAD, WEYL} are important background for differential structures and exotic spheres.}
\end{definition}

\begin{example}
The Ricci flow equation is an exotic PDE for $n$-dimensional Riemannian manifolds of dimension $n\ge 7$. (See \cite{PRA4, PRA14, PRA16, PRA17}.) (For complementary informations on the Ricci flow equation see also the following Refs. \cite{CHOW-CHU-GLIE-GUE-ISE-IVEY-KNOP-LU-LUO-NI-1, CHOW-CHU-GLIE-GUE-ISE-IVEY-KNOP-LU-LUO-NI-2, HAMIL1, HAMIL2, HAMIL3, HAMIL4, HAMIL5, PER1, PER2}.)
\end{example}

\begin{example}
The Navier-Stokes equation can be encoded on the affine fiber bundle $\pi:W\equiv M\times\mathbf{I}\times\mathbb{R}^2\to M$, $(x^\alpha,\dot x^i,p,\theta)_{0\le\alpha\le 3,1\le i\le 3}\mapsto(x^\alpha)$. (See \cite{PRA1, PRA6, PRA7, PRA9, PRA11, PRA14}.) Therefore, Cauchy manifolds are $3$-dimensional space-like manifolds. For such dimension do not exist exotic spheres. Therefore, the Navier-Stokes equation cannot be an exotic PDE. Similar considerations hold for PDE's of the classical continuum mechanics.
\end{example}

\begin{example}
The $n$-d'Alembert equation on $\mathbb{R}^n$ can be an exotic PDE for $n$-dimensional Riemannian manifolds of dimension $n\ge 7$. (See \cite{PRA18}.)
\end{example}

\begin{example}
The Einstein equation can be an exotic PDE for $n$-dimensional space-times of dimension $n\ge 7$. Similar considerations hold for generalized Einstein equations like Einstein-Maxwell equation, Einstein-Yang-Mills equation and etc.
\end{example}

\begin{theorem}[Integral bordism groups in exotic PDE's and stability]\label{bordism-groups-in-exotic-pdes-stability}
Let $E_k\subset J^k_n(W)$ be an exotic formally integrable and completely integrable PDE on the fiber bundle $\pi:W\to M$, $\dim W=n+m$, $\dim M=n$, such that $\dim E_k\ge 2n+1$, $\dim g_k\not=0$ and $\dim g_{k+1}\not=0$.
Then there exists a spectrum $\Xi_s$ such that for the singular integral $p$-(co)bordism groups can be expressed by means of suitable homotopy groups as reported in {\em(\ref{singular-integral-p-co-bordism-groups-exotic-pdes})}.
\begin{equation}\label{singular-integral-p-co-bordism-groups-exotic-pdes}
    \left\{
    \begin{array}{l}
      \Omega_{p,s}^{E_k}=\mathop{\lim}\limits_{r\to\infty}\pi_{p+r}(E_k^+\wedge\Xi_r)\\
      \Omega^{p,s}_{E_k}=\mathop{\lim}\limits_{r\to\infty}[S^rE_k^+,\Xi_{p+r}]\\
      \end{array}
    \right\}_{p\in\{0,1,\cdots,n-1\}}
\end{equation}

Furthermore, the singular integral bordism group for admissible smooth closed compact Cauchy manifolds, $N\subset E_k$, is given in {\em(\ref{singular-integral-bordism-group-n-1})}.
\begin{equation}\label{singular-integral-bordism-group-n-1}
    \Omega_{n-1,s}^{E_k}\cong\bigoplus_{p+q=n}H_p(W;\mathbb{Z}_2)\otimes_{\mathbb{Z}_2}\Omega_q.
\end{equation}
In the {\em homotopy equivalence full admissibility hypothesis}, i.e., by considering admissible only $(n-1)$-dimensional smooth Cauchy integral manifolds identified with homotopy spheres,
one has $\Omega^{E_k}_{n-1,s}=0$, when the space of conservation laws is not zero. So that $E_k$ becomes an extended $0$-crystal PDE. Then, there exists a global singular attractor, in the sense that all Cauchy manifolds, identified with homotopy $(n-1)$-spheres, bound singular manifolds.

Furthermore, if in $W$ we can embed all the homotopy $(n-1)$-spheres, (i.e. $\dim W\ge 2n+1$, and all such manifolds identify admissible smooth $(n-1)$-dimensional Cauchy manifolds of $E_k$), then two of such Cauchy manifolds bound a smooth solution iff they are diffeomorphic and one has the following bijective mapping: $\Omega^{E_k}_{n-1}\leftrightarrow\Theta_{n-1}$.\footnote{For the definition of the groups $\Theta_n$, see \cite{PRA17}.}

Moreover, if in $W$ we cannot embed all homotopy $(n-1)$-spheres, but only $S^{n-1}$, then in the {\em sphere full admissible hypothesis}, i.e., by considering admissible only $(n-1)$-dimensional smooth Cauchy integral manifolds identified with $S^{n-1}$, then $\Omega^{E_k}_{n-1}=0$. Therefore $E_k$ becomes a $0$-crystal PDE and there exists a global smooth attractor, in the sense that two of such smooth Cauchy manifolds, identified with $S^{n-1}$ bound smooth manifolds. Instead, two Cauchy manifolds identified with exotic $(n-1)$-spheres bound by means of singular solutions only.

All above smooth or singular solutions are unstable. Smooth solutions can be stabilized.
\end{theorem}

\begin{proof}
The relations (\ref{singular-integral-p-co-bordism-groups-exotic-pdes}) are direct applications of Theorem \ref{integral-spectrum-pdes}. Furthermore, under the hypotheses of theorem we can apply Theorem \ref{formal-integrability-integral-bordism-groups}. Thus we get directly (\ref{singular-integral-bordism-group-n-1}). Furthermore, under the homotopy equivalence full admissibility hypothesis, all admissible smooth $(n-1)$-dimensional Cauchy manifolds of $E_k$, are identified with all possible homotopy $(n-1)$-spheres. Moreover, all such Cauchy manifolds have same integral characteristic numbers. (The proof is similar to the one given for Ricci flow PDE's in \cite{PRA14, PRA17, AG-PRA1}.) Therefore, all such Cauchy manifolds belong to the same singular integral bordism class, hence $\Omega^{E_k}_{n-1,s}=0$. Thus in such a case $E_k$ becomes an extended $0$-crystal PDE. When $W\ge 2n+1$, all homotopy $(n-1)$-spheres can  be embedded in $W$ and so that in each smooth integral bordism class of $\Omega^{E_k}_{n-1}$ are contained homotopy $(n-1)$-spheres. Then, since two homotopy $(n-1)$-spheres bound a smooth solution of $E_k$ iff they are diffeomorphic, it follows that one has the bijection (but not isomorphism) $\Omega^{E_k}_{n-1}\cong\Theta_{n-1}$. In the sphere full admissibility hypothesis we get $\Omega^{E_k}_{n-1}=0$ and $E_k$ becomes a $0$-crystal PDE.

Let us assume now, that in $W$ we can embed only $S^{n-1}$ and not all exotic $(n-1)$-spheres. Then smooth Cauchy $(n-1)$-manifolds identified with exotic $(n-1)$-spheres are necessarily integral manifolds with Thom-Boardman singularities, with respect to the canonical projection $\pi_{k,0}:E_k\to W$. So solutions passing through such Cauchy manifolds are necessarily singular solutions. In such a case smooth solutions bord Cauchy manifolds identified with $S^{n-1}$, and two diffeomorphic Cauchy manifolds identified with two exotic $(n-1)$-spheres belonging to the same class in $\Theta_{n-1}$ cannot bound smooth solutions. Finally, if also $S^{n-1}$ cannot be embedded in $W$, then there are not smooth solutions bording smooth Cauchy $(n-1)$-manifolds in $E_k$, identified with $S^{n-1}$ or $\Sigma^{n-1}$ (i.e., exotic $(n-1)$-sphere). In other words $\Omega^{E_k}_{n-1}$ is not defined in such a case !
\end{proof}

We are ready to state the main result of this paper that completes Theorem 4.59 in \cite{PRA17}.

\begin{theorem}[Integral h-cobordism in Ricci flow PDE's]\label{integral-h-cobordism-in-Ricci-flow-pdes}
The Ricci flow equation for $n$-dimensional Riemannian manifolds, admits that starting from a $n$-dimensional sphere $S^n$, we can dynamically arrive, into a finite time, to any $n$-dimensional homotopy sphere $M$. When this is realized with a smooth solution, i.e., solution with characteristic flow without singular points, then $S^n\cong M$. The other homotopy spheres $\Sigma^n$, that are homeomorphic to $S^n$ only, are reached by means of singular solutions.

In particular, for $1\le n\le 6$,  one has also that any smooth $n$-dimensional homotopy sphere $M$ is diffeomorphic to $S^n$, $M\cong S^n$. In particular, the case $n=4$, is related to the proof that the smooth Poincar\'e  conjecture is true.
\end{theorem}
\begin{figure}[h]
\centering
\centerline{\includegraphics[height=4cm]{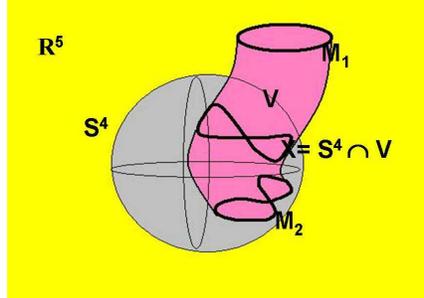}}
\caption{Embeddings of $3$-dimensional homotopy spheres in $S^4$ and smooth $3$-dimensional h-cobordism.}
\label{inside-outside-embeddings-3-homotopy-sphere-4-sphere}
\end{figure}

\begin{proof}
Let us consider some lemmas and definitions.
\begin{lemma}[Smooth $3$-dimensional h-cobordism theorem]\label{smooth-3-dimensional-h-cobordism-theorem}
Let $N_1$ and $N_2$ be $3$-dimensional smooth homotopy spheres. Then there exists a trivial smooth h-cobordism $V$, i.e., a $4$-dimensional manifold $V$, such that the following conditions are satisfied:

{\em(i)} $\partial V=N_1\sqcup N_2$;

{\em(ii)} The inclusions $N_i\hookrightarrow V$, $i=1,2$, are homotopy equivalences;

{\em(iii)} $V$ is diffeomorphic to the smooth manifold $N_1\times I$.
\end{lemma}

\begin{proof}
This lemma is a direct consequence of the Poincar\'e conjecture as proved by A. Pr\'astaro in \cite{PRA14}. In fact, there it is proved that $N_i$, $i=1,2$, can be identified with two smooth Cauchy manifolds of the Ricci flow equation $(RF)$, bording singular solutions $V$, that are h-cobordisms, but also smooth solutions $V'$ that are necessarily trivial h-bordisms. (See also \cite{PRA17}.) As a by-product it follows that $V'\cong N_i\times I$.
\end{proof}

\begin{lemma}[Smooth $4$-dimensional generalized Jordan-Brouwer-Sch\"onflies problem]\label{smooth-4-dimensional-schoenflies-problem}
A smoothly (piecewise-linearly) embedded $3$-sphere in the $4$-sphere $S^4$ bounds a smooth (piecewise-linear) $4$-disk $D^4\subset S^4$: any embedded $3$-sphere in $S^4$ separates it into two components having the same homology groups of a point.\footnote{It is well known that the Sch\"onflies problem is related to extensions of the Jordan-Brouwer theorem. (See, e.g., \cite{MAZUR}.) Let us emphasize that the lemma does not necessitate to work in the category of topological spaces. In fact, it is well known that topological embeddings $f:S^2\to S^3$ do not necessarily have simply connected the two separate components of $S^3\setminus f(S^2)$. In fact this is just the case of the {\em Alexander horned sphere} $\Sigma^2\subset S^3$ \cite{ALEXANDER}.}
\end{lemma}
\begin{proof}
All the reduced homology groups of the complements $Y\equiv S^n\setminus f(D^n)\subset S^n$ of smooth embeddings $f:D^k\to S^n$ are trivial ones: $\widetilde{H}_p(Y;\mathbb{Z})=0$, $p\ge 0$.\footnote{The {\em reduced homology groups} $\widetilde{H}_p(X)$, of non-empty space $X$, are the homology groups of the augmented chain complex: $\xymatrix{\cdots C_2(X)\ar[r]^{\partial_2}&C_1(X)\ar[r]^{\partial_1}&C_0(X)\ar[r]^{\epsilon}&\mathbb{Z}\ar[r]&0}$, where $\epsilon$ can be considered generated by the chain $[\varnothing]\mapsto X$, sending the simplex with no-vertices ({\em empty simplex}) to $X$, i.e., $\epsilon(\sum_in_i\sigma)=\sum_in_i$. Since $\epsilon\partial_i=0$, $\epsilon$ induces a map $H_0(X)\to\mathbb{Z}$ with kernel $\widetilde{H}_0(X)$, so one has $H_0(X)\cong \widetilde{H}_0(X)\bigoplus\mathbb{Z}$, and $H_p(X)\cong \widetilde{H}_p(X)$, $\forall p>0$. Therefore, we get $\widetilde{H}_0(pt)\cong 0$. Furthermore, one has $\widetilde{H}_p(X,A)\cong H_p(X,A)$, for any couple $(X,A)$, $X\supset A\not=\varnothing$, and $\widetilde{H}_p(X)=H_p(X,x_0)$.}
In (\ref{homology-n-sphere-setminus-embedded-k-sphere}) are given the reduced homology groups of $S^n\setminus f(S^k)$, for any smooth embedding $f:S^k\to S^n$, $k<n$.
\begin{equation}\label{homology-n-sphere-setminus-embedded-k-sphere}
  \scalebox{0.8}{$\widetilde{H}_p(S^n\setminus f(S^k);\mathbb{Z})=\left\{
  \begin{array}{ll}
    \mathbb{Z}&p=n-k-1\\
    0&\hbox{\rm otherwise.}\\
  \end{array}
  \right\}\Rightarrow\hskip 3pt \widetilde{H}_p(S^n\setminus f(S^{n-1});\mathbb{Z})=\left\{\begin{array}{ll}
    \mathbb{Z}&p=0\\
    0&\hbox{\rm otherwise.}\\
  \end{array}
  \right.$}
\end{equation}
In particular for $n=4$, we get $\widetilde{H}_0(S^4\setminus f(S^{3});\mathbb{Z})=\mathbb{Z}$, i.e., $H_0(S^4\setminus f(S^{3});\mathbb{Z})=\mathbb{Z}\bigoplus\mathbb{Z}$. Since $\widetilde{H}_\bullet$ preserves coproducts, i.e., takes arbitrary disjoint unions to direct sums, we get that $Z\equiv S^4\setminus f(S^{3}$ is made by two contractible, separate components of $S^4$. This agrees with Lemma \ref{smooth-3-dimensional-h-cobordism-theorem}. In fact, let consider a fixed $S^4\subset\mathbb{R}^5$, identified with the equation $\sum_{1\le i\le 5}(x^i)^2-1=0$, as representative of the framed cobordism class of $4$-dimensional spheres, since $\Omega^{fr}_4\cong\pi_4^S(S^0)=0$. Let $M_1\subset\mathbb{R}^5$ be a $3$-dimensional smooth homotopy sphere outside $S^4$, and let $f:M_1\to S^4$ be any embedding. Set $f(M_1)\equiv X\subset S^4$. Let $M_2\subset\mathbb{R}^5$ be another $3$-dimensional smooth homotopy sphere inside $S^4$. (See Fig. \ref{inside-outside-embeddings-3-homotopy-sphere-4-sphere}.) Since $\Omega_3=0$, we can find a smooth $4$-dimensional manifold $V$ such that $\partial V=M_1\sqcup M_2$ and such that $V\bigcap S^4=X$. From Lemma \ref{smooth-3-dimensional-h-cobordism-theorem} we can assume that $V$ is a trivial h-cobordism. This implies that $X$ cannot be knotted. As a by-product it follows that $X\cong\partial D^4\subset S^4$. The same result can be obtained by considering framed cobordism classes in $\Omega^{fr}_3\cong\pi_3^S(S^0)\cong\mathbb{Z}_{28}$, i.e., by considering intersections of $S^4\subset \mathbb{R}^5$ with $4$-dimensional planes $\mathbb{R}^4$, where embed representatives of $3$-dimensional framed homotopy spheres.
\end{proof}

\begin{lemma}[The smooth Poincar\'e conjecture]\label{the-smooth-poincare-conjecture}
The smooth ($4$-dimensional) Poincar\'e conjecture is true. In other words all compact, closed, $4$-dimensional smooth manifolds, $\Sigma^4$, homotopy equivalent to $S^4$, are diffeomorphic (other than homeomorphic) to $S^4$: $\Sigma^4\cong S^4$.
\end{lemma}

\begin{proof}
Existence of exotic $4$-spheres is related to the existence of exotic $4$-disks. Thus let us recall the definition of exotic $4$-disks.
\begin{definition}
An {\em exotic $4$-disk} (or {\em Mazur manifold}), is a contractible, compact, smooth $4$-dimensional manifold $\widetilde{D}^4$ which is homeomorphic, but not diffeomorphic, to the standard $4$-disk $D^4$.
\end{definition}
The boundary of an exotic $4$-disk is necessarily an homology $3$-sphere.
 So it is important to study the structure of such homology $3$-spheres. With this respect we shall introduce some further definitions and lemmas.
\begin{definition}
A periodic diffeomorphism $f$ of an orientable $3$-manifold $M$ has {\em trivial quotient} if the corresponding space of orbits, say $M_f$, is homeomorphic to $S^3$: $M_f\thickapprox S^3$.
\end{definition}
\begin{example}
The standard $S^3$ admits a periodic diffeomorphism $f$ of any order and with trivial quotient: $S^3_f\thickapprox S^3$.
\end{example}
It is well known from a theorem by Kervaire that for $n\ge 4$ the h-cobordism classes of homotopy $n$-spheres are isomorphic to the ones of the h-cobordism classes of homology $n$-spheres. The situation is instead different in dimension $n=3$. This depends from the following lemmas.

\begin{lemma}[Properties of homology $3$-spheres]\label{properties-homology-3-spheres}

{\em 1)} The connected sum of two homology $3$-spheres is a homology $3$-sphere too.

{\em 2) (Prime decomposition of $3$-manifolds) \cite{MILNOR3, SCOTT}}
Every homology $3$-sphere can be written as a connected sum of prime homology $3$-spheres in an essentially unique way. (A homology $3$-sphere that cannot be written as a connected sum of two homology $3$-spheres is called {\em irreducible} (or {\em prime}).)
\end{lemma}
\begin{example}
If $p$, $q$ and $r$ are pairwise relatively prime positive integers, then the {\em Brieskorn $3$-sphere} $\Sigma(p,q,r)$ is the homology $3$-sphere identified by the equations {\em(\ref{brieskorn-3-spheres-equations})} in $\mathbb{C}^3\cong\mathbb{R}^6$.
\begin{equation}\label{brieskorn-3-spheres-equations}
   (\Sigma(p,q,r)):\hskip 3pt\left\{
  \begin{array}{l}
    x^2+y^2+z^2-1=0\\
    x^p+y^q+z^r=0\\
  \end{array}
   \right\}\subset\mathbb{R}^6
\end{equation}
Thus $\Sigma(p,q,r)$ is a framed $3$-dimensional manifold $\Sigma(p,q,r)\subset\mathbb{R}^{3+n}$, $n\ge 3$. Furthermore $\Sigma(p,q,r)$ is homeomorphic to $S^3$ if one of $p$, $q$ and $r$ is $1$. Furthermore, $\Sigma(2,3,5)$ is the {\em Poincar\'e (homology) sphere}, called also {\em Poincar\'e dodecahedral space}. Its fundamental group ({\em binary icosahedral group}) is $\mathbb{Z}_{120}$. $\Sigma(2,3,5)$ cannot bound a contractible manifold because the {\em Rochlin invariant} provides an obstruction, hence the Poincar\'e homology sphere cannot be the boundary of an exotic $4$-disk.\footnote{The {\em Rokhlin invariant} of a spin $3$-manifold $X$ is the signature of any spin $4$-manifold $V$, such that $\partial V=X$, is well defined mod 16. A spin structure exists on a manifold $M$, if its second Stiefel-Whitney class is trivial: $w_2(M)=0$. These structures are classified by $H^1(M;\mathbb{Z}_2)\cong H_1(M;\mathbb{Z}_2)$. Therefore, homology $3$-spheres have an unique spin structure, hence for them the Rokhlin invariant is well defined. In particular the Poincar\'e homology sphere bounds a spin $4$-manifold with intersection form $E_8$, so its Rokhlin invariant is $1$.}
\end{example}
\begin{example}
Let $a_1,\cdots,a_r$, be integers all at least $2$ such that any are coprime. Then the {\em Seifert fiber space} \footnote{These are $3$-dimensional manifolds endowed with a $S^1$-bundle structure over a $2$-dimensional orbifold. (See, e.g., \cite{SCOTT}.)} $\{b,(o_1,0),(a_1,b_1),\cdots,(a_r,b_r)\}$, with  $b+\frac{b_1}{a_1}+\cdots+\frac{b_r}{a_r}=\frac{1}{a_1\cdots a_r}$, over the sphere with exceptional fibers of degrees $a_1,\cdots,a_r$, is a homology $3$-sphere. If $r$ is at most $2$, one has the standard $S^3$. If the $a$'s are $2$, $3$, and $5$ one has the Poincar\'e sphere. If there are at least three $a$'s not $2$, $3$, $5$, then one has an acyclic homology $3$-sphere with infinite fundamental groups that has a Thurston geometry modeled on the universal cover of $SL_2(\mathbb{R})$.
\end{example}

\begin{lemma}[Structures of homology $3$-spheres \cite{BOILEAU-PAOLUZZI-ZIMMERMANN}]\label{structures-of-homology-3-spheres}
An homology $3$-sphere $M$ is $S^3$ iff it admits four periodic diffeomorphisms $f_i$, $i=1,2,3,4$, with parwise different odd prime orders whose space of orbits is $S^3$, i.e., $S^3\thickapprox M_{f_i}$, $i=1,2,3,4$.

An irreducible, homology $3$-sphere, different from $S^3$, is the cyclic branched cover of odd prime order of at most four knots in $S^3$.
\end{lemma}

Now, Lemma \ref{smooth-3-dimensional-h-cobordism-theorem} implies that cannot exist exotic $4$-disks obtained by smoothly embedding $S^3$ into $S^4$. This means that the boundary of an exotic $4$-disk must necessarily be an homology $3$-sphere.
On the other hand, from the above lemmas it follows also that smooth homology spheres that can bound a contractible manifold are the ones homeomorphic to $S^3$. It follows that the boundary of an exotic $4$-disk cannot be any homology $3$-sphere, but only $3$-dimensional manifolds, homeomorphic to $S^3$, hence, after the proof of the Poincar\'e conjecture, must necessarily be $\partial\widetilde{D}^4\cong S^3$. So if there exist exotic $4$-disks, their exoticity must be localized in their interiors.

Let us, now, consider the relation between exotic $\mathbb{R}^4$'s, say $\widetilde{\mathbb{R}}^4$, and $4$-spheres. For our purposes it is enough to consider the case where the exoticity of $\widetilde{\mathbb{R}}^4$ is localized in a open compact subset $K\subset\widetilde{\mathbb{R}}^4$. (See Refs.\cite{DONALDSON1, DONALDSON2, FREEDMAN, FREEDMAN-QUINN, FREEDMAN-TAYLOR, GOMPF, TAUBES1, TAUBES2}.) Let us compactify $\widetilde{\mathbb{R}}^4$ to a point: $(\widetilde{\mathbb{R}}^4)^+\equiv\widetilde{\mathbb{R}}^4\bigcup\{\infty\}\equiv\Sigma^4 $. Then, the relation between $\Sigma^4$ and $S^4$ is given by the exact commutative diagram in (\ref{relation-between-sigma-4-and-s4}).
This diagram shows that the smooth $4$-dimensional manifold $\Sigma^4$ is homeomorphic to $S^4$, and it is a fiber bundle over $S^4$. This should have the consequence that $\Sigma\not\cong S^4$, unless $a$ is the identity mapping. Now the question is the following: does $\Sigma^4$ bound a contractible manifold $V$? (From results in \cite{PRA17} we know that $\Sigma^4$ bounds singular solutions of the Ricci flow equation.) Since $S^4=\partial D^5$, and taking into account the h-cobordism theorem in dimension $n=4$, in the category of topological manifolds, (Freedman), we can assume that $V\thickapprox S^4\times I$, hence $V\thickapprox D^5$. Whether $V$ is not diffeomorphic to $D^5$, we should conclude that there exist exotic $5$-disks. On the other hand, it is well known that do not exist exotic $\mathbb{R}^n$, for $n\not=4$. Therefore if there exists an exotic $D^5$, say $\widetilde{D}^5$, its exoticity must be localized on its boundary $\partial \widetilde{D}^5$. On the other hand, $\widetilde{D}^5\bigcup_{\partial \widetilde{D}^5}\widetilde{D}^5$ should be an exotic $5$-sphere. This is impossible, since do not exist exotic $5$-spheres. Therefore, must necessarily be $V\cong D^5$, hence $\Sigma^4=\partial V\cong \partial D^5=S^4$. This means that the process of compactification to a point of $\widetilde{\mathbb{R}}^4$ necessarily produces the collapse of $K$ to $\{\infty\}$. Really the closure $\overline{K}$ of the compact domain of exoticity in $\widetilde{\mathbb{R}}^4$ cannot have as boundary $\partial \overline{K}$ a simply connected $3$-dimensional manifold homotopy equivalent to $S^3$. In fact, in this case it should be $\partial \overline{K}\cong S^3$, hence $\overline{K}\cong D^4$, but this contradicts the assumption that in $K$ is localized the exoticity of $\widetilde{\mathbb{R}}^4$. On the other hand $\partial \overline{K}$ should coincide also with the boundary of the complement of $\overline{K}$ in $\Sigma^4$, that is necessarily an open $4$-disk $\mathop{D}\limits^{\circ}{}^4$, whether we assume that in the process of compactification the exoticity remains localized in $K$. This contradiction means that just in this process of compactification $K$ collapses to ${\infty}$ too: $ K\to{\infty}$. Thus we can conclude that the mapping $a$ in diagram (\ref{relation-between-sigma-4-and-s4}) is necessarily the identity, hence $\Sigma^4\cong S^4$.
\begin{equation}\label{relation-between-sigma-4-and-s4}
   \xymatrix{{\widetilde{\mathbb{R}}^4}\ar[d]_{\thickapprox}\ar@{^{(}->}[r]&
   {\widetilde{\mathbb{R}}^4}\bigcup\{\infty\}\equiv\Sigma^4\ar[d]_{\thickapprox}^{a}\ar[r]&
   \Sigma^4/K\bigcup\{\infty\}\ar@{=}[d]_{\wr}\\
   {\mathbb{R}^4}\ar@{^{(}->}[r]&\mathbb{R}^4\bigcup\{\infty\}\equiv S^4\ar@{=}[r]\ar[d]\ar@{=}[r]&S^4\\
   &0&}
\end{equation}

 Therefore, if do not exist exotic $4$-spheres, do not exist exotic  $4$-disks and vice versa. In fact, if there exists an exotic $4$-disk $\widetilde{D}^4$, we get that $\widetilde{D}^4\sharp \widetilde{D}^4\equiv\Sigma^4\thickapprox S^4$, where $\Sigma^4$ is an exotic $4$-sphere, (hence homeomorphic to $S^4$). Vice versa if one has an exotic $4$-sphere $\Sigma^4$, we can write $\Sigma^4\cong A\bigcup_XB$, where $X$ is a $3$-dimensional smooth manifold that separates $\Sigma^4$.\footnote{The existence of such a manifold $X$ can be proved following a strategy similar to the one to prove Lemma \ref{smooth-4-dimensional-schoenflies-problem}.} Then at least one of the submanifolds $A$ and $B$ should be an exotic $4$-disk. On the other hand, from the above results it follows that cannot exist exotic $4$-spheres.
 Therefore the smooth Poincar\'e conjecture is true.
\end{proof}

As direct consequences of above lemmas and by considering the proof of Theorem 4.59 in \cite{PRA17}, it follows that this theorem can be extended, now, to the case $n=4$ too, i.e., by using the same symbols defined in \cite{PRA17}, we can say that $\Theta_4=\Gamma_4=0$. In conclusion the proof of Theorem \ref{integral-h-cobordism-in-Ricci-flow-pdes} is down.
\end{proof}

\begin{cor}[Smooth $4$-dimensional h-cobordism theorem]\label{smooth-4-dimensional-h-cobordism-theorem}
The smooth h-cobordism theorem holds in dimension $4$.
\end{cor}

\end{document}